\documentclass[11pt,reqno]{amsart}

\usepackage{mathrsfs}
\usepackage{amsmath,amssymb,amsthm}

\usepackage{graphicx}
\usepackage{latexsym}
\usepackage{mathrsfs}
\usepackage{tikz}

\usepackage{comment}
\usepackage{enumerate}
\usepackage{listings}
\usepackage{mathtools}
\usepackage{booktabs}
\usepackage{multirow}
\usepackage{pdflscape}
\usepackage{lscape}
\usepackage{todonotes}

\usetikzlibrary{cd,arrows,matrix,backgrounds,positioning,calc,decorations.markings,decorations.pathmorphing,decorations.pathreplacing}
\tikzset{black/.style={circle,fill=black,inner sep=3pt,outer sep=3pt},
         white/.style={circle,fill=white,draw=black,inner sep=3pt,outer sep=3pt},
}
\input xy

\xyoption{all}

\numberwithin{equation}{section}
\newtheorem{theorem}{Theorem}[section]

\newtheorem{thm}[theorem]{Theorem}
\newtheorem{cor}[theorem]{Corollary}
\newtheorem{lemm}[theorem]{Lemma}
\newtheorem{prop}[theorem]{Proposition}
\newtheorem{question}[theorem]{Question}

\theoremstyle{definition}
\newtheorem{definition-theorem}[theorem]{Definition-Theorem}
\newtheorem{defi}[theorem]{Definition}
\newtheorem{remk}[theorem]{Remark}
\newtheorem{exam}[theorem]{Example}
\newtheorem{notation}[theorem]{Notation}

\newcommand{\tors}{\mathsf{tors}\hspace{.01in}}

\newcommand{\ftors}{\mathsf{f\mbox{-}tors}\hspace{.01in}}

\newcommand{\Hasse}{\mathsf{Hasse}}

\newcommand{\TT}{\operatorname{\mathcal T}\nolimits}

\newcommand{\VV}{\operatorname{\mathcal V}\nolimits}

\newcommand{\I}{\operatorname{\overline{ \ell}}\nolimits}

\usepackage[right]{lineno}

\newcommand{\add}{\mathsf{add}\hspace{.01in}}
\newcommand{\proj}{\mathsf{proj}\hspace{.01in}}

\newcommand{\id}{\operatorname{id}\nolimits}

\newcommand{\Hom}{\operatorname{Hom}\nolimits}

\newcommand{\Ext}{\operatorname{Ext}\nolimits}

\newcommand{\RHom}{\mathbf{R}\strut\kern-.2em\operatorname{Hom}\nolimits}

\def\dim{\mathop{\mathrm{dim}}\nolimits}

\def\Ker{\mathop{\mathrm{Ker}}\nolimits}

\def\Hom{\mathop{\mathrm{Hom}}\nolimits}
\def\End{\mathop{\mathrm{End}}\nolimits}
\def\Ext{\mathop{\mathrm{Ext}}\nolimits}

\def\RHom{\mathop{\mathbb R\mathrm{Hom}}\nolimits}
\DeclareMathOperator{\moduleCategory}{\mathsf{mod}} \renewcommand{\mod}{\moduleCategory}

\newcommand{\ind}{\mbox{\rm ind}\hspace{.01in}}

\newcommand{\itaurigid}{\mathsf{i\tau\mbox{-}rigid}\hspace{.01in}}
\newcommand{\taurigid}{\tau\mbox{-}\mathsf{rigid}\hspace{.01in}}
\newcommand{\stautilt}{\mathsf{s\tau\mbox{-}tilt}\hspace{.01in}}
\renewcommand{\ind}{\mathsf{ind}\hspace{.01in}}
\newcommand{\Dim}{\mathrm{Dim}\hspace{.01in}}
\newcommand{\area}{\mathrm{area}\hspace{.01in}}

\newcommand{\taurigidpair}{\mathsf{\tau\mbox{-}rigid\mbox{-}pair}\hspace{.01in}}

\newcommand{\link}{\mathrm{lk}\hspace{.01in}}

\newcommand{\Fac}{\mathsf{Fac}\hspace{.01in}}
\def\add{\mathsf{add}\hspace{.01in}}

\newcommand{\Cat}{\mathrm{Cat}\hspace{.01in}}
\newcommand{\Eul}{\mathrm{Eul}\hspace{.01in}}
\begin{document}
\title[Dimensions of $\tau$-tilting modules]
{Dimensions of $\tau$-tilting modules over path algebras and preprojective algebras of Dynkin type}
\author{Toshitaka Aoki}
\address{Toshitaka Aoki,
Graduate School of Human Development and Environment, Kobe University, 3-11 Tsurukabuto, Nada-ku, Kobe 657-8501, Japan}
\email{toshitaka.aoki@people.kobe-u.ac.jp}

\thanks{}
\author{Yuya Mizuno}
\address{Yuya Mizuno, 
Faculty of Liberal Arts, Sciences and Global Education, Osaka Metropolitan University, 1-1 Gakuen-cho, Naka-ku, Sakai, Osaka 599-8531, Japan}
\email{yuya.mizuno@omu.ac.jp}


\begin{abstract}
In this paper, we introduce a new generating function called $d$-polynomial for the dimensions of $\tau$-tilting modules over a given finite dimensional algebra. 
Firstly, we study basic properties of $d$-polynomials and show that it can be realized as a certain sum of the $f$-polynomials of the simplicial complexes arising from $\tau$-rigid pairs. 
Secondly, we give explicit formulas of $d$-polynomials for preprojective algebras and path algebras of Dynkin quivers by using a close relation with $W$-Eulerian polynomials and $W$-Narayana polynomials. 
Thirdly, we consider the ordinary and exponential generating functions defined from $d$-polynomials and give closed-form expressions in the case of preprojective algebras and path algebras of Dynkin type $\mathbb{A}$. 
\end{abstract}

\maketitle
\tableofcontents

\section{Introduction}
\subsection{Background} 
Since its introduction, $\tau$-tilting theory \cite{AIR14}, which can also be formulated as 2-term silting theory, has been extensively studied. 
Recently, the geometric viewpoint using simplicial complexes, fans and polytopes has lead to several new insights and questions. 
Among others, a representation-theoretic interpretation of $f$-vectors and $h$-vectors of $g$-simplicial complexes 
provide a new approach to the study of $\tau$-tilting theory (see \cite{Hille15,DIJ19,AMN20,AHIKM22,AHIKM23} for example). 
Let us briefly recall our basic  terminology. 

Let $A$ be a finite dimensional algebra over a field $k$. We denote by $\taurigidpair A$ the set of isomorphism classes of basic $\tau$-rigid pairs for $A$ (Definition \ref{tau tilting2}). 
We assume that $A$ is $g$-finite (i.e., $\taurigidpair A$ is finite) and $|A|=n$. 
Then, the $g$-simplicial complex $\Delta(A)$ of $A$ is defined to be a simplicial complex whose $j$-simplices are given by 
    $$
    \taurigidpair^{j+1} A := \{(M,P)\in \taurigidpair A \mid \text{$|M|+|P| = j+1$}\} \quad \text{($0 \leq j\leq n-1$).}
    $$  
Recall that the $f$-polynomial of $A$ is defined by 
\begin{equation*}
    f(A;t) := 
    \sum_{(M,P)\in \taurigidpair A} t^{n-|M|-|P|} = 
    \sum_{j=0}^n f_{j-1} t^{n-j},   
\end{equation*}
where $f_j = \#\taurigidpair^{j+1}A$ is the number of $j$-simplices of $\Delta(A)$. 
The $h$-polynomial is given by 
\begin{equation*}
    h(A;t) = f(A;t-1). 
\end{equation*}

Many interesting properties of the $f$-polynomials, together with the $h$-polynomials, have been widely studied and it has recently turned out that they also play an important role in the study of representation theory. 
In this paper, we introduce a new 
generating function called $d$-polynomial of $A$, which counts the dimension of $\tau$-tilting modules and can be regarded as a (weighted) variant of $f$-polynomials.

\begin{defi}[Subsection \ref{Enumeration}]
We define 
\begin{equation}
    d(A;t) := \sum_{(M,P)\in \taurigidpair A} \dim_k M \cdot t^{n-|M|-|P|} = \sum_{j=1}^n d_{j-1}t^{n-j}.
\end{equation}
That is, each coefficient $d_j\in \mathbb{Z}$ is the sum of the $k$-dimensions 
\begin{equation*}
    d_j = \sum_{(M,P) \in \taurigidpair^{j+1} A} \dim_{k} M \quad (0 \leq j \leq n-1).  
\end{equation*}
We call $d(A;t)$ the \emph{d-polynomial} of $A$. 
\end{defi}
In particular, the number $d_0$ (respectively, $d_{n-1}$) is the sum of the $k$-dimensions of all indecomposable $\tau$-rigid (respectively, basic support $\tau$-tilting) modules up to isomorphism. 
We remark that two finite dimensional algebras having isomorphic $g$-simplicial complexes (hence the same $f$- and $h$-polynomials) may have different $d$-polynomials. 
In fact, every local algebra $\Lambda$ has the same $g$-simplicial complex (up to isomorphism) with the $f$-polynomial $t+2$, whereas it has the $d$-polynomial $d(\Lambda;t) =\dim_k \Lambda$. 
Further examples can be found in Remark \ref{remk:EJR-resuction} and Example \ref{example:BTA}. 
As a result, our $d$-polynomial provides a new numerical invariant for $\tau$-tilting modules that differs from the $f$- and $h$-polynomials.

\subsection{Our results} 
The first aim of this paper is to study basic properties of $d$-polynomials. 
There is a close connection between $d$-polynomials and $f$-polynomials via the reduction method due to Jasso \cite{Jasso15}, recalled in Section \ref{Enumeration}. 
Using this technique, the $d$-polynomials 
can be calculated as follows.

\begin{prop}[Proposition \ref{prop:d and f}]\label{intro1}
We have 
    \begin{equation}
    d(A; t) = \sum_{M} \dim_k M \cdot  f(C_M;t), 
\end{equation}
where $M$ runs over all indecomposable $\tau$-rigid $A$-modules up to isomorphisms, 
and $C_M$ is the algebra obtained by the reduction at $M$. 
\end{prop}

This result asserts that our $d$-polynomial can be realized as a sum of the $f$-polynomials of some $g$-simplicial complexes. 
Thanks to this, the $d$-polynomials share several properties with $f$-polynomials and $h$-polynomials, such as palindromicity and unimodality (Corollary \ref{palindromic}).

The second aim of this paper is to give formulas of $d$-polynomials for preprojective algebras and path algebras of Dynkin quivers. 
Let us recall that $\tau$-tilting theory of these classes of algebras are closely related to important classes of combinatorial objects. 
More precisely, the set of support $\tau$-tilting modules of a preprojective algebra (respectively, path algebra) can be 
parameterized by the elements of the corresponding Weyl group (respectively, Cambrian combinatorics). We refer to \cite{IT09,Mizuno14,IRRT18,DIRRT23} for background. 
In particular, as pointed out in \cite{AHIKM22}, 
the $h$-vectors 
of preprojective algebras and path algebras can be realized by well-studied integer sequences, 
called $W$-Eulerian numbers and $W$-Narayana numbers respectively (see Sections \ref{section ppalg} and \ref{section path alg}). 
Hence, the $W$-Eulerian polynomials and $W$-Narayana polynomials, which are denoted by   
\begin{equation*}
    \Eul(W(Q);t) \quad \text{and} \quad \Cat(W(Q);t),  
\end{equation*}
associated with Dynkin quivers $Q$ coincide with the $h$-polynomials of preprojective algebras and path algebras of $Q$,  respectively.

Through the above connection together with Proposition  \ref{intro1}, we can describe their $d$-polynomials by using $W$-Eulerian polynomials and $W$-Narayana polynomials 
as follows. 
Now, we denote by $Q_{\I}$ the quiver obtained from $Q$ by deleting the vertex $\ell$. 

\begin{thm}[Theorem \ref{thm:orbit ppalg}] \label{intro thm:orbit ppalg}
Let $\Pi := \Pi(Q)$ be the preprojective algebra of a Dynkin quiver $Q$. Then, we have 
\begin{equation*}
d(\Pi;t) = \sum_{\ell\in Q_0} \Dim([e_{\ell}\Pi])\Eul(W(Q_{\I}); t+1), 
\end{equation*}
where $\Dim([e_{\ell}\Pi])$ is the sum of the $k$-dimensions of all indecomposable $\tau$-rigid submodules of the indecomposable projective module $e_{\ell} \Pi$ corresponding to the vertex $\ell$ up to isomorphism.
\end{thm}

\begin{thm}[Theorem \ref{thm:OOpathalg}] \label{intro thm:OOpathalg}
Let $H=kQ$ be a path algebra of a Dynkin quiver $Q$ and $\Pi = \Pi(Q)$ the preprojective algebra of $Q$. 
Then, we have 
\begin{equation*}
    d(H;t) = \sum_{\ell\in Q_0} \dim_k (e_{\ell}\Pi)  
    \Cat(W(Q_{\I});t+1). 
\end{equation*}
In particular, the polynomial $d(H;t)$ does not depend on an orientation of $Q$. 
\end{thm}

In more detail, we can give an explicit formula of the $d$-polynomials for each Dynkin type $\mathbb{A}$, $\mathbb{D}$ and $\mathbb{E}$ in Sections \ref{section ppalg} and \ref{section path alg}. 
In addition, these polynomials for low rank can be found in tables of each subsection.

The third aim of this paper is to use the method of generating functions, which is defined over the ring of formal power series, to capture a behavior of certain invariants in $\tau$-tilting theory for a given family of finite dimensional algebras. 
We define four kinds of generating functions as follows. 

\begin{defi}
Let $A_{\bullet} := (A_n)_{n\geq0}$ be a (possibly infinite) sequence of finite dimensional algebras such that $A_n$ is $g$-finite for each $n$.
We define generating functions
$$
\mathcal{H}(A_{\bullet};t,z):= \sum_{n\geq 0} h(A_n;t) z^{n} \quad \text{and} \quad 
\mathscr{H}(A_{\bullet};t,z):= \sum_{n\geq 0} h(A_n;t) \frac{z^{n}}{n!}, 
$$
which we call the \emph{ordinary and exponential generating functions of $h$-polynomials of $A_{\bullet}$}, respectively. 
Similarly, we call   
$$
\mathcal{D}(A_{\bullet};t,z):= \sum_{n\geq 0} d(A_n;t) z^{n} \quad \text{and} \quad 
\mathscr{D}(A_{\bullet};t,z):= \sum_{n\geq 0} d(A_n;t) \frac{z^{n}}{n!}
$$
the \emph{ordinary and exponential generating functions of $d$-polynomials of $A_{\bullet}$}, respectively. 
\end{defi}
Notice that the generating functions of $f$-polynomials can be obtained from those of $h$-polynomials as
\begin{equation*}
    \mathcal{H}(A_{\bullet};t-1,z) = \sum_{n\geq 0} f(A_n;t)z^n \quad \text{and} \quad 
    \mathscr{H}(A_{\bullet};t-1,z) = \sum_{n\geq 0} f(A_n;t)\frac{z^n}{n!}.
\end{equation*}

Now, let $\Pi(\mathbb{A}_{n})$ be a preprojective algebra of type $\mathbb{A}_n$ for $n\geq 0$, and let $\Pi(\mathbb{A}_{\bullet}):= (\Pi(\mathbb{A}_{n-1}))_{n\geq 0}$. 
Since the (classical) Eulerian polynomials coincide with $h$-polynomials of preprojective
algebras of type $\mathbb{A}$, 
Euler's famous result (see Proposition \ref{prop:Euler}) on the exponential generating function of Eulerian polynomials 
implies the following result. 
\begin{prop}[Proposition \ref{prop:Euler-ppaA}]
    \begin{equation*}    
    \mathscr{H}(\Pi(\mathbb{A}_{\bullet});t,z) = 
    \frac{t-1}{t-e^{z(t-1)}}.
    \end{equation*}
\end{prop}

It is natural to consider the analogue for $d$-polynomials. 
Using their description in terms of Eulerian polynomials (Theorem \ref{thm:dpoly ppaA}), we can also get a closed-form expression for the exponential generating function of $d$-polynomials of $\Pi(\mathbb{A}_{\bullet})$. 
Our result is the following. 

\begin{theorem}[Theorem \ref{thm:gen ppaA}]
In the above notations, we have  
\begin{equation*}
\mathscr{D}(\Pi(\mathbb{A}_{\bullet});t,z) = \frac{z^2t^4 e^{2zt}}{2(t+1 - e^{zt})^4}. 
\end{equation*}
\end{theorem}

We also give a similar result for path algebras of type $\mathbb{A}$ (Theorem \ref{thm:gen pathA}).  
Contrary to the case of preprojective algebras, it relates to the ordinary generating functions of (classical) Narayana polynomials. 
At the end of this paper (Section \ref{question}), we pose several questions related to further development.


\bigskip 
\begin{notation}
Throughout this paper, let $k$ be a field.
For a finite dimensional $k$-algebra $A$, we denote by $\mod A$ the category of finitely generated right $A$-modules, by $\proj A$ the category of finitely generated projective $A$-modules. 
For an $A$-module $X$, we denote by $|X|$ the number of pairwise non-isomorphic indecomposable direct summands of $X$. In addition, we denote by $\add X$ the smallest full subcategory which contains $X$ and is closed under taking direct sums, summands and isomorphisms. We denote by $\Fac X$ the full subcategory consisting of all modules which are quotients of finite direct sums of $X$. 
We say that a full subcategory $\TT$ of $\mod A$ is \emph{contravariantly finite} if for each $X\in\mod A$, there is a map $f\colon T\to X$ with $T\in\TT$ such that $(T',f) \colon\Hom_A(T',T)\to \Hom_A(T',X)$ is surjective for all $T'\in\TT$. 
Dually, a \emph{covariantly finite subcategory} is defined. We say that $\TT$ is \emph{functorially finite} if it is contravariantly and covariantly finite. 
\end{notation}

\section{Preliminaries}\label{sec:Preliminaries}
Let $A$ be a finite dimensional algebra over a field $k$. 

\subsection{$\tau$-tilting theory}\label{Support tau tilting modules} 
In this section, we recall basics of $\tau$-tilting theory \cite{AIR14}. We denote by $\tau$ the Auslander-Reiten translation of $\mod A$. 

\begin{defi}\label{tau tilting}
Let $M\in \mod A$. 
\begin{enumerate}[\rm (1)]
    \item We call $M$ \emph{$\tau$-rigid} if $\Hom_{A}(M,\tau M)=0$.
    \item We call $M$ \emph{$\tau$-tilting} 
    if $M$ is $\tau$-rigid and $|M|=|A|$. 
    \item We call $M$ \emph{support $\tau$-tilting} if there exists an idempotent $e$ of $A$ such that $M$ is a $\tau$-tilting $(A/AeA)$-module.
\end{enumerate}
\end{defi}

\begin{defi}\label{tau tilting2} Let $M\in\mod A$ and $P \in \proj A$. 
\begin{enumerate}[\rm (1)]
    \item We call a pair $(M,P)$ \emph{$\tau$-rigid} if $M$ is $\tau$-rigid and $\Hom_A(P,M)=0$.
\item We call a pair $(M,P)$ \emph{support $\tau$-tilting} if it is a $\tau$-rigid pair and $|M|+|P|=|A|$. 
    \end{enumerate}
\end{defi}

For a pair $(M,P)$, we say that it is \emph{basic} if $M$ and $P$ are basic. In addition, we say that $(M,P)$ is a \emph{direct summand} of $(N,Q)$ if $M$ (respectively, $P$) is a direct summand of $N$ (respectively, $Q$). 
Note that a basic support $\tau$-tilting module 
determines a basic support $\tau$-tilting pair uniquely and hence we can identify basic support $\tau$-tilting modules with basic support $\tau$-tilting pairs (see \cite{AIR14}).
We denote by $\stautilt A$ (respectively, $\taurigid A$, $\itaurigid A$) the set of isomorphism classes of basic support $\tau$-tilting (respectively, basic $\tau$-rigid, indecomposable $\tau$-rigid) $A$-modules, by $\taurigidpair A$ the set of isomorphism classes of basic $\tau$-rigid pairs for $A$. 

Next, we recall the notion of torsion classes. 

\begin{defi}Let $\TT$ be a full subcategory of $\mod A$. 
We call $\TT$ \emph{torsion class} (respectively, \emph{torsion-free class}) if it is closed under taking factor modules (respectively, submodules) and extensions. An $A$-module $M \in \mathcal{T}$ is called \emph{Ext-projective} if $\Ext^1_A(M,X)=0$ for all $X \in \mathcal{T}$.
We denote by $P(\mathcal{T})$ the direct sum of one copy of each of the indecomposable Ext-projective modules in $\mathcal{T}$ up to isomorphism.
\end{defi}

We denote by $\tors A$ (respectively, $\ftors A$) the set of torsion classes (respectively, functorially finite torsion classes) in $\mod A$. 

We use the following relationship between support $\tau$-tilting modules and torsion classes.

\begin{definition-theorem}
\label{poset iso}
Let $A$ be a finite dimensional $k$-algebra. 
\begin{enumerate}[\rm (1)]
\item \cite[Theorem 2.7]{AIR14} There is a bijection 
\[\stautilt A\xrightarrow{\sim}\ftors A\]
given by $\stautilt A\ni T\mapsto\Fac T\in\ftors A$ and $\ftors A \ni \mathcal{T} \mapsto P(\mathcal{T})\in \stautilt A$. 
Via this correspondence, we regard $\stautilt A$ as a poset defined by inclusion of $\ftors A$. 
\item \cite[Theorem 1.2]{DIJ19} The following conditions are equivalent: 
\begin{enumerate}
    \item[\rm (a)] The set $\itaurigid A$ is finite.  
    \item[\rm (b)] The set $\stautilt A$ is finite.
    \item[\rm (c)] $\tors A = \ftors A$. 
\end{enumerate} 
If one of the above equivalent conditions (a)-(c) holds, then we say that $A$ is \emph{g-finite} (This is called \emph{$\tau$-tilting finite} in \cite{DIJ19}).
\end{enumerate}
\end{definition-theorem}

For a given $U\in\taurigid A$, let
$$\stautilt_U (A) := \{T \in \stautilt A \mid U\in\add T\}.$$

\begin{theorem}\cite[Theorem 2.10]{AIR14} \label{completion}
Let $U$ be a basic $\tau$-rigid module. Then, ${^\perp}(\tau U)$ is a functorially finite torsion class and $P({^\perp}(\tau U))$ is a basic $\tau$-tilting $A$-module having $U$ as its direct summand. 
\end{theorem}

In the above, $T := P({^\perp}(\tau U))$ is called the \emph{maximal completion} of $U$. 
In addition, let $C_{U} := \End_A(T)/[U]$, where $[U]$ denotes the ideal consisting of all endomorphisms $f\colon T\to T$ factoring through $\add U$. We call $C_U$ \emph{the algebra obtained by the reduction at $U$}.  

\begin{thm}\label{reduction}\cite{Jasso15} 
Let $U$ be a basic $\tau$-rigid $A$-module. 
Then, we have $|C_U| = |A| - |U|$ and a poset isomorphism 
$$\stautilt_U (A)\cong \stautilt C_U.$$
\end{thm}

\begin{exam}\label{example:reduction}
    Let $A$ be a finite dimensional algebra over $k$. 
    For a given primitive idempotent $e\in A$, the maximal completion of $eA$ is $A$, and hence $C_{eA} \cong A/AeA$. 
\end{exam}

\subsection{Preprojective algebras}\label{sec:ppa}

Let $Q = (Q_0,Q_1)$ be a Dynkin quiver of type $\mathbb{A}$, $\mathbb{D}$ or $\mathbb{E}$, where $Q_0$ (respectively, $Q_1$) is the set of vertices (respectively, arrows) of $Q$. 
We recall the definition of preprojective algebras. 

\begin{defi}\label{preprojective}
The preprojective algebra associated to $Q$ is the algebra
$$
\Pi(Q):=k\overline{Q}/\langle \sum_{a\in Q_1} (aa^* - a^*a)\rangle, 
$$
where $\overline{Q}$ is the double quiver of $Q$, which is obtained from $Q$ by adding for each arrow 
$a:i\rightarrow j$ in $Q_1$ an arrow $a^*:i\leftarrow j$ pointing in the opposite direction. 
\end{defi}

Let $Q$ be a Dynkin quiver with $n$ vertices and $\Pi=\Pi(Q)$ the preprojective algebra of $Q$.  
Let $W=W(Q):=\langle s_i \mid i\in Q_0\rangle$ be the Coxeter group of the underlying graph of the quiver $Q$.  
We denote by $l(w)$ the length of $w\in W.$ 
For $u,w\in W$,  
we write $u\leq w$ if there exist $s_{i_1},\ldots, s_{i_m}$ such that $w=s_{i_m}\ldots s_{i_1}u$ and $l(s_{i_j}\ldots s_{i_1}u)=l(u)+j$ for $0\leq j\leq m$. 
Then, the relation $\leq$ gives a partial order on $W$ called \emph{(left) weak order}.

We recall a bijection between $\stautilt\Pi$ and $W$. 
For each $i\in Q_0$, let $I_i:=\Pi(1-e_i)\Pi$ be an ideal in $\Pi$, where $e_i$ is the primitive idempotent of $\Pi$ corresponding to $i$. 
We denote by $\langle I_1,\ldots,I_n\rangle$ the set of ideals of $\Pi$ which can be written as 
$$I_{i_1}I_{i_2}\cdots I_{i_m}$$ for some $m\geq0$ and $i_1,\ldots,i_m\in Q_0$. 
Then, we have the following result.

\begin{theorem}\label{tau weyl}\cite{Mizuno14}
We have a bijection $W\to\langle I_1,\ldots,I_n\rangle$, which is given by $w\mapsto I_w =
I_{i_1}I_{i_{2}}\cdots I_{i_m}$ for any reduced 
expression $w=s_{i_1}\cdots s_{i_m}$.
Moreover, it gives a poset anti-isomorphism between 
$$W \overset{\sim}{\longrightarrow} \stautilt \Pi.$$ 
In particular, every indecomposable $\tau$-rigid module is of the form $e_iI_w$, and hence it is a submodule of $e_i\Pi$ for some $i\in Q_0$. 
\end{theorem}

By Theorem \ref{tau weyl}, $\#\stautilt \Pi = \#W$ holds and it is given by the following table.  
\begin{equation}\label{table:orderW}
\begin{tabular}{ccccccccc}\hline
    $Q$& $\mathbb{A}_n$ & $\mathbb{D}_n$& $\mathbb{E}_6$ &$\mathbb{E}_7$ &$ \mathbb{E}_8$ \\ 
    $\# W$ & $(n+1)!$  & $2^{n-1} n!$ & $51840$& $2903040$ & $696729600$ \\ \hline 
\end{tabular}    
\end{equation}

\section{$d$-polynomials in $\tau$-tilting theory}
Throughout this section, let $A$ be a finite dimensional $k$-algebra. 
We assume that $|A|=n$ and $A$ is $g$-finite. 
An aim of this section is to introduce a generating function, which we call $d$-polynomial of $A$, for the $k$-dimensions of $\tau$-rigid $A$-modules.

\subsection{Simplicial complexes of $\tau$-rigid pairs} \label{Enumeration}
We begin with the definition of a simplicial complex  arising from $\tau$-rigid pairs. 

\begin{defi}
    For $j\geq 0$, let 
    $$\taurigidpair^j A := \{(M,P)\in \taurigidpair A \mid \text{$|M|+|P| = j$}\}.$$ 
    We define a simplicial complex $\Delta(A)$ as follows: The set of $j$-simplices is given by $\taurigidpair^{j+1}A$.  
\end{defi}

Notice that $\Delta(A)$ is naturally isomorphic to the simplicial complex given by $2$-term silting complexes and called \emph{$g$-simplicial complex} \cite{AHIKM22}. 
We denote by 
\begin{equation*}
    (f_{-1},f_0,\ldots,f_{n-1}) \quad \text{and} \quad (h_{0},h_1,\ldots,h_{n})
\end{equation*}
the $f$-vector and $h$-vector of $\Delta(A)$. 
That is, 
\begin{equation*}
    f_{-1} := 1 \quad \text{and} \quad f_{j} := \# \taurigidpair^{j+1} A
\end{equation*}
is the number of $j$-simplices of $\Delta(A)$ for $0\leq j \leq n-1$, and $h_{j}$ is defined by 
\begin{equation*}
    h_j = \sum_{i=0}^j (-1)^{j-i} \binom{n-i}{j-i}f_{i-1}\quad  (0 \leq j \leq n). 
\end{equation*}
In addition, we set  
\begin{equation*}
    f(\Delta(A);t) := \sum_{j=0}^n f_{j-1} t^{n-j} \quad \text{and} \quad h(\Delta(A); t):= \sum_{j=0}^n h_{j}t^{n-j} 
\end{equation*}
and call them \emph{$f$-polynomial} and \emph{$h$-polynomial} of $\Delta(A)$, which are related by \begin{equation}\label{h=f-1}
h(\Delta(A);t) = f(\Delta(A); t-1). 
\end{equation}
Thus, 
\begin{equation}\label{h to f}
    f_{j-1} = \sum_{i=0}^j \binom{n-i}{j-i}h_{i} \quad (0\leq j \leq n).  
\end{equation}
For simplicity, we will write $f(A;t) := f(\Delta(A);t)$ and $h(A;t) := h(\Delta(A);t)$. 

Now, we introduce the following polynomial in one variable $t$ associated with $A$. 

\begin{defi}
Let 
\begin{equation}\label{dim-polynomial}
    d(A;t) := \sum_{(M,P)\in \taurigidpair A} \dim_kM \cdot t^{n-|M|-|P|} = \sum_{j=1}^n d_{j-1}t^{n-j}.
\end{equation}
That is, each coefficient $d_j\in \mathbb{Z}$ computes the sum of the $k$-dimensions 
\begin{equation*}
    d_j = \sum_{(M,P) \in \taurigidpair^{j+1} A} \dim_{k} M \quad (0 \leq j \leq n-1).  
\end{equation*}
We call $d(A;t)$ the \emph{$d$-polynomial} of $A$. 
\end{defi}

By definition, the number $d_0$ (respectively, $d_{n-1}$) is the sum of the $k$-dimensions of all indecomposable $\tau$-rigid (respectively, basic support $\tau$-tilting) modules up to isomorphism. 
Namely, 
\begin{eqnarray}\label{d_0 and d_n-1}
    \Dim(\itaurigid A) &:=& d_0 = \sum_{M\in \itaurigid A} \dim_kM \quad \text{and} \\ \label{st:d_0 and d_n-1} 
    \Dim(\stautilt A) 
    &:=& d_{n-1} = \sum_{M\in \stautilt A} \dim_kM.
\end{eqnarray}

\begin{remk} \label{remk:EJR-resuction}
Let $I$ be an ideal in $A$ which is nilpotent and is contained in the center of $A$. 
Then, it is shown in \cite{EJR18} (see also \cite{DIRRT23}) that there are isomorphisms 
\begin{equation*}
    \stautilt A \simeq \stautilt (A/I) 
    \quad \text{and} \quad \Delta(A)\simeq \Delta(A/I). 
\end{equation*} 
In particular, we have $f(A;t) = f(A/I;t)$ and $h(A;t) = h(A/I;t)$. 
However, their $d$-polynomials $d(A;t)$ and $d(A/I;t)$ are different in general. 
For example, every local algebra $\Lambda$ has the $d$-polynomial $d(\Lambda;t) = \dim_k\Lambda$ which is constant. 
\end{remk}

\begin{exam}\label{ex:d-polynomial}
\begin{enumerate}[\rm (1)]
    \item Let $Q$ be a quiver $Q=(1\to 2 \to 3)$ and $kQ$ the path algebra of $Q$. 
    Then, its $g$-simplicial complex $\Delta(kQ)$ is described in the left figure of Figure \ref{fig:example g}, where $P_i[1]$ means a $\tau$-rigid pair $(0,P_i)$ for the indecomposable projective module $P_i$ corresponding to $i$. 
    In fact, we have 
    \begin{eqnarray*}
    \taurigidpair^1 kQ &=& \left\{
    \footnotesize
    \begin{tikzpicture}[baseline=-3mm]
    \coordinate(x) at(45:0.25);
    \coordinate(y) at(-45:0.25);
    \node at($0*(x)+0*(y)$) {$1$}; 
    \node at($0*(x)+1*(y)$) {$2$}; 
    \node at($0*(x)+2*(y)$) {$3$}; 
    \end{tikzpicture}, 
    \begin{tikzpicture}[baseline=-2mm]
    \coordinate(x) at(45:0.25);
    \coordinate(y) at(-45:0.25);
    \node at($0*(x)+0*(y)$) {$2$}; 
    \node at($0*(x)+1*(y)$) {$3$}; 
    \end{tikzpicture}, 
    \begin{tikzpicture}[baseline=-2mm]
    \coordinate(x) at(45:0.25);
    \coordinate(y) at(-45:0.25);
    \node at($0*(x)+0*(y)$) {$3$}; 
    \end{tikzpicture},
    \begin{tikzpicture}[baseline=-2mm]
    \coordinate(x) at(45:0.25);
    \coordinate(y) at(-45:0.25);
    \node at($0*(x)+0*(y)$) {$1$}; 
    \end{tikzpicture}, 
    \begin{tikzpicture}[baseline=-2mm]
    \coordinate(x) at(45:0.25);
    \coordinate(y) at(-45:0.25);
    \node at($0*(x)+0*(y)$) {$2$}; 
    \end{tikzpicture}, 
    \begin{tikzpicture}[baseline=-2mm]
    \coordinate(x) at(45:0.25);
    \coordinate(y) at(-45:0.25);
    \node at($0*(x)+0*(y)$) {$1$}; 
    \node at($0*(x)+1*(y)$) {$2$}; 
    \end{tikzpicture}, 
    P_1[1], P_2[1], P_3[1] 
    \right\}, \\ 
    \taurigidpair^2 kQ &=& 
    \footnotesize
    \left\{
    \begin{tikzpicture}[baseline=-3mm]
    \coordinate(x) at(45:0.25);
    \coordinate(y) at(-45:0.25);
    \node at($0*(x)+0*(y)$) {$1$}; 
    \node at($0*(x)+1*(y)$) {$2$}; 
    \node at($0*(x)+2*(y)$) {$3$}; 
    \end{tikzpicture}
    \oplus 
    \begin{tikzpicture}[baseline=-2mm]
    \coordinate(x) at(45:0.25);
    \coordinate(y) at(-45:0.25);
    \node at($0*(x)+0*(y)$) {$2$}; 
    \node at($0*(x)+1*(y)$) {$3$}; 
    \end{tikzpicture}, 
    \begin{tikzpicture}[baseline=-3mm]
    \coordinate(x) at(45:0.25);
    \coordinate(y) at(-45:0.25);
    \node at($0*(x)+0*(y)$) {$1$}; 
    \node at($0*(x)+1*(y)$) {$2$}; 
    \node at($0*(x)+2*(y)$) {$3$}; 
    \end{tikzpicture}
    \oplus 
    \begin{tikzpicture}[baseline=-2mm]
    \coordinate(x) at(45:0.25);
    \coordinate(y) at(-45:0.25);
    \node at($0*(x)+0*(y)$) {$3$}; 
    \end{tikzpicture},
    \begin{tikzpicture}[baseline=-3mm]
    \coordinate(x) at(45:0.25);
    \coordinate(y) at(-45:0.25);
    \node at($0*(x)+0*(y)$) {$1$}; 
    \node at($0*(x)+1*(y)$) {$2$}; 
    \node at($0*(x)+2*(y)$) {$3$}; 
    \end{tikzpicture}
    \oplus 
    \begin{tikzpicture}[baseline=-2mm]
    \coordinate(x) at(45:0.25);
    \coordinate(y) at(-45:0.25);
    \node at($0*(x)+0*(y)$) {$1$}; 
    \end{tikzpicture},
    \begin{tikzpicture}[baseline=-3mm]
    \coordinate(x) at(45:0.25);
    \coordinate(y) at(-45:0.25);
    \node at($0*(x)+0*(y)$) {$1$}; 
    \node at($0*(x)+1*(y)$) {$2$}; 
    \node at($0*(x)+2*(y)$) {$3$}; 
    \end{tikzpicture}
    \oplus 
    \begin{tikzpicture}[baseline=-2mm]
    \coordinate(x) at(45:0.25);
    \coordinate(y) at(-45:0.25);
    \node at($0*(x)+0*(y)$) {$1$}; 
    \node at($0*(x)+1*(y)$) {$2$}; 
    \end{tikzpicture}, 
    \begin{tikzpicture}[baseline=-3mm]
    \coordinate(x) at(45:0.25);
    \coordinate(y) at(-45:0.25);
    \node at($0*(x)+0*(y)$) {$1$}; 
    \node at($0*(x)+1*(y)$) {$2$}; 
    \node at($0*(x)+2*(y)$) {$3$}; 
    \end{tikzpicture}
    \oplus 
    \begin{tikzpicture}[baseline=-2mm]
    \coordinate(x) at(45:0.25);
    \coordinate(y) at(-45:0.25);
    \node at($0*(x)+0*(y)$) {$2$}; 
    \end{tikzpicture},
    \begin{tikzpicture}[baseline=-2mm]
    \coordinate(x) at(45:0.25);
    \coordinate(y) at(-45:0.25);
    \node at($0*(x)+0*(y)$) {$2$}; 
    \node at($0*(x)+1*(y)$) {$3$}; 
    \end{tikzpicture}
    \oplus 
    \begin{tikzpicture}[baseline=-2mm]
    \coordinate(x) at(45:0.25);
    \coordinate(y) at(-45:0.25);
    \node at($0*(x)+0*(y)$) {$3$}; 
    \end{tikzpicture},
    \right. \\   && \quad \ \footnotesize
    \begin{tikzpicture}[baseline=-2mm]
    \coordinate(x) at(45:0.25);
    \coordinate(y) at(-45:0.25);
    \node at($0*(x)+0*(y)$) {$2$}; 
    \node at($0*(x)+1*(y)$) {$3$}; 
    \end{tikzpicture} 
    \oplus 
    \begin{tikzpicture}[baseline=-2mm]
    \coordinate(x) at(45:0.25);
    \coordinate(y) at(-45:0.25);
    \node at($0*(x)+0*(y)$) {$2$}; 
    \end{tikzpicture}, 
    \begin{tikzpicture}[baseline=-2mm]
    \coordinate(x) at(45:0.25);
    \coordinate(y) at(-45:0.25);
    \node at($0*(x)+0*(y)$) {$2$}; 
    \node at($0*(x)+1*(y)$) {$3$}; 
    \end{tikzpicture} 
    \oplus P_1[1],  
    \begin{tikzpicture}[baseline=-2mm]
    \coordinate(x) at(45:0.25);
    \coordinate(y) at(-45:0.25);
    \node at($0*(x)+0*(y)$) {$3$}; 
    \end{tikzpicture} 
    \oplus 
    \begin{tikzpicture}[baseline=-2mm]
    \coordinate(x) at(45:0.25);
    \coordinate(y) at(-45:0.25);
    \node at($0*(x)+0*(y)$) {$1$}; 
    \end{tikzpicture}, 
    \begin{tikzpicture}[baseline=-2mm]
    \coordinate(x) at(45:0.25);
    \coordinate(y) at(-45:0.25);
    \node at($0*(x)+0*(y)$) {$3$}; 
    \end{tikzpicture} 
    \oplus P_1[1], 
    \begin{tikzpicture}[baseline=-2mm]
    \coordinate(x) at(45:0.25);
    \coordinate(y) at(-45:0.25);
    \node at($0*(x)+0*(y)$) {$3$}; 
    \end{tikzpicture} 
    \oplus P_2[1], 
    \begin{tikzpicture}[baseline=-2mm]
    \coordinate(x) at(45:0.25);
    \coordinate(y) at(-45:0.25);
    \node at($0*(x)+0*(y)$) {$1$}; 
    \end{tikzpicture} 
    \oplus 
    \begin{tikzpicture}[baseline=-2mm]
    \coordinate(x) at(45:0.25);
    \coordinate(y) at(-45:0.25);
    \node at($0*(x)+0*(y)$) {$1$}; 
    \node at($0*(x)+1*(y)$) {$2$}; 
    \end{tikzpicture},  \\   && \quad \ \footnotesize
    \begin{tikzpicture}[baseline=-2mm]
    \coordinate(x) at(45:0.25);
    \coordinate(y) at(-45:0.25);
    \node at($0*(x)+0*(y)$) {$1$}; 
    \end{tikzpicture} 
    \oplus P_2[1], 
    \begin{tikzpicture}[baseline=-2mm]
    \coordinate(x) at(45:0.25);
    \coordinate(y) at(-45:0.25);
    \node at($0*(x)+0*(y)$) {$1$}; 
    \end{tikzpicture} 
    \oplus P_3[1], 
    \begin{tikzpicture}[baseline=-2mm]
    \coordinate(x) at(45:0.25);
    \coordinate(y) at(-45:0.25);
    \node at($0*(x)+0*(y)$) {$2$}; 
    \end{tikzpicture} 
    \oplus 
    \begin{tikzpicture}[baseline=-2mm]
    \coordinate(x) at(45:0.25);
    \coordinate(y) at(-45:0.25);
    \node at($0*(x)+0*(y)$) {$1$}; 
    \node at($0*(x)+1*(y)$) {$2$}; 
    \end{tikzpicture}, 
    \begin{tikzpicture}[baseline=-2mm]
    \coordinate(x) at(45:0.25);
    \coordinate(y) at(-45:0.25);
    \node at($0*(x)+0*(y)$) {$2$}; 
    \end{tikzpicture} 
    \oplus P_1[1], 
    \begin{tikzpicture}[baseline=-2mm]
    \coordinate(x) at(45:0.25);
    \coordinate(y) at(-45:0.25);
    \node at($0*(x)+0*(y)$) {$2$}; 
    \end{tikzpicture} 
    \oplus P_3[1], 
    \begin{tikzpicture}[baseline=-2mm]
    \coordinate(x) at(45:0.25);
    \coordinate(y) at(-45:0.25);
    \node at($0*(x)+0*(y)$) {$1$}; 
    \node at($0*(x)+1*(y)$) {$2$}; 
    \end{tikzpicture} 
    \oplus P_3[1], 
    \\   && \quad\ \footnotesize  
    \left. 
    P_1[1] \oplus P_2[1], P_1[1]\oplus P_3[1], P_2[1]\oplus P_3[1] 
    \begin{tikzpicture}[baseline=-3mm]
    \coordinate(x) at(45:0.25);
    \coordinate(y) at(-45:0.25);
    \node at($1*(x)+0*(y)$) {}; 
    \end{tikzpicture}\!\!\!
    \right\}, \\ 
    \taurigidpair^3 kQ &=& \left\{ \footnotesize
    \begin{tikzpicture}[baseline=-3mm]
    \coordinate(x) at(45:0.25);
    \coordinate(y) at(-45:0.25);
    \node at($0*(x)+0*(y)$) {$1$}; 
    \node at($0*(x)+1*(y)$) {$2$}; 
    \node at($0*(x)+2*(y)$) {$3$}; 
    \end{tikzpicture}
    \oplus 
    \begin{tikzpicture}[baseline=-3mm]
    \coordinate(x) at(45:0.25);
    \coordinate(y) at(-45:0.25);
    \node at($0*(x)+0*(y)$) {$2$}; 
    \node at($0*(x)+1*(y)$) {$3$}; 
    \end{tikzpicture}
    \oplus
    \begin{tikzpicture}[baseline=-3mm]
    \coordinate(x) at(45:0.25);
    \coordinate(y) at(-45:0.25);
    \node at($0*(x)+0*(y)$) {$3$}; 
    \end{tikzpicture}, 
    \begin{tikzpicture}[baseline=-3mm]
    \coordinate(x) at(45:0.25);
    \coordinate(y) at(-45:0.25);
    \node at($0*(x)+0*(y)$) {$1$}; 
    \node at($0*(x)+1*(y)$) {$2$}; 
    \node at($0*(x)+2*(y)$) {$3$}; 
    \end{tikzpicture}
    \oplus 
    \begin{tikzpicture}[baseline=-3mm]
    \coordinate(x) at(45:0.25);
    \coordinate(y) at(-45:0.25);
    \node at($0*(x)+0*(y)$) {$2$}; 
    \node at($0*(x)+1*(y)$) {$3$}; 
    \end{tikzpicture}
    \oplus
    \begin{tikzpicture}[baseline=-3mm]
    \coordinate(x) at(45:0.25);
    \coordinate(y) at(-45:0.25);
    \node at($0*(x)+0*(y)$) {$2$}; 
    \end{tikzpicture}, 
    \begin{tikzpicture}[baseline=-3mm]
    \coordinate(x) at(45:0.25);
    \coordinate(y) at(-45:0.25);
    \node at($0*(x)+0*(y)$) {$1$}; 
    \node at($0*(x)+1*(y)$) {$2$}; 
    \node at($0*(x)+2*(y)$) {$3$}; 
    \end{tikzpicture}
    \oplus 
    \begin{tikzpicture}[baseline=-3mm]
    \coordinate(x) at(45:0.25);
    \coordinate(y) at(-45:0.25);
    \node at($0*(x)+0*(y)$) {$1$}; 
    \end{tikzpicture}
    \oplus
    \begin{tikzpicture}[baseline=-3mm]
    \coordinate(x) at(45:0.25);
    \coordinate(y) at(-45:0.25);
    \node at($0*(x)+0*(y)$) {$3$}; 
    \end{tikzpicture}, 
    \begin{tikzpicture}[baseline=-3mm]
    \coordinate(x) at(45:0.25);
    \coordinate(y) at(-45:0.25);
    \node at($0*(x)+0*(y)$) {$1$}; 
    \node at($0*(x)+1*(y)$) {$2$}; 
    \node at($0*(x)+2*(y)$) {$3$}; 
    \end{tikzpicture}
    \oplus 
    \begin{tikzpicture}[baseline=-3mm]
    \coordinate(x) at(45:0.25);
    \coordinate(y) at(-45:0.25);
    \node at($0*(x)+0*(y)$) {$1$}; 
    \node at($0*(x)+1*(y)$) {$2$}; 
    \end{tikzpicture}
    \oplus
    \begin{tikzpicture}[baseline=-3mm]
    \coordinate(x) at(45:0.25);
    \coordinate(y) at(-45:0.25);
    \node at($0*(x)+0*(y)$) {$2$}; 
    \end{tikzpicture}, \right. 
    \\ && \quad \footnotesize
    \begin{tikzpicture}[baseline=-3mm]
    \coordinate(x) at(45:0.25);
    \coordinate(y) at(-45:0.25);
    \node at($0*(x)+0*(y)$) {$1$}; 
    \node at($0*(x)+1*(y)$) {$2$}; 
    \node at($0*(x)+2*(y)$) {$3$}; 
    \end{tikzpicture}
    \oplus 
    \begin{tikzpicture}[baseline=-3mm]
    \coordinate(x) at(45:0.25);
    \coordinate(y) at(-45:0.25);
    \node at($0*(x)+0*(y)$) {$1$}; 
    \node at($0*(x)+1*(y)$) {$2$}; 
    \end{tikzpicture}
    \oplus
    \begin{tikzpicture}[baseline=-3mm]
    \coordinate(x) at(45:0.25);
    \coordinate(y) at(-45:0.25);
    \node at($0*(x)+0*(y)$) {$1$}; 
    \end{tikzpicture},   
    \begin{tikzpicture}[baseline=-3mm]
    \coordinate(x) at(45:0.25);
    \coordinate(y) at(-45:0.25);
    \node at($0*(x)+0*(y)$) {$2$}; 
    \node at($0*(x)+1*(y)$) {$3$};  
    \end{tikzpicture}
    \oplus 
    \begin{tikzpicture}[baseline=-3mm]
    \coordinate(x) at(45:0.25);
    \coordinate(y) at(-45:0.25);
    \node at($0*(x)+0*(y)$) {$3$}; 
    \end{tikzpicture} \oplus P_1[1], 
    \begin{tikzpicture}[baseline=-3mm]
    \coordinate(x) at(45:0.25);
    \coordinate(y) at(-45:0.25);
    \node at($0*(x)+0*(y)$) {$2$}; 
    \node at($0*(x)+1*(y)$) {$3$};  
    \end{tikzpicture}
    \oplus 
    \begin{tikzpicture}[baseline=-3mm]
    \coordinate(x) at(45:0.25);
    \coordinate(y) at(-45:0.25);
    \node at($0*(x)+0*(y)$) {$2$}; 
    \end{tikzpicture} \oplus P_1[1], 
    \begin{tikzpicture}[baseline=-3mm]
    \coordinate(x) at(45:0.25);
    \coordinate(y) at(-45:0.25);
    \node at($0*(x)+0*(y)$) {$3$};  
    \end{tikzpicture}
    \oplus 
    \begin{tikzpicture}[baseline=-3mm]
    \coordinate(x) at(45:0.25);
    \coordinate(y) at(-45:0.25);
    \node at($0*(x)+0*(y)$) {$1$}; 
    \end{tikzpicture} \oplus P_2[1], 
    \\ && \quad \footnotesize 
    \begin{tikzpicture}[baseline=-3mm]
    \coordinate(x) at(45:0.25);
    \coordinate(y) at(-45:0.25);
    \node at($0*(x)+0*(y)$) {$3$};  
    \end{tikzpicture}
    \oplus P_1[1] \oplus P_2[1], 
    \begin{tikzpicture}[baseline=-3mm]
    \coordinate(x) at(45:0.25);
    \coordinate(y) at(-45:0.25);
    \node at($0*(x)+0*(y)$) {$1$};  
    \end{tikzpicture}
    \oplus 
    \begin{tikzpicture}[baseline=-3mm]
    \coordinate(x) at(45:0.25);
    \coordinate(y) at(-45:0.25);
    \node at($0*(x)+0*(y)$) {$1$};
    \node at($0*(x)+1*(y)$) {$2$};
    \end{tikzpicture} \oplus P_3[1], 
    \begin{tikzpicture}[baseline=-3mm]
    \coordinate(x) at(45:0.25);
    \coordinate(y) at(-45:0.25);
    \node at($0*(x)+0*(y)$) {$1$};  
    \end{tikzpicture}
    \oplus P_2[1] \oplus P_3[1], 
    \\ && \quad \footnotesize 
    \left. 
    \begin{tikzpicture}[baseline=-3mm]
    \coordinate(x) at(45:0.25);
    \coordinate(y) at(-45:0.25);
    \node at($0*(x)+0*(y)$) {$2$};  
    \end{tikzpicture} 
    \oplus 
    \begin{tikzpicture}[baseline=-3mm]
    \coordinate(x) at(45:0.25);
    \coordinate(y) at(-45:0.25);
    \node at($0*(x)+0*(y)$) {$1$};
    \node at($0*(x)+1*(y)$) {$2$};
    \end{tikzpicture} \oplus P_3[1], 
    \begin{tikzpicture}[baseline=-3mm]
    \coordinate(x) at(45:0.25);
    \coordinate(y) at(-45:0.25);
    \node at($0*(x)+0*(y)$) {$2$};  
    \end{tikzpicture}
    \oplus P_1[1] \oplus P_2[1], 
    P_1[1] \oplus P_2[1] \oplus P_3[1] 
    \right\}. 
    \end{eqnarray*}
    Thus, the $f$-polynomial and $h$-polynomial are given as follows, respectively. 
    \begin{equation*}
        f(kQ;t) = t^3 + 9t^2 + 21t + 14 \quad \text{and} \quad  
        h(kQ;t) = t^3 + 6t^2 + 6t + 1. 
    \end{equation*}
    In addition, the $d$-polynomial is given by 
    \begin{equation*}    
    d(kQ;t) = 10t^2 + 46t + 46.
    \end{equation*}
    \item Let $A:=kQ/I$, where
    \begin{equation*}
        Q = \left( \xymatrix{1 \ar@<1mm>[r]^-{a} \ar@<-1mm>@{<-}[r]_-{a^*} & 2 \ar@<1mm>[r]^-{b} \ar@<-1mm>@{<-}[r]_-{b^*} & 3 }\right) 
        \quad \text{and} \quad 
        I = \langle aa^*, a^*a - bb^*, b^*b \rangle. 
    \end{equation*}
    This is the preprojective algebra of type $\mathbb{A}_3$. 
    In this case, the $g$-simplicial complex of $A$ is described in the right figure of Figure \ref{fig:example g}. 
    Thus, we have 
    \begin{equation*}
        f(A;t) = t^3 + 14t^2 + 36t + 24 \quad \text{and} \quad  
        h(A;t) = t^3 + 11t^2 + 11t + 1.   
    \end{equation*}
    From a direct calculation, we find that the $d$-polynomial is given by 
    \begin{equation*}
        d(A;t) = 24t^2 + 120t + 120.    
    \end{equation*}
\end{enumerate}

\begin{figure}[t]
\begin{tabular}{ccccc}
    $\begin{tikzpicture}[baseline=0mm, scale=1.1]
    \node(x) at(200:0.8) {}; 
    \node(-x) at($-1*(x)$) {}; 
    \node(y) at(-10:1.9) {}; 
    \node(-y) at($-1*(y)$) {}; 
    \node(z) at(90:1.9) {}; 
    \node(-z) at($-1*(z)$) {}; 

    \coordinate(1) at($1*(x) + 0*(y) + 0*(z)$); 
    \coordinate(2) at($0*(x) + 1*(y) + 0*(z)$); 
    \coordinate(3) at($0*(x) + 0*(y) + 1*(z)$); 
    \coordinate(4) at($-1*(x) + 0*(y) + 0*(z)$); 
    \coordinate(5) at($0*(x) + -1*(y) + 0*(z)$); 
    \coordinate(6) at($0*(x) + 0*(y) + -1*(z)$); 
    \coordinate(7) at($1*(x) + -1*(y) + 0*(z)$); 
    \coordinate(8) at($0*(x) + 1*(y) + -1*(z)$); 
    \coordinate(9) at($1*(x) + 0*(y) + -1*(z)$); 

    \draw[dashed] (4)--(5); 
    \draw[dashed] (4)--(6); 
    \draw[dashed] (5)--(6); 
    \draw[dashed] (2)--(4); 
    \draw[dashed] (3)--(4); 
    \draw[dashed] (3)--(5); 
    \draw[dashed] (4)--(8); 
    \draw[dashed] (5)--(7); 
    \draw[dashed] (8)--(6); 
    \draw[dashed] (7)--(6); 
    \draw[dashed] (9)--(6); 

    \draw[] (1)--(2); 
    \draw[] (1)--(3); 
    \draw[] (2)--(3); 
    \draw[] (1)--(7); 
    \draw[] (3)--(7); 
    \draw[] (1)--(8); 
    \draw[] (2)--(8); 
    \draw[] (1)--(9); 
    \draw[] (7)--(9); 
    \draw[] (8)--(9); 
    
    \node[fill=white, inner sep = 0.1mm, shift={(0.2,0)}] at (1) 
    {$\scriptsize
    \begin{tikzpicture}[baseline=2mm]
    \coordinate(x) at(45:0.25);
    \coordinate(y) at(-45:0.25);
    \node at($0*(x)+0*(y)$) {$1$}; 
    \node at($0*(x)+1*(y)$) {$2$}; 
    \node at($0*(x)+2*(y)$) {$3$}; 
    \end{tikzpicture}$
    };
    \node[fill=white, inner sep = 0.5mm] at (2) 
    {$\scriptsize
    \begin{tikzpicture}[baseline=2mm]
    \coordinate(x) at(45:0.25);
    \coordinate(y) at(-45:0.25);
    \node at($0*(x)+0*(y)$) {$2$}; 
    \node at($0*(x)+1*(y)$) {$3$}; 
    \end{tikzpicture}$
    };
    \node[fill=white, inner sep = 0.5mm] at (3) 
    {\scriptsize$
    \begin{tikzpicture}[baseline=2mm]
    \coordinate(x) at(45:0.25);
    \coordinate(y) at(-45:0.25);
    \node at($0*(x)+0*(y)$) {$3$}; 
    \end{tikzpicture}$
    };
    \node[fill=white, inner sep = 0.5mm] at (4) 
    {\scriptsize $P_1[1]$
    };
    \node[fill=white, inner sep = 0mm, shift={(0.2,0)}] at (5) 
    {\scriptsize $P_2[1]$};
    \node[fill=white, inner sep = 0.5mm] at (6) 
    {\scriptsize $P_3[1]$};
    \node[fill=white, inner sep = 0.5mm] at (7) 
    {\scriptsize $
    \begin{tikzpicture}[baseline=2mm]
    \coordinate(x) at(45:0.25);
    \coordinate(y) at(-45:0.25);
    \node at($0*(x)+0*(y)$) {$1$}; 
    \end{tikzpicture}$
    };
    
    \node[fill=white, inner sep = 0.5mm] at (8) 
    {\scriptsize $
    \begin{tikzpicture}[baseline=2mm]
    \coordinate(x) at(45:0.25);
    \coordinate(y) at(-45:0.25);
    \node at($0*(x)+0*(y)$) {$2$}; 
    \end{tikzpicture}$
    };
    \node[fill=white, inner sep = 0.5mm] at (9) 
    {\scriptsize $
    \begin{tikzpicture}[baseline=2mm]
    \coordinate(x) at(45:0.25);
    \coordinate(y) at(-45:0.25);
    \node at($0*(x)+0*(y)$) {$1$}; 
    \node at($0*(x)+1*(y)$) {$2$}; 
    \end{tikzpicture}$
    };
    
\end{tikzpicture}$
& & 
$\begin{tikzpicture}[baseline=0mm, scale=1.1]
    \node(x) at(200:0.9) {}; 
    \node(-x) at($-1*(x)$) {}; 
    \node(y) at(-10:1.8) {}; 
    \node(-y) at($-1*(y)$) {}; 
    \node(z) at(90:1.8) {}; 
    \node(-z) at($-1*(z)$) {};

    \coordinate(1) at($1*(x) + 0*(y) + 0*(z)$); 
    \coordinate(2) at($0*(x) + 1*(y) + 0*(z)$); 
    \coordinate(3) at($0*(x) + 0*(y) + 1*(z)$); 
    \coordinate(4) at($-1*(x) + 0*(y) + 0*(z)$); 
    \coordinate(5) at($0*(x) + -1*(y) + 0*(z)$); 
    \coordinate(6) at($0*(x) + 0*(y) + -1*(z)$); 
    \coordinate(7) at($-1*(x) + 1*(y) + 0*(z)$); 
    \coordinate(8) at($1*(x) + -1*(y) + 1*(z)$); 
    \coordinate(9) at($0*(x) + 1*(y) + -1*(z)$); 
    \coordinate(10) at($-1*(x) + 0*(y) + 1*(z)$); 
    \coordinate(11) at($0*(x) + -1*(y) + 1*(z)$); 
    \coordinate(12) at($1*(x) + -1*(y) + 0*(z)$); 
    \coordinate(13) at($1*(x) + 0*(y) + -1*(z)$); 
    \coordinate(14) at($-1*(x) + 1*(y) + -1*(z)$); 

    \draw[] (1)--(2); 
    \draw[] (1)--(3); 
    \draw[] (2)--(3); 
    \draw[dashed] (4)--(5); 
    \draw[dashed] (4)--(6); 
    \draw[dashed] (5)--(6); 
    \draw[] (2)--(7); 
    \draw[] (3)--(7); 
    \draw[] (1)--(8); 
    \draw[] (3)--(8); 
    \draw[] (1)--(9); 
    \draw[] (2)--(9); 
    \draw[] (3)--(10); 
    \draw[] (7)--(10); 
    \draw[] (7)--(9); 
    \draw[] (3)--(11); 
    \draw[] (8)--(11); 
    \draw[] (1)--(12); 
    \draw[] (8)--(12); 
    \draw[] (1)--(13); 
    \draw[] (9)--(13); 
    \draw[dashed] (7)--(4); 
    \draw[dashed] (10)--(4); 
    \draw[] (10)--(11); 
    \draw[] (7)--(14); 
    \draw[] (9)--(14); 
    \draw[dashed] (11)--(12); 
    \draw[] (12)--(13); 
    \draw[dashed] (9)--(6); 
    \draw[dashed] (13)--(6); 
    \draw[dashed] (11)--(4); 
    \draw[dashed] (14)--(4); 
    \draw[dashed] (12)--(6); 
    \draw[dashed] (14)--(6); 
    \draw[dashed] (11)--(5); 
    \draw[dashed] (12)--(5); 
    \node[fill=white, inner sep = 0.3mm] at (1) 
    {$\scriptsize
    \begin{tikzpicture}[baseline=2mm]
    \coordinate(x) at(-135:0.25);
    \coordinate(y) at(-45:0.25);
    \node at($0*(x)+0*(y)$) {$1$}; 
    \node at($0*(x)+1*(y)$) {$2$}; 
    \node at($0*(x)+2*(y)$) {$3$}; 
    \end{tikzpicture}$
    };
    
    \node[fill=white, inner sep = 0.3mm] at (2) 
    {$\scriptsize
    \begin{tikzpicture}[baseline=2mm]
    \coordinate(x) at(-135:0.25);
    \coordinate(y) at(-45:0.25);
    \node at($0*(x)+0*(y)$) {$2$}; 
    \node at($0*(x)+1*(y)$) {$3$}; 
    \node at($1*(x)+0*(y)$) {$1$};
    \node at($1*(x)+1*(y)$) {$2$};
    \end{tikzpicture}$
    };
    \node[fill=white, inner sep = 0.3mm, shift={(0,-0.1)}] at (3) 
    {$\scriptsize
    \begin{tikzpicture}[baseline=2mm]
    \coordinate(x) at(-135:0.25);
    \coordinate(y) at(-45:0.25);
    \node at($0*(x)+0*(y)$) {$3$}; 
    \node at($1*(x)+0*(y)$) {$2$}; 
    \node at($2*(x)+0*(y)$) {$1$};
    \end{tikzpicture}$
    };
    \node[fill=white, inner sep = 0.3mm] at (4) 
    {\scriptsize $P_1[1]$}; 
    \node[fill=white, inner sep = 0.3mm] at (5) 
    {\scriptsize $P_2[1]$};
    \node[fill=white, inner sep = 0.3mm] at (6) 
    {\scriptsize $P_3[1]$};
    \node[fill=white, inner sep = 0.3mm] at (7) 
    {$\scriptsize
    \begin{tikzpicture}[baseline=2mm]
    \coordinate(x) at(-135:0.25);
    \coordinate(y) at(-45:0.25);
    \node at($0*(x)+0*(y)$) {$2$}; 
    \node at($0*(x)+1*(y)$) {$3$};
    \end{tikzpicture}$
    };
    
    \node[fill=white, inner sep = 0.3mm] at (8) 
    {$\scriptsize
    \begin{tikzpicture}[baseline=2mm]
    \coordinate(x) at(-135:0.25);
    \coordinate(y) at(-45:0.25);
    \node at($0*(x)+1*(y)$) {$3$}; 
    \node at($1*(x)+0*(y)$) {$1$};
    \node at($1*(x)+1*(y)$) {$2$};
    \end{tikzpicture}$
    };
    \node[fill=white, inner sep = 0.3mm] at (9) 
    {$\scriptsize
    \begin{tikzpicture}[baseline=2mm]
    \coordinate(x) at(-135:0.25);
    \coordinate(y) at(-45:0.25);
    \node at($0*(x)+0*(y)$) {$2$}; 
    \node at($1*(x)+0*(y)$) {$1$};
    \end{tikzpicture}$
    };
    \node[fill=white, inner sep = 0.3mm] at (10) 
    {$\scriptsize
    \begin{tikzpicture}[baseline=2mm]
    \coordinate(x) at(-135:0.25);
    \coordinate(y) at(-45:0.25);
    \node at($0*(x)+0*(y)$) {$3$}; 
    \node at($1*(x)+0*(y)$) {$2$};
    \end{tikzpicture}$
    };
    
    \node[fill=white, inner sep = 0.3mm] at (11) 
    {$\scriptsize
    \begin{tikzpicture}[baseline=2mm]
    \coordinate(x) at(-135:0.25);
    \coordinate(y) at(-45:0.25);
    \node at($0*(x)+0*(y)$) {$3$};
    \end{tikzpicture}$
    };
    \node[fill=white, inner sep = 0.3mm] at (12) 
    {$\scriptsize
    \begin{tikzpicture}[baseline=2mm]
    \coordinate(x) at(-135:0.25);
    \coordinate(y) at(-45:0.25);
    \node at($0*(x)+0*(y)$) {$1$};
    \end{tikzpicture}$
    };
    \node[fill=white, inner sep = 0.3mm] at (13) 
    {$\scriptsize
    \begin{tikzpicture}[baseline=2mm]
    \coordinate(x) at(-135:0.25);
    \coordinate(y) at(-45:0.25);
    \node at($0*(x)+0*(y)$) {$1$}; 
    \node at($0*(x)+1*(y)$) {$2$}; 
    \end{tikzpicture}$
    };
    \node[fill=white, inner sep = 0.3mm] at (14) 
    {$\scriptsize
    \begin{tikzpicture}[baseline=2mm]
    \coordinate(x) at(-135:0.25);
    \coordinate(y) at(-45:0.25);
    \node at($0*(x)+0*(y)$) {$2$}; 
    \end{tikzpicture}$
    };
\end{tikzpicture}
$ 
\end{tabular}
    \caption{The $g$-simplicial complexes in Example \ref{ex:d-polynomial}.}
    \label{fig:example g}
\end{figure}
\end{exam}

In the next example, we observe that $f$-polynomials (and $h$-polynomials) do not determine isomorphism classes of $g$-simplicial complexes. This follows from \cite[Theorem 5.1]{AMN20}, which states that every Brauer tree algebra with $n$ edges has the $h$-polynomial of the form $\sum_{j=0}^n\binom{n}{j}^2 t^j$ while their $g$-simplicial complexes are different.  

\begin{exam}\label{example:BTA}
Let $B_1$ and $B_2$ be finite dimensional symmetric algebras defined as follows. 
\begin{eqnarray*}
    B_1 &:=& k\left( \xymatrix{1 \ar@<1mm>[r]^-{a} \ar@<-1mm>@{<-}[r]_-{b} & 2 \ar@<1mm>[r]^-{c} \ar@<-1mm>@{<-}[r]_-{d} & 3 }\right) 
    /\langle aba, dcd, ac, db, ba - cd \rangle, \quad \text{and}  \\ 
    B_2 &:=& k\Biggl( 
    \begin{tikzpicture}[baseline=-2mm]
        \node at (0,0) {\small$
    \xymatrix@C=10pt@R=10pt{& 1 \ar@{->}[rd]^a& \\
    3 \ar@{->}[ru]^c & & 2 \ar@{->}[ll]_b}$}; 
    \end{tikzpicture}
    \Biggr) 
    /\langle abca,bcab,cabc \rangle.
\end{eqnarray*}
Then, $B_1$ (respectively, $B_2$) is a Brauer tree algebra corresponding to a Brauer line (respectively, a Brauer star) having $3$ edges. 
In Figure \ref{fig:BTA}, we describe the $g$-simplicial complex $\Delta(B_1)$ (respectively, $\Delta(B_2)$) on the left (respectively, right) side. 
One can check that they are non-isomorphic to each other but satisfy $h(B_1;t) = h(B_2;t) = t^3 + 9t^2 + 9t + 1$. 
On the other hand, we have  
\begin{equation*}
    d(B_1;t) = 20t^2 + 98t + 98 \quad \text{and}\quad d(B_2;t) = 21t^2 + 102t +102  
\end{equation*}
in this case. 
\end{exam}

\begin{figure}
    \begin{tabular}{cccccc}
$
\begin{tikzpicture}[baseline=0mm, scale=1.1]
    \node(x) at(200:0.8) {}; 
    \node(-x) at($-1*(x)$) {}; 
    \node(y) at(-10:1.9) {}; 
    \node(-y) at($-1*(y)$) {}; 
    \node(z) at(90:1.9) {}; 
    \node(-z) at($-1*(z)$) {}; 
    \coordinate(1) at($1*(x) + 0*(y) + 0*(z)$); 
    \coordinate(2) at($0*(x) + 1*(y) + 0*(z)$); 
    \coordinate(3) at($0*(x) + 0*(y) + 1*(z)$); 
    \coordinate(4) at($-1*(x) + 0*(y) + 0*(z)$); 
    \coordinate(5) at($0*(x) + -1*(y) + 0*(z)$); 
    \coordinate(6) at($0*(x) + 0*(y) + -1*(z)$); 
    \coordinate(7) at($-1*(x) + 1*(y) + 0*(z)$); 
    \coordinate(8) at($1*(x) + -1*(y) + 1*(z)$); 
    \coordinate(9) at($0*(x) + 1*(y) + -1*(z)$); 
    \coordinate(10) at($0*(x) + -1*(y) + 1*(z)$); 
    \coordinate(11) at($1*(x) + -1*(y) + 0*(z)$); 
    \coordinate(12) at($-1*(x) + 1*(y) + -1*(z)$); 

    \draw[] (1)--(2); 
    \draw[] (1)--(3); 
    \draw[] (2)--(3); 
    \draw[dashed] (4)--(5); 
    \draw[dashed] (4)--(6); 
    \draw[dashed] (5)--(6); 
    \draw[] (2)--(7); 
    \draw[] (3)--(7); 
    \draw[] (1)--(8); 
    \draw[] (3)--(8); 
    \draw[] (1)--(9); 
    \draw[] (2)--(9); 
    \draw[dashed] (3)--(4); 
    \draw[dashed] (7)--(4); 
    \draw[] (7)--(9); 
    \draw[] (3)--(10); 
    \draw[] (8)--(10); 
    \draw[] (1)--(11); 
    \draw[] (8)--(11); 
    \draw[] (1)--(6); 
    \draw[] (9)--(6); 
    \draw[] (7)--(12); 
    \draw[dashed] (12)--(4); 
    \draw[dashed] (10)--(4); 
    \draw[] (9)--(12); 
    \draw[dashed] (12)--(6); 
    \draw[] (11)--(6); 
    \draw[dashed] (10)--(11); 
    \draw[dashed] (10)--(5); 
    \draw[dashed] (11)--(5); 

    \node[fill=white, inner sep = 0.5mm] at (1) {$
    \scriptsize
    \begin{tikzpicture}[baseline=2mm]
    \coordinate(x) at(0:0.25);
    \coordinate(y) at(90:0.25);
    \node at($0*(x)+0*(y)$) {$1$}; 
    \node at($0*(x)+1*(y)$) {$2$}; 
    \node at($0*(x)+2*(y)$) {$1$}; 
    \end{tikzpicture}$}; 
    \node[fill=white, inner sep = 0.5mm] at (2) {$
    \scriptsize
    \begin{tikzpicture}[baseline=2mm]
    \coordinate(x) at(45:0.25);
    \coordinate(y) at(-45:0.25);
    \node at($0*(x)+0*(y)$) {$1$}; 
    \node at($1*(x)+0*(y)$) {$2$}; 
    \node at($0*(x)+1*(y)$) {$2$};
    \node at($1*(x)+1*(y)$) {$3$};
    \end{tikzpicture}$}; 
    \node[fill=white, inner sep = 0.5mm] at (3) {$
    \scriptsize
    \begin{tikzpicture}[baseline=2mm]
    \coordinate(x) at(0:0.25);
    \coordinate(y) at(90:0.25);
    \node at($0*(x)+0*(y)$) {$3$}; 
    \node at($0*(x)+1*(y)$) {$2$}; 
    \node at($0*(x)+2*(y)$) {$3$}; 
    \end{tikzpicture}$}; 
    \node[fill=white, inner sep = 0.3mm] at (4) {$
    \scriptsize
    \begin{tikzpicture}[baseline=2mm]
    \coordinate(x) at(0:0.25);
    \coordinate(y) at(90:0.25);
    \node at($0*(x)+0*(y)$) {$P_1[1]$}; 
    \end{tikzpicture}$}; 
    \node[fill=white, inner sep = 0.3mm] at (5) {$
    \scriptsize
    \begin{tikzpicture}[baseline=2mm]
    \coordinate(x) at(0:0.25);
    \coordinate(y) at(90:0.25);
    \node at($0*(x)+0*(y)$) {$P_2[1]$}; 
    \end{tikzpicture}$}; 
    \node[fill=white, inner sep = 0.3mm] at (6) {$
    \scriptsize
    \begin{tikzpicture}[baseline=2mm]
    \coordinate(x) at(0:0.25);
    \coordinate(y) at(90:0.25);
    \node at($0*(x)+0*(y)$) {$P_3[1]$}; 
    \end{tikzpicture}$}; 
    \node[fill=white, inner sep = 0.5mm] at (7) {$
    \scriptsize
    \begin{tikzpicture}[baseline=2mm]
    \coordinate(x) at(0:0.25);
    \coordinate(y) at(90:0.25);
    \node at($0*(x)+2*(y)$) {$2$}; 
    \node at($0*(x)+1*(y)$) {$3$}; 
    \end{tikzpicture}$}; 
    \node[fill=white, inner sep = 0.5mm] at (8) {$\scriptsize
    \begin{tikzpicture}[baseline=2mm]
    \coordinate(x) at(45:0.25);
    \coordinate(y) at(-45:0.25);
    \node at($0*(x)+0*(y)$) {$1$}; 
    \node at($0*(x)+1*(y)$) {$2$}; 
    \node at($1*(x)+1*(y)$) {$3$}; 
    
    \end{tikzpicture}$}; 
    \node[fill=white, inner sep = 0.5mm] at (9) {$
    \scriptsize
    \begin{tikzpicture}[baseline=2mm]
    \coordinate(x) at(0:0.25);
    \coordinate(y) at(90:0.25);
    \node at($0*(x)+2*(y)$) {$2$}; 
    \node at($0*(x)+1*(y)$) {$1$}; 
    \end{tikzpicture}$}; 
    \node[fill=white, inner sep = 0.5mm] at (10) {$
    \scriptsize
    \begin{tikzpicture}[baseline=2mm]
    \coordinate(x) at(0:0.25);
    \coordinate(y) at(90:0.25);
    \node at($0*(x)+0*(y)$) {$3$}; 
    \end{tikzpicture}$}; 
    \node[fill=white, inner sep = 0.5mm] at (11) {$
    \scriptsize
    \begin{tikzpicture}[baseline=2mm]
    \coordinate(x) at(0:0.25);
    \coordinate(y) at(90:0.25);
    \node at($0*(x)+0*(y)$) {$1$}; 
    \end{tikzpicture}$}; 
    \node[fill=white, inner sep = 0.5mm] at (12) {$
    \scriptsize
    \begin{tikzpicture}[baseline=2mm]
    \coordinate(x) at(0:0.25);
    \coordinate(y) at(90:0.25);
    \node at($0*(x)+0*(y)$) {$2$}; 
    \end{tikzpicture}$}; 
\end{tikzpicture}
$  & & 
$
\begin{tikzpicture}[baseline=0mm, scale=1.1]
    \node(x) at(200:0.8) {}; 
    \node(-x) at($-1*(x)$) {}; 
    \node(y) at(-10:1.9) {}; 
    \node(-y) at($-1*(y)$) {}; 
    \node(z) at(90:1.9) {}; 
    \node(-z) at($-1*(z)$) {}; 
    \coordinate(1) at($1*(x) + 0*(y) + 0*(z)$); 
    \coordinate(2) at($0*(x) + 1*(y) + 0*(z)$); 
    \coordinate(3) at($0*(x) + 0*(y) + 1*(z)$); 
    \coordinate(4) at($-1*(x) + 0*(y) + 0*(z)$); 
    \coordinate(5) at($0*(x) + -1*(y) + 0*(z)$); 
    \coordinate(6) at($0*(x) + 0*(y) + -1*(z)$); 
    \coordinate(7) at($-1*(x) + 0*(y) + 1*(z)$); 
    \coordinate(8) at($1*(x) + -1*(y) + 0*(z)$); 
    \coordinate(9) at($0*(x) + 1*(y) + -1*(z)$); 
    \coordinate(10) at($0*(x) + -1*(y) + 1*(z)$); 
    \coordinate(11) at($-1*(x) + 1*(y) + 0*(z)$); 
    \coordinate(12) at($1*(x) + 0*(y) + -1*(z)$); 

    \draw[] (1)--(2); 
    \draw[] (1)--(3); 
    \draw[] (2)--(3); 
    \draw[dashed] (4)--(5); 
    \draw[dashed] (4)--(6); 
    \draw[dashed] (5)--(6); 
    \draw[] (2)--(7); 
    \draw[] (3)--(7); 
    \draw[] (1)--(8); 
    \draw[] (3)--(8); 
    \draw[] (1)--(9); 
    \draw[] (2)--(9); 
    \draw[] (3)--(10); 
    \draw[] (7)--(10); 
    \draw[] (2)--(11); 
    \draw[] (7)--(11); 
    \draw[] (8)--(10); 
    \draw[] (1)--(12); 
    \draw[] (8)--(12); 
    \draw[] (9)--(11); 
    \draw[] (9)--(12); 
    \draw[dashed] (7)--(5); 
    \draw[dashed] (10)--(5); 
    \draw[dashed] (7)--(4); 
    \draw[dashed] (11)--(4); 
    \draw[dashed] (8)--(5); 
    \draw[dashed] (8)--(6); 
    \draw[dashed] (12)--(6); 
    \draw[dashed] (9)--(4); 
    \draw[dashed] (9)--(6); 

    \node[fill=white, inner sep = 0.5mm] at (1) {$
    \scriptsize
    \begin{tikzpicture}[baseline=2mm]
    \coordinate(x) at(0:0.25);
    \coordinate(y) at(90:0.25);
    \node at($0*(x)+2*(y)$) {$1$}; 
    \node at($0*(x)+1*(y)$) {$2$}; 
    \node at($0*(x)+0*(y)$) {$3$}; 
    \node at($0*(x)+-1*(y)$) {$1$};
    \end{tikzpicture}$}; 
    \node[fill=white, inner sep = 0.5mm] at (2) {$
    \scriptsize
    \begin{tikzpicture}[baseline=2mm]
    \coordinate(x) at(0:0.25);
    \coordinate(y) at(90:0.25);
    \node at($0*(x)+2*(y)$) {$2$}; 
    \node at($0*(x)+1*(y)$) {$3$}; 
    \node at($0*(x)+0*(y)$) {$1$}; 
    \node at($0*(x)+-1*(y)$) {$2$};
    \end{tikzpicture}$}; 
    \node[fill=white, inner sep = 0.5mm] at (3) {$
    \scriptsize
    \begin{tikzpicture}[baseline=2mm]
    \coordinate(x) at(0:0.25);
    \coordinate(y) at(90:0.25);
    \node at($0*(x)+2*(y)$) {$3$}; 
    \node at($0*(x)+1*(y)$) {$1$}; 
    \node at($0*(x)+0*(y)$) {$2$};
    \node at($0*(x)+-1*(y)$) {$3$};
    \end{tikzpicture}$}; 
    \node[fill=white, inner sep = 0.3mm] at (4) {$
    \scriptsize
    \begin{tikzpicture}[baseline=2mm]
    \coordinate(x) at(0:0.25);
    \coordinate(y) at(90:0.25);
    \node at($0*(x)+0*(y)$) {$P_1[1]$};
    \end{tikzpicture}$}; 
    \node[fill=white, inner sep = 0.3mm, shift={(0.2,0)}] at (5) {$
    \scriptsize
    \begin{tikzpicture}[baseline=2mm]
    \coordinate(x) at(0:0.25);
    \coordinate(y) at(90:0.25);
    \node at($0*(x)+0*(y)$) {$P_2[1]$};
    \end{tikzpicture}$}; 
    \node[fill=white, inner sep = 0.3mm] at (6) {$
    \scriptsize
    \begin{tikzpicture}[baseline=2mm]
    \coordinate(x) at(0:0.25);
    \coordinate(y) at(90:0.25);
    \node at($0*(x)+0*(y)$) {$P_3[1]$};
    \end{tikzpicture}$}; 
    \node[fill=white, inner sep = 0.5mm] at (7) {$
    \scriptsize
    \begin{tikzpicture}[baseline=2mm]
    \coordinate(x) at(0:0.25);
    \coordinate(y) at(90:0.25);
    \node at($0*(x)+0*(y)$) {$3$}; 
    \end{tikzpicture}$}; 
    \node[fill=white, inner sep = 0.5mm] at (8) {$
    \scriptsize
    \begin{tikzpicture}[baseline=2mm]
    \coordinate(x) at(0:0.25);
    \coordinate(y) at(90:0.25);
    \node at($0*(x)+0*(y)$) {$1$}; 
    \end{tikzpicture}$}; 
    \node[fill=white, inner sep = 0.5mm] at (9) {$
    \scriptsize
    \begin{tikzpicture}[baseline=2mm]
    \coordinate(x) at(0:0.25);
    \coordinate(y) at(90:0.25);
    \node at($0*(x)+0*(y)$) {$2$}; 
    \end{tikzpicture}$}; 
    \node[fill=white, inner sep = 0.5mm] at (10) {$
    \scriptsize
    \begin{tikzpicture}[baseline=2mm]
    \coordinate(x) at(0:0.25);
    \coordinate(y) at(90:0.25);
    \node at($0*(x)+2*(y)$) {$3$}; 
    \node at($0*(x)+1*(y)$) {$1$}; 
    \end{tikzpicture}$}; 
    \node[fill=white, inner sep = 0.5mm] at (11) {$
    \scriptsize
    \begin{tikzpicture}[baseline=2mm]
    \coordinate(x) at(0:0.25);
    \coordinate(y) at(90:0.25);
    \node at($0*(x)+2*(y)$) {$2$}; 
    \node at($0*(x)+1*(y)$) {$3$}; 
    \end{tikzpicture}$}; 
    \node[fill=white, inner sep = 0.5mm] at (12) {$
    \scriptsize
    \begin{tikzpicture}[baseline=2mm]
    \coordinate(x) at(0:0.25);
    \coordinate(y) at(90:0.25);
    \node at($0*(x)+2*(y)$) {$1$}; 
    \node at($0*(x)+1*(y)$) {$2$}; 
    \end{tikzpicture}$}; 
\end{tikzpicture}
$
    \end{tabular}
    \caption{The $g$-simplicial complexes in Example \ref{example:BTA}.}
    \label{fig:BTA}
\end{figure}

\subsection{A connection between $d$-polynomials and 
$f$-polynomials}
In this subsection, we study fundamental properties of $d$-polynomials of finite dimensional algebras.

Firstly, we study the case when $A$ is a product $A=B\times C$ of two finite dimensional algebras $B$ and $C$ which are $g$-finite. 
In this case, there is a natural bijection 
\begin{equation}\label{eq:product alg}
\taurigidpair B \times \taurigidpair C\xrightarrow{\sim} \taurigidpair A    
\end{equation}
given by taking direct sum of $\tau$-rigid pairs, regarded as $A$-modules. 
It induces an isomorphism $\Delta(A)\cong \Delta(B)\times \Delta(C)$ of simplicial complexes, and hence 
\begin{equation} \label{eq:f-product}    
f(A;t) = f(B;t)f(C;t).
\end{equation}

We have the following result.

\begin{prop}\label{prop:product and d}
    Suppose that $A=B\times C$ is a product of two finite dimensional algebras $B$ and $C$ which are $g$-finite. Then, the following equation holds.  
    \begin{equation}
        d(A;t) = d(B;t)f(C;t) + f(B;t)d(C;t).
    \end{equation}
\end{prop}

\begin{proof}
For simplicity, let $\mathcal{B}:=\taurigidpair B$ and $\mathcal{C} := \taurigidpair C$. 
In addition, for a given $\tau$-rigid pair $a := (M,P)$, 
we write $|a| := |M|+|P|$ and $\alpha(a) := \dim_k M$.
Using the bijection \eqref{eq:product alg}, we have 
\begin{eqnarray*}
    d(A;t) &=& \sum_{(b,c) \in \mathcal{B}\times \mathcal{C}}(\alpha(b) + \alpha(c)) 
    t^{|A|-|b|-|c|} \\ 
    &=& \left(\sum_{(b,c)\in \mathcal{B}\times \mathcal{C}} \alpha(b) t^{|A|-|b|-|c|}\right) + 
    \left(\sum_{(b,c)\in \mathcal{B}\times \mathcal{C}} \alpha(c) t^{|A|-|b|-|c|}\right) \\ 
    &=& \left(\sum_{b\in \mathcal{B}} \alpha(b) t^{|B|-|b|}\right)\left(\sum_{c\in \mathcal{C}}  t^{|C|-|c|} \right) + 
    \left(\sum_{b\in \mathcal{B}} t^{|B|-|b|}\right)\left(\sum_{c\in \mathcal{C}}  \alpha(c) t^{|C|-|c|} \right) \\
    &=& d(B;t)f(C;t) + f(B;t)d(C;t)
\end{eqnarray*}
as desired, where we use that $|A| = |B| + |C|$ by definition. 
\end{proof}

\begin{exam}\label{example:product and d}
Applying \eqref{eq:f-product} and Proposition \ref{prop:product and d}, we can deduce that a product $A^m = A\times \cdots \times A$ ($m$ times) has the $d$-polynomial of the form 
$$d(A^m; t) = md(A;t)f(A;t)^{m-1}.$$
For example, we calculate that $d(k^m;t) = m(t+2)^{m-1}$ by $d(k;t) = 1$ and $f(k;t)=t+2$. 
\end{exam}

Next, we use the reduction method from Theorem \ref{reduction} to study $d$-polynomials. 
For a given basic $\tau$-rigid pair $(M,P)$ for $A$, we define
\begin{equation*}
    \link _{\Delta(A)}(M,P), 
\end{equation*}
called the \emph{link} of $(M,P)$, to be the subset of $\Delta(A)$ consisting of all pairs $(N,Q)$ such that $(M\oplus N, P\oplus Q)$ is a basic $\tau$-rigid pair (Hence, $M$ and $N$ (respectively, $P$ and $Q$) have no common isomorphic direct summands).  
We naturally regard it as a simplicial complex. 
If $\Delta(A)$ is clear, then we omit it. 
In addition, let $\link(M) := \link_{\Delta(A)}(M,0)$ for a given basic $\tau$-rigid $A$-module $M$.

\begin{prop}\label{prop:reduction f}
For a given basic $\tau$-rigid $A$-module $M$, we have 
\[\link (M) \cong \Delta(C_{M}),\]
where $C_{M}$ is the algebra obtained by the reduction at $M$. 
\end{prop}

\begin{proof}
It follows from the isomorphism of fans given by \cite[Theorem 4.9]{AHIKM22}. 
\end{proof}

Let $M$ be an indecomposable $\tau$-rigid $A$-module. By Proposition \ref{prop:reduction f} and as $|C_M|=|A|-|M|=n-1$, we have 
\begin{equation*}
    f(\link(M); t) = f(C_{M}; t). 
\end{equation*} 
More explicitly, the corresponding $f$-vector $(f^M_{-1},\ldots, f^M_{n-2})$ of $\link(M)$ is given by  
\begin{equation}\label{eq:fMj}
    f_j^M = \#\{(N,Q)\in \taurigidpair^{j+2} A\mid \text{$(M,0)$ is a direct summand of $(N,Q)$}\} 
\end{equation}
for all $-1\leq j \leq n-2$.

There is a connection between the $d$-polynomial and $f$-polynomial via links as follows. 
\begin{prop}\label{prop:d and f}
We have 
\begin{equation}\label{eq:d and f}
    d(A; t) = \sum_{M\in \itaurigid A} \dim_kM \cdot  f(\link(M);t). 
\end{equation}
\end{prop}

\begin{proof}
For each $M\in \itaurigid A$, we denote by $(f_{-1}^M,\ldots,f_{n-2}^M)$ the $f$-vector of $\link(M)$. 
Since $\dim_k$ is additive (i.e., $\dim_k(M\oplus N) = \dim_k M + \dim_k N$), we see using \eqref{eq:fMj} that we have
\begin{eqnarray*}
    d_j &=& \sum_{M\in \itaurigid A} \dim_kM \cdot f^M_{j-1}  
\end{eqnarray*}  
for all $0\leq j \leq n-1$. 
They imply the desired equation \eqref{eq:d and f}.
\end{proof}

The above result asserts that the $d$-polynomial $d(A;t)$ is the $f$-vector of a disjoint union of $g$-simplicial complexes. 
From this fact, we have the following properties. 
Now, we say that a polynomial $p(t) = \sum_{i=0}^n a_i t^i$ is \emph{palindromic} if $a_i = a_{n-i}$ for all $i$, and \emph{unimodal} if there exists an integer $i$ such that 
$a_1\leq a_2 \leq \cdots \leq a_i \geq \cdots \geq a_n$. 

\begin{cor}\label{palindromic}
    The polynomial $d(A;t-1)$ is palindromic and unimodal of degree $n-1$. 
\end{cor}

\begin{proof}
It was shown in \cite[Theorem 1.2 and 1.3]{AHIKM22} that the $h$-polynomial $h(A;t)$ is a palindromic and unimodal polynomial of degree $n$. 

By Proposition \ref{prop:d and f}, 
one can write 
\begin{equation}\label{d;x+1}
    d(A;t-1) = \sum_{M\in \itaurigid A} \dim_{k} M \cdot h(C_{M};t),   
\end{equation}
where $C_{M}$ is an algebra such that $|C_M| = n-1$. 
In particular, $h(C_M;t)$ is a palindromic and unimodal polynomial of degree $n-1$ for all $M \in \itaurigid A$. 
Then, \eqref{d;x+1} should be palindromic and unimodal since it is the sum of palindromic and unimodal polynomials of degree $n-1$. 
\end{proof}

Finally, we give the following definition to simplify our computation of $d$-polynomials.

\begin{defi}\label{assumption}
We say that an equivalence relation $\sim$ on the set $\itaurigid A$ is \emph{compatible with links} if $M\sim N$ implies $\link (M)\cong \link (N)$ for all $M,N\in\itaurigid A$. The equivalence class of $M$ under $\sim$ will be denoted by $[M]$. 
\end{defi}

We notice that 
\begin{equation*} 
    \link(M) \cong \link(N) \Rightarrow 
    f(\link(M);t) = f(\link(N);t)  
\end{equation*}
holds for two indecomposable $\tau$-rigid modules $M$ and $N$. 
However, the converse does not hold in general (see Example \ref{example:BTA} and (\ref{h=f-1})).

In the setting of Definition \ref{assumption}, we have the following result. Now, for a given finite set $\mathcal{S}$ of isomorphism classes of $A$-modules, let  
$$
\Dim(\mathcal{S}) := \sum_{M\in \mathcal{S}} \dim_k M. 
$$

\begin{prop}\label{prop:orbit decomposition}
Let $U_1,\ldots,U_r$ be complete representatives of $\itaurigid A$ under an equivalence relation $\sim$ which is compatible with links. Then, we have 
    \begin{equation}\label{eq:decomp dA}
        d(A;t) = \sum_{\ell=1}^r\Dim ([U_{\ell}]) f(\link(U_{\ell});t).
    \end{equation}
    In particular, we have 
\begin{equation*}
         \Dim(\itaurigid A) = \sum_{\ell=1}^r\Dim( [U_\ell]) \quad \text{and}\quad 
         \Dim(\stautilt A) = \sum_{\ell=1}^r \Dim([U_\ell]) \cdot \#\stautilt_{U_{\ell}} A. 
     \end{equation*}
\end{prop}

\begin{proof}
    By Proposition \ref{prop:d and f}, we have 
    \begin{eqnarray*}
        d(A;t) &=& \sum_{M\in \itaurigid A} \dim_kM\cdot f(\link(M);t) \\ 
        &=& \sum_{\ell = 1}^r \left\{\left(\sum_{M\in [U_{\ell}]} \dim_kM \right) f(\link(U_{\ell});t) \right\} = \sum_{\ell=1}^r \Dim ([U_{\ell}]) f(\link(U_\ell);t) 
    \end{eqnarray*} 
    as desired. 
    The latter assertion is immediate from the former one (cf. (\ref{d_0 and d_n-1}) and \eqref{st:d_0 and d_n-1})). 
\end{proof}

\section{$d$-polynomials of preprojective algebras}\label{section ppalg}
Let $Q$ be a Dynkin quiver with $n$ vertices. 
Let $W=W(Q):=\langle s_i \mid i\in Q_0\rangle$ be the Coxeter group of the underlying graph of $Q$ and $\Delta(W)$ the Coxeter complex associated with $W$. 

We first recall a description of the $h$-vector of $\Delta(W)$. 
For an element $w\in W$, let 
\begin{equation*}
    {\rm Des}(w) := \{s_i \mid l(w) > l(s_iw)\} \quad \text{and} \quad {\rm des}(w):= \#{\rm Des}(w), 
\end{equation*}
 where $l(w)$ is the length of $w$ (see Section \ref{sec:ppa}). Then, we call 
\begin{equation*}
    \left<\begin{matrix}
        W\\ j
    \end{matrix}\right> := \#\{w\in W \mid {\rm des}(w) = j\} \quad \text{$(0\leq j \leq n)$}
\end{equation*}
the \emph{$W$-Eulerian numbers}. The \emph{$W$-Eulerian polynomial} is defined to be a generating function for descents:
\begin{equation*}
    \Eul(Q;t) = \Eul(W;t) := \sum_{w\in W} t^{{\rm des}(w)} = \sum_{j=0}^n 
    \left<\begin{matrix}
        W\\ j
    \end{matrix}\right>
    t^j.
\end{equation*}
For simplicity, we will denote the corresponding $W$-Eulerian polynomials of Dynkin diagrams 
$\mathbb{A}_n$, $\mathbb{D}_n$ and $\mathbb{E}_n$ by 
$\Eul(\mathbb{A}_n;t)$, $\Eul(\mathbb{D}_n;t)$ and 
$\Eul(\mathbb{E}_n;t)$, respectively.

\begin{thm}{\rm (see \cite[Theorem 11.3]{Petersen15} for example)}\label{thm:W-E is h-vector}
The $h$-polynomial of $\Delta(W(Q))$ coincides with the $W$-Eulerian polynomial $\Eul(Q;t)$. 
\end{thm}

It also gives the $h$-vector of the preprojective algebra of $Q$ as follows. 

\begin{thm}\cite[Theorem 7.4]{AHIKM22}\label{thm:h of ppalg}
    We have an isomorphism $\Delta(\Pi(Q))\cong \Delta(W(Q))$ of simplicial complexes. 
    In particular, we have $h(\Pi(Q);t)=\Eul(Q;t)$ and $f(\Pi(Q);t) = \Eul(Q;t+1)$. 
\end{thm}

We refer to \cite[Section 11.4]{Petersen15} for tables of $W$-Eulerian numbers for low rank. 
In addition, for a given disjoint union $Q=Q_1\sqcup \cdots \sqcup Q_{m}$ of Dynkin quivers, we define its $W$-Eulerian polynomial by a product 
\begin{equation}\label{disjoint}
\Eul(Q;t):=\prod_{i=1}^m \Eul(Q_i;t).
\end{equation} 

\subsection{Reduction at $\tau$-rigid submodules of $e_{\ell}\Pi$}\label{sec:reduction ppalg}

Let $Q$ be a Dynkin quiver with $n$ vertices. We fix some notations. 
Let $W = W(Q) := \langle s_i \mid i\in Q_0\rangle$ be the Coxeter group of the underlying graph of $Q$. 
For a vertex $\ell \in Q_0$, let $\I:=Q_0 \setminus \{\ell\}$. 
We denote by $Q_{\I}$ the quiver obtained from $Q$ by removing the vertex $\ell$ and  
by $W_{\I}$ the subgroup of $W$ generated by the set $\{s_j\mid j \in \I\}$. Thus, $W_{\I} = W(Q_{\I})$. 
Notice that $Q_{\I}$ is not necessarily connected.

For our computation of the $d$-polynomial of $\Pi:=\Pi(Q)$, we use the following equivalence relation provided by Theorem \ref{tau weyl}. 

\begin{defi}
We write $M\sim_{\rm s} N$ if $M\subset e_{\ell} \Pi$ and $N\subset e_{\ell}\Pi$ for some $\ell \in Q_0$, where $e_{\ell}$ is the primitive idempotent of $\Pi$ corresponding to $\ell$. 
Then, it defines an equivalence relation on the set $\itaurigid \Pi$.  
\end{defi}
 
Under the above relation $\sim_{\rm s}$, by Theorem \ref{tau weyl}, 
a set $\{e_{\ell} \Pi \mid \ell \in Q_0\}$ forms a set of complete representatives, and their equivalence classes are given by 
$$
[e_{\ell}\Pi]_{\rm s} := \{X\in \itaurigid \Pi\mid X\subset e_{\ell} \Pi\}. 
$$ 

We have the following result. 

\begin{thm}\label{thm:orbit ppalg}
    Let $\Pi=\Pi(Q)$ be the preprojective algebra of a Dynkin quiver $Q$. Then, the following statements hold. 
    \begin{enumerate}[\rm (1)]
        \item For any $\ell\in Q_0$ and $X\in \itaurigid \Pi$ such that $X \subset e_{\ell}\Pi$, we have 
        \begin{equation*}\label{fell ppa}
        \link (X) \cong \Delta(W(Q_{\I})).  
\end{equation*}
In particular, the equivalence relation $\sim_{\rm s}$ on $\itaurigid \Pi$ is compatible with links.

    \item We have 
    \begin{equation}\nonumber 
    d(\Pi;t) = \sum_{\ell\in Q_0} 
    \Dim ([e_\ell\Pi]_{\rm s}) \Eul(Q_{\I};t+1). 
    \end{equation}
    \item We have 
\begin{equation}\nonumber
    \Dim(\itaurigid \Pi) = \sum_{\ell\in Q_0} \Dim( [e_{\ell}\Pi]_{\rm s}) \quad \text{and} \quad \Dim(\stautilt \Pi) = \sum_{\ell\in Q_0} \Dim([e_\ell \Pi]_{\rm s}) \cdot \# W(Q_{\I}). 
    \end{equation}
\end{enumerate}
\end{thm}

In the rest of this subsection, we prove Theorem \ref{thm:orbit ppalg}. A key observation is the following poset anti-isomorphism.

\begin{prop}\label{prop: poset WI}
Let $X\in\itaurigid \Pi$ and assume $X\subset e_\ell\Pi$.
Then, 
we have the following anti-isomorphism of posets and isomorphism of simplicial complexes 
$$\stautilt_X\Pi\cong W_{\I}\quad \textnormal{and}\quad 
\link(X)\cong \Delta(W_{\I}).$$
\end{prop}

For a proof, we need the following results.

\begin{lemm}\label{interval}
Let $A$ be a finite dimensional algebra which is $g$-finite and $|A|=n$ (hence $\tors A=\ftors A$ by Definition-Theorem \ref{poset iso}). 
Let $\Hasse(\tors A)$ be the Hasse quiver of the poset $\tors A$. 
For $X\in\itaurigid A$, the following statements hold. 

\begin{enumerate}[\rm (a)]
\item We have a poset isomorphism 
$$\stautilt_X A\cong[\Fac X,{}^\perp(\tau X)],$$
where $[\Fac X,{}^\perp(\tau X)]$ denotes the interval between $\Fac X$ and ${}^\perp(\tau X)$ in $\tors A$.
\item
We have  
$$
{}^\perp(\tau X)=\bigvee\{\VV\in\tors A \mid 
\text{$\exists \VV\to \Fac X$ in $\Hasse(\tors A)$} 
\}.
$$

\item There exists a unique $T\in\stautilt A$ such that 
$\Fac X =\Fac T$ and $X$ is a direct summand of $T$. 
Moreover, there is an irreducible left mutation $\mu_Y^L(T)$ of $T$ if and only if $Y=X$. In particular, 
there exists precisely $1$ arrow starting at $\Fac X$ and 
$n-1$ arrows ending at $\Fac X$ in $\Hasse(\tors A)$.

\end{enumerate}
\end{lemm}

\begin{proof}
(a) This follows from \cite[Proposition 2.9]{AIR14}. 

(b) Let $C := C_{X}$ be the algebra obtained by reduction at $X$. 
By \cite[Theorem 1.5]{Jasso15}, we have a poset isomorphism 
\[
[\Fac X,{}^\perp(\tau X)]\cong \tors C,
\]
and $\Fac C = \mod C$ is a unique maximal element in the poset $\tors C$. 
Thus, we get (b). 

(c) The statement follows from \cite[Lemma 3.6, 3.8]{AHIKM22} and the fact that the Hasse quiver is $n$-regular.
\end{proof}

We also recall the following well-known fact.

\begin{prop}\label{subgroup generated by s_i}\cite[Lemma 3.2.4]{BB05}
Let $w\in W$ and $\ell\in Q_0$. Assume that we have $s_j w<w$ for all $j\in\I$ and $s_\ell w>w$. Then, we have
$$\bigwedge\{s_jw\ |\ j\in\I \}=w_0(\I)w,$$ 
where $w_0(\I)$ is the the longest element of $W_{\I}$. 
\end{prop}

Now, we give a proof of Proposition \ref{prop: poset WI}. 

\begin{proof}[Proof of Proposition \ref{prop: poset WI}]
By Lemma \ref{interval}(a), we have 
$$\stautilt_X\Pi\cong[\Fac X,{}^\perp(\tau X)].$$
On the other hand, by Theorem \ref{tau weyl}, there exist elements $w,v\in W$ such that $\Fac I_w=\Fac X$, $\Fac I_v={}^\perp(\tau X)$
and an anti-isomorphism 
\begin{equation}\label{eq:anti-isom vw}
[\Fac X,{}^\perp(\tau X)]\cong [v,w]. 
\end{equation}

We will show that $$v=\bigwedge\{s_jw\ |\ j\in\I \}.$$

By our assumption and Theorem \ref{tau weyl}, we have $X=e_\ell I_x$ for some $x\in W$. 
Since $I_w=e_1I_w\oplus\cdots\oplus e_nI_w$ and $X$ is a direct summand of $I_w$, we have $X=e_\ell I_w$. 
Thus, Lemma \ref{interval}(c) implies that 
there is a unique left mutation $\mu_{e_\ell I_w}^L(I_w)$ of $I_w$, which is isomorphic to $I_\ell I_w$ by \cite[Lemma 2.1]{Mizuno14}. 
By Lemma \ref{interval}(b) and (c), we have 
\begin{eqnarray*} 
{}^\perp(\tau X)&=&\bigvee\{\VV\in\tors \Pi \mid 
\text{$\exists \VV\to \Fac I_w$ in $\Hasse(\tors \Pi)$} 
\}\\
&=&\bigvee\{\Fac(\mu_{e_j I_w}^R (I_w)) \mid j\in\I \}\\
&=&\bigvee\{\Fac(I_{s_jw}) \mid j\in\I \}\\
&=&\Fac\left(\bigvee \{I_{s_jw}\mid {j \in \I}\}\right), 
\end{eqnarray*}
where $\mu_{e_j I_w}^R$ denotes the right mutation of $I_w$. Since ${}^\perp(\tau X)=\Fac(I_v)$ and 
the map $\stautilt\Pi\to\tors\Pi$ given by $I_w\mapsto \Fac I_w$ is an isomorphism, we have $$
I_v=\bigvee \{I_{s_jw} \mid j\in\I \}. 
$$ 
Thus, the anti-isomorphism of Theorem \ref{tau weyl}  gives 
$$v=\bigwedge\{s_jw\ |\ j\in\I \}.$$
Applying Theorem \ref{subgroup generated by s_i}, we have 
$v=w_0(\I)w.$
Moreover, by \cite[Proposition 3.1.6]{BB05}, we have an isomorphism 
$$[v,w]=[w_0(\I)w,w]\cong[\id,w(w_0(\I)w)^{-1}]\cong W_{\I}.$$
Thus, we get the first statement. Moreover, since the isomorphism gives the isomorphism of fans \cite{AHIKM22}, we also obtain the second isomorphism.  
\end{proof}

Using this result, we show Theorem \ref{thm:orbit ppalg}. 

\begin{proof}[Proof of Theorem \ref{thm:orbit ppalg}]
(1) 
This is Proposition \ref{prop: poset WI}. 
(2) and (3) follow from (1) and Proposition \ref{prop:orbit decomposition}. 
\end{proof}

In the rest of this section, we give explicit formulas of 
$d$-polynomials for preprojective algebras of Dynkin types $\mathbb{A}$, $\mathbb{D}$ and $\mathbb{E}$. 
Our computation is based on Theorem \ref{thm:orbit ppalg}. 
That is, we compute 
$$
\Dim([e_{\ell}\Pi]_{\rm s}) \quad \text{and} \quad \Eul(Q_{\I};t+1)
$$ 
for each vertex $\ell$.

For our convenience, we label the vertices of Dynkin graphs of types $\mathbb{A}_n$ ($n\geq 1$), $\mathbb{D}_n$ ($n\geq 4$) and $\mathbb{E}_n$ ($n=6,7,8$) as follows: 
    \begin{eqnarray*}
    &&\begin{tikzpicture}
    \node (a) at(-2.5,0) {$\mathbb{A}_n$: }; 
    \node (a) at(-1,0) {$1$}; 
    \node (b) at(0,0) {$2$};
    \node (c) at(1,0) {$3$}; 
    \node (d) at(2,0) {}; 
    \node (x) at(2.5,0) {$\cdots$};
    \node (e) at(3,0) {}; 
    \node (f) at(4,0) {$n.$}; 
    \draw[-] (a)--(b); 
    \draw[-] (b)--(c)--(d); 
    \draw[-] (e)--(f); 
    \end{tikzpicture}\\
    &&\begin{tikzpicture}
    \node (a) at(-2.5,0) {$\mathbb{D}_n$: };
    \node (a) at(-1,0.7) {$1$}; 
    \node (aa) at(-1,-0.7) {$-1$}; 
    \node (b) at(0,0) {$2$};
    \node (c) at(1,0) {$3$}; 
    \node (d) at(2,0) {}; 
    \node (x) at(2.5,0) {$\cdots$};
    \node (e) at(3,0) {}; 
    \node (f) at(4.3,0) {$(n-1)$.}; 
    \draw[-] (a)--(b)--(aa); 
    \draw[-] (b)--(c)--(d); 
    \draw[-] (e)--(f); 
    \end{tikzpicture} \\ 
&&\begin{tikzpicture}
    \node (a) at(-2.5,0) {$\mathbb{E}_n$: };
    \node (a) at(-1,0) {$1$}; 
    \node (b) at(0,0) {$2$};
    \node (c) at(1,0) {$3$}; 
    \node (g) at(1,1) {$4$}; 
    \node (d) at(2,0) {$5$};
    \node (r) at(3,0) {};
    
    \node (x) at(3.5,0) {$\cdots$};
    \node (e) at(4,0) {}; 
    \node (f) at(5,0) {$n.$}; 
    \draw[-] (a)--(b); 
    \draw[-] (b)--(c)--(d)--(r); 
    \draw[-] (e)--(f); 
    \draw[-] (c)--(g); 
    \end{tikzpicture} 
\end{eqnarray*}

\subsection{Type $\mathbb{A}$} \label{sec:ppaA}
Let $Q$ be a Dynkin quiver of type $\mathbb{A}_n$ and $\Pi:=\Pi(Q)$ the preprojective algebra of $Q$. 
For $\ell\in Q_0=\{1,\ldots,n\}$, the representation of $e_\ell\Pi$ is given as follows.

\begin{equation}\label{proj:ppaA}    
\begin{split}
\begin{tikzpicture}
    \coordinate(x) at(1.8,0);
    \coordinate(y) at(0,0.9);
    \node(00) at ($0*(x)+0*(y)$) {\tiny$n$}; 
    \node(10) at ($1*(x)+0*(y)$) {\tiny$n-1$};
    \node(20) at ($2*(x)+0*(y)$) {\tiny$n-2$}; 
    \node(30) at ($3*(x)+0*(y)$) {\tiny$\cdots$};
    \node(40) at ($4*(x)+0*(y)$) {\tiny$n-\ell+3$};
    \node(50) at ($5*(x)+0*(y)$) {\tiny$n-\ell+2$}; 
    \node(60) at ($6*(x)+0*(y)$) {\tiny$n-\ell+1$};
    \node(01) at ($0*(x)+1*(y)$) {\tiny$n-1$}; 
    \node(11) at ($1*(x)+1*(y)$) {\tiny$n-2$};
    \node(21) at ($2*(x)+1*(y)$) {\tiny$n-3$}; 
    \node(31) at ($3*(x)+1*(y)$) {\tiny$\cdots$};
    \node(41) at ($4*(x)+1*(y)$) {\tiny$n-\ell+2$};
    \node(51) at ($5*(x)+1*(y)$) {\tiny$n-\ell+1$}; 
    \node(61) at ($6*(x)+1*(y)$) {\tiny$n-\ell$};
    \node(02) at ($0*(x)+2*(y)$) {\tiny\rotatebox{90}{$\cdots$}}; 
    \node(12) at ($1*(x)+2*(y)$) {\tiny\rotatebox{90}{$\cdots$}};
    \node(22) at ($2*(x)+2*(y)$) {\tiny\rotatebox{90}{$\cdots$}}; 
    \node(32) at ($3*(x)+2*(y)$) {\tiny$\cdots$};
    \node(42) at ($4*(x)+2*(y)$) {\tiny\rotatebox{90}{$\cdots$}};
    \node(52) at ($5*(x)+2*(y)$) {\tiny\rotatebox{90}{$\cdots$}}; 
    \node(62) at ($6*(x)+2*(y)$) {\tiny\rotatebox{90}{$\cdots$}};
    \node(03) at ($0*(x)+3*(y)$) {\tiny$\ell+2$}; 
    \node(13) at ($1*(x)+3*(y)$) {\tiny$\ell+1$};
    \node(23) at ($2*(x)+3*(y)$) {\tiny$\ell$}; 
    \node(33) at ($3*(x)+3*(y)$) {\tiny$\cdots$};
    \node(43) at ($4*(x)+3*(y)$) {\tiny$5$};
    \node(53) at ($5*(x)+3*(y)$) {\tiny$4$}; 
    \node(63) at ($6*(x)+3*(y)$) {\tiny$3$};
    \node(04) at ($0*(x)+4*(y)$) {\tiny$\ell+1$}; 
    \node(14) at ($1*(x)+4*(y)$) {\tiny$\ell$};
    \node(24) at ($2*(x)+4*(y)$) {\tiny$\ell-1$}; 
    \node(34) at ($3*(x)+4*(y)$) {\tiny$\cdots$};
    \node(44) at ($4*(x)+4*(y)$) {\tiny$4$};
    \node(54) at ($5*(x)+4*(y)$) {\tiny$3$}; 
    \node(64) at ($6*(x)+4*(y)$) {\tiny$2$};
    \node(05) at ($0*(x)+5*(y)$) {\tiny$\ell$}; 
    \node(15) at ($1*(x)+5*(y)$) {\tiny$\ell-1$};
    \node(25) at ($2*(x)+5*(y)$) {\tiny$\ell-2$}; 
    \node(35) at ($3*(x)+5*(y)$) {\tiny$\cdots$};
    \node(45) at ($4*(x)+5*(y)$) {\tiny$3$};
    \node(55) at ($5*(x)+5*(y)$) {\tiny$2$}; 
    \node(65) at ($6*(x)+5*(y)$) {\tiny$1$};
    \draw[->] (00)--(10); 
    \draw[->] (10)--(20); 
    \draw[->] (20)--(30);
    \draw[->] (30)--(40);
    \draw[->] (40)--(50);
    \draw[->] (50)--(60);
    \draw[->] (01)--(11);
    \draw[->] (11)--(21);
    \draw[->] (21)--(31);
    \draw[->] (31)--(41);
    \draw[->] (41)--(51);
    \draw[->] (51)--(61);
    \draw[->] (03)--(13);
    \draw[->] (13)--(23);
    \draw[->] (23)--(33);
    \draw[->] (33)--(43);
    \draw[->] (43)--(53);
    \draw[->] (53)--(63);
    \draw[->] (04)--(14);
    \draw[->] (14)--(24);
    \draw[->] (24)--(34);
    \draw[->] (34)--(44);
    \draw[->] (44)--(54);
    \draw[->] (54)--(64);
    \draw[->] (05)--(15);
    \draw[->] (15)--(25);
    \draw[->] (25)--(35);
    \draw[->] (35)--(45);
    \draw[->] (45)--(55);
    \draw[->] (55)--(65);
    \draw[<-] (00)--(01);
    \draw[<-] (01)--(02);
    \draw[<-] (02)--(03);
    \draw[<-] (03)--(04);
    \draw[<-] (04)--(05);
    \draw[<-] (10)--(11);
    \draw[<-] (11)--(12);
    \draw[<-] (12)--(13);
    \draw[<-] (13)--(14);
    \draw[<-] (14)--(15);
    \draw[<-] (20)--(21);
    \draw[<-] (21)--(22);
    \draw[<-] (22)--(23);
    \draw[<-] (23)--(24);
    \draw[<-] (24)--(25);
    \draw[<-] (40)--(41);
    \draw[<-] (41)--(42);
    \draw[<-] (42)--(43);
    \draw[<-] (43)--(44);
    \draw[<-] (44)--(45);
    \draw[<-] (50)--(51);
    \draw[<-] (51)--(52);
    \draw[<-] (52)--(53);
    \draw[<-] (53)--(54);
    \draw[<-] (54)--(55);
    \draw[<-] (60)--(61);
    \draw[<-] (61)--(62);
    \draw[<-] (62)--(63);
    \draw[<-] (63)--(64);
    \draw[<-] (64)--(65);
\end{tikzpicture}
\end{split}
\end{equation}

Here, each number $i$ shows a one-dimensional $k$-vector space $k$ lying on the vertex $i$, and each arrow is the identity map of $k$ (see \cite[Section 6.1]{IRRT18} for example). Under this description, submodules of $e_\ell\Pi$ correspond bijectively to sub-quivers that are closed under successors. 
For our convenience, we describe them in terms of lattice paths. 
\begin{defi}
A \emph{lattice path} is a path in $\mathbb{Z}^2$ that takes only steps East, from $(i, j)$ to $(i+1, j)$, and North, from $(i, j)$ to $(i, j+1)$. 
\end{defi}

For two positive integers $s,t>0$, we denote by $\mathbb{L}(s,t)$ the set of lattice paths $\sigma$ from $(0, 0)$ to $(s, t)$. 
By definition, any lattice path $\sigma$ in $\mathbb{L}(s,t)$ has length $s+t$ and every integer vector appearing in $\sigma$ lies in the rectangle $\mathcal{S} := [0,s]\times[0,t]\subset \mathbb{Z}^2$. 
Then, we denote by $\area(\sigma)$ the number of unit squares in $\mathcal{S}$ underneath the path $\sigma$. 
Here, we say that a \emph{unit square} is a subset of $\mathbb{R}^2$ of the form $[i,i+1]\times[j,j+1]$ with $i,j\in \mathbb{Z}$.

Now, we fix a vertex $\ell\in Q_0=\{1,\ldots,n\}$. 
We consider the set of lattice paths $\mathbb{L}(\ell,n-\ell+1)$ and the corresponding rectangle $\mathcal{S} = [0,\ell]\times [0,n-\ell+1]\subset \mathbb{R}^2$. 
We associate each unit square $[i,i+1]\times[j,j+1]\subset \mathcal{S}$ to the vertex of \eqref{proj:ppaA} located at $i$-th column from the left and $j$-th row from the bottom. For example, the unit square $[0,1]\times[0,1]$ corresponds to the lower left corner labeled by $n$. 
Under this correspondence, to each lattice path $\sigma\in \mathbb{L}(\ell,n-\ell+1)$, we associate $M_{\sigma}\subset e_\ell\Pi$ as the submodule given by the sub-quiver of \eqref{proj:ppaA} consisting of all vertices  underneath $\sigma$, which is closed under successors by definition. We display below some lattice paths $\sigma\in \mathbb{L}(4,3)$ together with the corresponding submodule $M_{\sigma}$ in the case of $n=6$ and $\ell = 4$. 

\begin{equation*}    
\begin{tabular}{ccccccccc}
    \begin{tikzpicture}
    \coordinate(x) at(0.6,0);
    \coordinate(y) at(0,0.6);
    \coordinate (00) at ($0*(x)+0*(y)$); 
    \coordinate (10) at ($1*(x)+0*(y)$); 
    \coordinate (20) at ($2*(x)+0*(y)$); 
    \coordinate (30) at ($3*(x)+0*(y)$); 
    \coordinate (40) at ($4*(x)+0*(y)$); 
    
    \coordinate (01) at ($0*(x)+1*(y)$); 
    \coordinate (11) at ($1*(x)+1*(y)$); 
    \coordinate (21) at ($2*(x)+1*(y)$); 
    \coordinate (31) at ($3*(x)+1*(y)$); 
    \coordinate (41) at ($4*(x)+1*(y)$); 

    \coordinate (02) at ($0*(x)+2*(y)$); 
    \coordinate (12) at ($1*(x)+2*(y)$); 
    \coordinate (22) at ($2*(x)+2*(y)$); 
    \coordinate (32) at ($3*(x)+2*(y)$); 
    \coordinate (42) at ($4*(x)+2*(y)$); 
    
    \coordinate (03) at ($0*(x)+3*(y)$); 
    \coordinate (13) at ($1*(x)+3*(y)$); 
    \coordinate (23) at ($2*(x)+3*(y)$); 
    \coordinate (33) at ($3*(x)+3*(y)$); 
    \coordinate (43) at ($4*(x)+3*(y)$); 

    \draw[dotted, opacity=0.8] (00)rectangle(43);
    \draw[dotted, opacity=0.8] (01)--(41);
    \draw[dotted, opacity=0.8] (02)--(42);
    \draw[dotted, opacity=0.8] (10)--(13);
    \draw[dotted, opacity=0.8] (20)--(23);
    \draw[dotted, opacity=0.8] (30)--(33);

    \node at($0.5*(x)+0.5*(y)$) {\footnotesize$6$};
    \node at($1.5*(x)+0.5*(y)$) {\footnotesize$5$};
    \node at($2.5*(x)+0.5*(y)$) {\footnotesize$4$};
    \node at($3.5*(x)+0.5*(y)$) {\footnotesize$3$};
    \node at($0.5*(x)+1.5*(y)$) {\footnotesize$5$};
    \node at($1.5*(x)+1.5*(y)$) {\footnotesize$4$};
    \node at($2.5*(x)+1.5*(y)$) {\footnotesize$3$};
    \node at($3.5*(x)+1.5*(y)$) {\footnotesize$2$};
    \node at($0.5*(x)+2.5*(y)$) {\footnotesize$4$};
    \node at($1.5*(x)+2.5*(y)$) {\footnotesize$3$};
    \node at($2.5*(x)+2.5*(y)$) {\footnotesize$2$};
    \node at($3.5*(x)+2.5*(y)$) {\footnotesize$1$};
    \node at(00) {\footnotesize$\bullet$};
    \node at(43) {\footnotesize$\bullet$};
    \draw[thick] (00)--(40)--(43);
    \end{tikzpicture}
    & \quad 
    \begin{tikzpicture}
    \coordinate(x) at(0.6,0);
    \coordinate(y) at(0,0.6);
    \coordinate (00) at ($0*(x)+0*(y)$); 
    \coordinate (10) at ($1*(x)+0*(y)$); 
    \coordinate (20) at ($2*(x)+0*(y)$); 
    \coordinate (30) at ($3*(x)+0*(y)$); 
    \coordinate (40) at ($4*(x)+0*(y)$); 
    
    \coordinate (01) at ($0*(x)+1*(y)$); 
    \coordinate (11) at ($1*(x)+1*(y)$); 
    \coordinate (21) at ($2*(x)+1*(y)$); 
    \coordinate (31) at ($3*(x)+1*(y)$); 
    \coordinate (41) at ($4*(x)+1*(y)$); 

    \coordinate (02) at ($0*(x)+2*(y)$); 
    \coordinate (12) at ($1*(x)+2*(y)$); 
    \coordinate (22) at ($2*(x)+2*(y)$); 
    \coordinate (32) at ($3*(x)+2*(y)$); 
    \coordinate (42) at ($4*(x)+2*(y)$); 
    
    \coordinate (03) at ($0*(x)+3*(y)$); 
    \coordinate (13) at ($1*(x)+3*(y)$); 
    \coordinate (23) at ($2*(x)+3*(y)$); 
    \coordinate (33) at ($3*(x)+3*(y)$); 
    \coordinate (43) at ($4*(x)+3*(y)$); 

    \fill[white!30!lightgray] (30)rectangle(41); 
    \draw[very thick] (00)--(30)--(31)--(41)--(43);
    
    \draw[dotted, opacity=0.8] (00)rectangle(43);
    \draw[dotted, opacity=0.8] (01)--(41);
    \draw[dotted, opacity=0.8] (02)--(42);
    \draw[dotted, opacity=0.8] (10)--(13);
    \draw[dotted, opacity=0.8] (20)--(23);
    \draw[dotted, opacity=0.8] (30)--(33);

    \node at($0.5*(x)+0.5*(y)$) {\footnotesize$6$};
    \node at($1.5*(x)+0.5*(y)$) {\footnotesize$5$};
    \node at($2.5*(x)+0.5*(y)$) {\footnotesize$4$};
    \node at($3.5*(x)+0.5*(y)$) {\footnotesize$3$};
    \node at($0.5*(x)+1.5*(y)$) {\footnotesize$5$};
    \node at($1.5*(x)+1.5*(y)$) {\footnotesize$4$};
    \node at($2.5*(x)+1.5*(y)$) {\footnotesize$3$};
    \node at($3.5*(x)+1.5*(y)$) {\footnotesize$2$};
    \node at($0.5*(x)+2.5*(y)$) {\footnotesize$4$};
    \node at($1.5*(x)+2.5*(y)$) {\footnotesize$3$};
    \node at($2.5*(x)+2.5*(y)$) {\footnotesize$2$};
    \node at($3.5*(x)+2.5*(y)$) {\footnotesize$1$};
    \node at(00) {\footnotesize$\bullet$};
    \node at(43) {\footnotesize$\bullet$};
    \end{tikzpicture} 
    & \quad 
     \begin{tikzpicture}
    \coordinate(x) at(0.6,0);
    \coordinate(y) at(0,0.6);
    \coordinate (00) at ($0*(x)+0*(y)$); 
    \coordinate (10) at ($1*(x)+0*(y)$); 
    \coordinate (20) at ($2*(x)+0*(y)$); 
    \coordinate (30) at ($3*(x)+0*(y)$); 
    \coordinate (40) at ($4*(x)+0*(y)$); 
    
    \coordinate (01) at ($0*(x)+1*(y)$); 
    \coordinate (11) at ($1*(x)+1*(y)$); 
    \coordinate (21) at ($2*(x)+1*(y)$); 
    \coordinate (31) at ($3*(x)+1*(y)$); 
    \coordinate (41) at ($4*(x)+1*(y)$); 

    \coordinate (02) at ($0*(x)+2*(y)$); 
    \coordinate (12) at ($1*(x)+2*(y)$); 
    \coordinate (22) at ($2*(x)+2*(y)$); 
    \coordinate (32) at ($3*(x)+2*(y)$); 
    \coordinate (42) at ($4*(x)+2*(y)$); 
    
    \coordinate (03) at ($0*(x)+3*(y)$); 
    \coordinate (13) at ($1*(x)+3*(y)$); 
    \coordinate (23) at ($2*(x)+3*(y)$); 
    \coordinate (33) at ($3*(x)+3*(y)$); 
    \coordinate (43) at ($4*(x)+3*(y)$); 
    
    \fill[white!30!lightgray] (00)rectangle(41);
    \fill[white!30!lightgray] (21)rectangle(42);
    \fill[white!30!lightgray] (32)rectangle(43);
    \draw[very thick] (00)--(01)--(21)--(22)--(32)--(33)--(43);
    
    \draw[dotted, opacity = 0.8] (00)rectangle(43);
    \draw[dotted, opacity = 0.8] (01)--(41);
    \draw[dotted, opacity = 0.8] (02)--(42);
    \draw[dotted, opacity = 0.8] (10)--(13);
    \draw[dotted, opacity = 0.8] (20)--(23);
    \draw[dotted, opacity = 0.8] (30)--(33);

    \node at($0.5*(x)+0.5*(y)$) {\footnotesize$6$};
    \node at($1.5*(x)+0.5*(y)$) {\footnotesize$5$};
    \node at($2.5*(x)+0.5*(y)$) {\footnotesize$4$};
    \node at($3.5*(x)+0.5*(y)$) {\footnotesize$3$};
    \node at($0.5*(x)+1.5*(y)$) {\footnotesize$5$};
    \node at($1.5*(x)+1.5*(y)$) {\footnotesize$4$};
    \node at($2.5*(x)+1.5*(y)$) {\footnotesize$3$};
    \node at($3.5*(x)+1.5*(y)$) {\footnotesize$2$};
    \node at($0.5*(x)+2.5*(y)$) {\footnotesize$4$};
    \node at($1.5*(x)+2.5*(y)$) {\footnotesize$3$};
    \node at($2.5*(x)+2.5*(y)$) {\footnotesize$2$};
    \node at($3.5*(x)+2.5*(y)$) {\footnotesize$1$};
    \node at(00) {\footnotesize$\bullet$};
    \node at(43) {\footnotesize$\bullet$};
    \end{tikzpicture} & \quad 
     \begin{tikzpicture}
    \coordinate(x) at(0.6,0);
    \coordinate(y) at(0,0.6);
    \coordinate (00) at ($0*(x)+0*(y)$); 
    \coordinate (10) at ($1*(x)+0*(y)$); 
    \coordinate (20) at ($2*(x)+0*(y)$); 
    \coordinate (30) at ($3*(x)+0*(y)$); 
    \coordinate (40) at ($4*(x)+0*(y)$); 
    
    \coordinate (01) at ($0*(x)+1*(y)$); 
    \coordinate (11) at ($1*(x)+1*(y)$); 
    \coordinate (21) at ($2*(x)+1*(y)$); 
    \coordinate (31) at ($3*(x)+1*(y)$); 
    \coordinate (41) at ($4*(x)+1*(y)$); 

    \coordinate (02) at ($0*(x)+2*(y)$); 
    \coordinate (12) at ($1*(x)+2*(y)$); 
    \coordinate (22) at ($2*(x)+2*(y)$); 
    \coordinate (32) at ($3*(x)+2*(y)$); 
    \coordinate (42) at ($4*(x)+2*(y)$); 
    
    \coordinate (03) at ($0*(x)+3*(y)$); 
    \coordinate (13) at ($1*(x)+3*(y)$); 
    \coordinate (23) at ($2*(x)+3*(y)$); 
    \coordinate (33) at ($3*(x)+3*(y)$); 
    \coordinate (43) at ($4*(x)+3*(y)$); 
    
    \fill[white!30!lightgray] (00)rectangle(43); 
    \draw[very thick] (00)--(03)--(43);
    
    \draw[dotted, opacity=0.8] (00)rectangle(43);
    \draw[dotted, opacity=0.8] (01)--(41);
    \draw[dotted, opacity=0.8] (02)--(42);
    \draw[dotted, opacity=0.8] (10)--(13);
    \draw[dotted, opacity=0.8] (20)--(23);
    \draw[dotted, opacity=0.8] (30)--(33);
    
    \node at($0.5*(x)+0.5*(y)$) {\footnotesize$6$};
    \node at($1.5*(x)+0.5*(y)$) {\footnotesize$5$};
    \node at($2.5*(x)+0.5*(y)$) {\footnotesize$4$};
    \node at($3.5*(x)+0.5*(y)$) {\footnotesize$3$};
    \node at($0.5*(x)+1.5*(y)$) {\footnotesize$5$};
    \node at($1.5*(x)+1.5*(y)$) {\footnotesize$4$};
    \node at($2.5*(x)+1.5*(y)$) {\footnotesize$3$};
    \node at($3.5*(x)+1.5*(y)$) {\footnotesize$2$};
    \node at($0.5*(x)+2.5*(y)$) {\footnotesize$4$};
    \node at($1.5*(x)+2.5*(y)$) {\footnotesize$3$};
    \node at($2.5*(x)+2.5*(y)$) {\footnotesize$2$};
    \node at($3.5*(x)+2.5*(y)$) {\footnotesize$1$};
    \node at(00) {\footnotesize$\bullet$};
    \node at(43) {\footnotesize$\bullet$};
    \end{tikzpicture} \\ 
    \footnotesize$\area(\sigma)=0$ & 
    \footnotesize$\area(\sigma)=1$&
    \footnotesize$\area(\sigma)=7$&
    \footnotesize$\area(\sigma)=12$\\
    \footnotesize$M_{\sigma}=0$ &
    \footnotesize$M_{\sigma}=3$ &
    \footnotesize$M_{\sigma}=\tiny
    \begin{tikzpicture}[baseline=2mm]
    \coordinate(x) at(135:0.25);
    \coordinate(y) at(45:0.25);
    \node at($3*(x)+0*(y)$) {$6$}; 
    \node at($2*(x)+0*(y)$) {$5$}; 
    \node at($1*(x)+0*(y)$) {$4$}; 
    \node at($0*(x)+0*(y)$) {$3$}; 
    \node at($1*(x)+1*(y)$) {$3$}; 
    \node at($0*(x)+1*(y)$) {$2$}; 
    \node at($0*(x)+2*(y)$) {$1$}; 
    \end{tikzpicture}$ &
    \footnotesize$M_{\sigma}=e_{4}\Pi$ 
\end{tabular}
\end{equation*}

\begin{prop} \cite[Theorem 6.1]{IRRT18}\label{prop:bij ppalgA}
For $\ell\in Q_0=\{1,\ldots,n\}$, every submodule of $e_\ell\Pi$ is $\tau$-rigid. Furthermore, we have a bijection 
\begin{equation*}
     \mathbb{L}(\ell,n-\ell+1) \xrightarrow{\sim} [e_\ell\Pi]_{\rm s} 
     \cup \{0\} \quad (\sigma \mapsto M_\sigma)
\end{equation*}
satisfying $\dim_kM_\sigma = \area(\sigma)$. 
In particular, we have 
\begin{equation*}
    \#[e_\ell \Pi]_{\rm s} +1 = \#\mathbb{L}(\ell,n-\ell+1) = \binom{n+1}{\ell}. 
\end{equation*}

\end{prop}

Now, for two positive integers $p \geq q$, 
let 
\begin{equation}\label{eq:T} 
    {\rm T}(p,q):=\frac{(p+1)(p+2)}{2}\binom{p}{q} = 
    \frac{(p+2)!}{2\cdot q!(p-q)!}.   
\end{equation}
The corresponding integer sequence is well-known and it can be found in \underline{A094305} ({O}nline {E}ncyclopedia of {I}nteger {S}equences \cite{OEIS}). 
Using the above bijection, we have the following. 

\begin{lemm}\label{lem:OOpathalgA}
    For $\ell\in Q_0$, we have 
    \begin{eqnarray*}
        \Dim([e_\ell\Pi]_{\rm s}) = \sum_{\sigma\in \mathbb{L}(\ell,n-\ell+1)} \area (\sigma) = {\rm T}(n-1,n-\ell).
    \end{eqnarray*}
\end{lemm}

\begin{proof}
    The left most equality follows from Proposition \ref{prop:bij ppalgA}. 
    We now show the right most equality. 
    Let ${\rm S}(n-1,n-\ell)$ be the sum of $\area(\sigma)$ of all $\sigma\in \mathbb{L}(\ell,n-\ell+1)$. 
    
    The set $\mathbb{L}(\ell,n-\ell+1)$ is a disjoint union of subsets $\mathbb{L}^{\rm E}(\ell,n-\ell+1)$ and $\mathbb{L}^{\rm N}(\ell,n-\ell+1)$ consisting of all lattice paths whose first steps are East and North, respectively.  
    By ignoring the first step of each lattice path, we have bijections 
    \begin{equation*}
        a\colon \mathbb{L}(\ell-1,n-\ell+1) \simeq \mathbb{L}^{\rm E}(\ell,n-\ell+1) 
        \quad \text{and} \quad 
        b\colon \mathbb{L}(\ell,n-\ell) \simeq \mathbb{L}^{\rm N}(\ell,n-\ell+1) 
    \end{equation*}  
    such that $\area (a(\sigma)) = \area (\sigma)$ for any $\sigma\in \mathbb{L}(\ell-1,n-\ell+1)$ and $\area(b(\sigma')) = \area (\sigma') + \ell$ for any $\sigma' \in \mathbb{L}(\ell,n-\ell)$. 
    Then, we get
    \begin{eqnarray}
        {\rm S}(n-1,n-\ell) &=& \sum_{\sigma\in \mathbb{L}^{\rm E}(\ell,n-\ell+1)} \area(\sigma) + \sum_{\sigma'\in \mathbb{L}^{\rm N}(\ell,n-\ell+1)} \area(\sigma') \nonumber\\ 
        &=&  \sum_{\sigma\in \mathbb{L}(\ell-1,n-\ell+1)} \area(\sigma) + \sum_{\sigma'\in \mathbb{L}(\ell,n-\ell)} \big(\area(\sigma') + \ell \big) \nonumber
        \\ 
        &=&{\rm S}(n-2,n-\ell-1) + {\rm S}(n-2,n-\ell) + \ell \binom{n}{\ell}. \label{eq:reccST} 
    \end{eqnarray}
    We clearly have ${\rm S}(0,0) = {\rm T}(0,0) = 1$. To prove the claim, it is enough to show that ${\rm T}(n-1,n-\ell)$ satisfies the same recurrence relation \eqref{eq:reccST}. 
    By \eqref{eq:T}, we write 
    \begin{equation*}
        {\rm T}(n-2,n-\ell-1) = \frac{n!(n-\ell)}{2(n-\ell)!(\ell-1)!} \quad \text{and} \quad {\rm T}(n-2,n-\ell) = \frac{n!(\ell-1)}{2(n-\ell)!(\ell-1)!}. 
    \end{equation*}
    Then, we obtain the desired equality by  
    \begin{equation*}
        {\rm T}(n-2,n-\ell-1) + {\rm T}(n-2,n-\ell) + \ell \binom{n}{\ell} 
        =\frac{(n+1)!}{2(n-\ell)!(\ell-1)!} = {\rm T}(n-1,n-\ell), 
    \end{equation*}
    where we use $\ell \binom{n}{\ell} = \frac{n!}{(n-\ell)!(\ell-1)!}$.
\end{proof}

We now obtain the following explicit formula.

\begin{thm}\label{thm:dpoly ppaA}
    Let $\Pi=\Pi(Q)$ be the preprojective algebra of a Dynkin quiver $Q$ of type $\mathbb{A}_n$. Then, we have 
    \begin{equation}\label{eq:dPi A}
        d(\Pi;t) = \sum_{\ell=1}^n {\rm T}(n-1,n-\ell) 
        \Eul(\mathbb{A}_{\ell-1};t+1) \Eul(\mathbb{A}_{n-\ell};t+1).  
    \end{equation}
    In particular, we have 
    \begin{eqnarray}\label{eq:Lah numbers A}
        \Dim(\itaurigid \Pi) = n(n+1)2^{n-2} \quad \text{and} \quad 
        \Dim(\stautilt\Pi) = \frac{1}{6}(n+2)!\binom{n+1}{2}.
    \end{eqnarray}
For $n \leq 9$, we describe coefficients $d_j$ $(0\leq j \leq n-1)$ of the polynomial \eqref{eq:dPi A} in Table \ref{Table:dimppaA}. 

\begin{table}[h]\scriptsize
\renewcommand{\arraystretch}{0.95}
\begin{tabular}{r| r r r r r r r r r r}
    $n\backslash\,j$ & $0$ & $1$ & $2$ & $3$ & $4$ & $5$ & $6$ & $7$ & $8$\\ \hline 
    $1$ & $1$  \\ 
    $2$ & $6$& $12$ \\
    $3$ & $24$& $120$& $120$ \\ 
    $4$ & $80$& $760$& $1800$& $1200$  \\ 
    $5$ & $240$& $3900$& $16500$& $25200$& $12600$ & \\ 
    $6$ & $672$& $17724$& $119700$& $315000$& $352800$& $141120$  \\ 
    $7$ & $1792$& $74480$& $756560$& $3057600$& $5762400$& $5080320$& $1693440$  \\ 
    $8$ & $4608$& $296496$& $4369680$& $25492320$& $71971200$& $104993280$& $76204800$& $21772800$  \\ 
    $9$ & $11520$& $1134900$& $23701500$& $192099600$& $762841800$& $1638403200$& $1943222400$& $1197504000$& $299376000$ \\
\end{tabular} 
\\  \ 
    \caption{Numbers $d_j$ for $d$-polynomials of preprojective algebras of type $\mathbb{A}$}
    \label{Table:dimppaA}
\end{table}
\end{thm}

\begin{proof}
The underlying graph of the quiver $Q_{\I}$ is a disjoint union $\mathbb{A}_{\ell-1}\sqcup\mathbb{A}_{n-\ell}$. By  (\ref{disjoint}), it implies
$$
\Eul(Q_{\I};t) = \Eul(\mathbb{A}_{\ell-1};t) \Eul(\mathbb{A}_{n-\ell};t).
$$

Using Theorem \ref{thm:orbit ppalg}(2) and Lemma \ref{lem:OOpathalgA}, 
we get the desired equation \eqref{eq:dPi A}. 
For the latter assertions, by Theorem \ref{thm:orbit ppalg}(3) and Lemma \ref{lem:OOpathalgA}, we have
$$\Dim(\itaurigid \Pi)=\sum_{\ell=1}^n\Dim([e_\ell \Pi]_{\rm s}) =
\sum_{\ell=1}^n{\rm T}(n-1,n-\ell)=n(n+1)2^{n-2}.$$

On the other hand, we have 
\begin{equation*}
    \#W(Q_{\I}) = \#W(\mathbb{A}_{\ell-1}) \# W(\mathbb{A}_{n-\ell}) = \ell!(n-\ell+1)! 
\end{equation*}
by \eqref{table:orderW}. Then, we get 
\begin{eqnarray*}
\Dim(\stautilt\Pi)=\sum_{\ell=1}^n\Dim([e_\ell\Pi]_{\rm s})\cdot \#W(Q_{\I})
=\frac{(n+1)!}{2}\sum_{\ell=1}^n \ell(n-\ell+1) 
=\frac{1}{6}(n+2)!\binom{n+1}{2}. 
\end{eqnarray*}
It finishes the proof. 
\end{proof}

We remark that the numbers in the right-hand side of \eqref{eq:Lah numbers A} for support $\tau$-tilting modules are listed as a part of \emph{Lah numbers} \underline{A001754} ({O}nline {E}ncyclopedia of {I}nteger {S}equences \cite{OEIS}). It would be interesting to give a direct connection between these two different objects.

\subsection{Type $\mathbb{D}$}
For Dynkin quivers of type $\mathbb{D}$, 
we will obtain the following.

\begin{prop}\label{prop:OOppalgD}
Let $\Pi=\Pi(Q)$ be the preprojective algebra of a Dynkin quiver $Q$ of type $\mathbb{D}_n$. 
For $\ell\in Q_0=\{\pm1,2,\ldots,n-1\}$, we have 
    \begin{equation}\label{eq:dimQell D}
    \Dim([e_{\ell}\Pi]_{\rm s}) = 
    \begin{cases}
        n(n-1)2^{n-3} & \text{if $\ell = \pm 1$}, \\
        (n-\ell)(n+\ell-1)2^{n-\ell-1}\binom{n}{\ell} & \text{if $\ell \neq \pm1$.}
    \end{cases}
\end{equation}
\end{prop}

After proving Proposition \ref{prop:OOppalgD}, we obtain the following result.

\begin{thm}
Let $\Pi=\Pi(Q)$ be the preprojective algebra of a Dynkin quiver $Q$ of type $\mathbb{D}_n$. Then, we have 
\begin{equation}\label{eq:dPi D}
\begin{array}{llllc}\displaystyle
    d(\Pi;t) &=& \displaystyle n(n-1)2^{n-2}\Eul(\mathbb{A}_{n-1};t+1) \\
    &+& \displaystyle\sum_{\ell = 2}^{n-1}(n-\ell)(n+\ell-1)2^{n-\ell-1}\binom{n}{\ell} \Eul(\mathbb{D}_{\ell};t+1) \Eul(\mathbb{A}_{n-\ell-1};t+1). 
\end{array}
\end{equation}
Here, we regard $\mathbb{D}_2 = \mathbb{A}_1\times \mathbb{A}_1$ and $\mathbb{D}_3 = \mathbb{A}_3$.
In particular, we have 
\begin{eqnarray*}
    \Dim(\itaurigid \Pi) &=& n(n-1)2^{n-2}+\sum_{\ell=2}^{n-1}(n+\ell-1)(n-\ell)2^{n-\ell-1}\binom{n}{\ell} \quad \text{and} \\ 
    \Dim(\stautilt \Pi) &=& 2n(n-1)2^{n-3}n!+
    \sum_{\ell=2}^{n-1}(n+\ell-1)(n-\ell)2^{2n-\ell-2}n!(n-\ell)!
    \binom{n}{\ell}.
\end{eqnarray*}
For $n \leq 9$, we describe coefficients $d_j$ $(0\leq j \leq n-1)$ of the polynomial \eqref{eq:dPi D} in Table \ref{Table:dimppaD}. 

\begin{table}[h]\scriptsize
\renewcommand{\arraystretch}{0.95}
\begin{tabular}{r|rrrrrrrrr}
    $n\backslash\,j$ & $0$ & $1$ & $2$ & $3$ & $4$ & $5$ & $6$ & $7$ & $8$\\ \hline 
    $4$ & $192$& $1728$& $4032$& $2688$  \\ 
    $5$ & $1200$& $18400$& $76000$& $115200$& $57600$ \\ 
    $6$ & $6360$& $165360$& $1091520$& $2839680$& $3168000$& $1267200$ \\ 
    $7$ & $30072$& $1331904$& $13374144$& $53437440$& $100101120$& $88058880$& $29352960$  \\ 
    $8$ & $131040$& $9935296$& $147721728$& $855421952$& $2399846400$& $3488808960$& $2528870400$& $722534400$ \\ 
    $9$ & $537696$& $70013952$& $1517147136$& $12300115968$& $48615727104$& $104021729280$& $123112120320$& $75804180480$& $18951045120$ \\
\end{tabular}  
\\  \ 
    \caption{Numbers $d_j$ for $d$-polynomials of  preprojective algebras of type $\mathbb{D}$}
    \label{Table:dimppaD}
\end{table}
\end{thm}

\begin{proof}
    The underlying graph of $Q_{\I}$ is $\mathbb{A}_{n-1}$ for $\ell=\pm1$ and $\mathbb{D}_{\ell}\sqcup \mathbb{A}_{n-\ell-1}$ for $\ell\neq \pm1$. 
    Thus, 
    \begin{equation}\label{eq:EulQell D}
    \Eul(Q_{\I};t) =  \begin{cases} 
    \Eul(\mathbb{A}_{n-1};t) & \text{if $\ell=\pm1$,} \\ 
    \Eul(\mathbb{D}_{\ell};t)\Eul(\mathbb{A}_{n-\ell};t) & \text{if $\ell\neq \pm1$.}
    \end{cases}
    \end{equation}
    By \eqref{eq:dimQell D} and \eqref{eq:EulQell D}, we get the desired equation \eqref{eq:dPi D}. 
    In addition, by Theorem \ref{thm:orbit ppalg}(3), we have the latter assertions since $\# W(\mathbb{A}_m) = (m+1)!$ and $\#W(\mathbb{D}_{m}) = 2^{m-1}m!$ by \eqref{table:orderW}. 
\end{proof}

In the rest, we prove Proposition \ref{prop:OOppalgD} for each case of $\ell \in Q_0=\{\pm 1,2,\ldots,n-1\}$. 

\subsubsection{Case $\ell = \pm1$} \label{subsec:Dpm1}
Let $\ell = \pm1$. 
In this case, our discussion is similar to the case of type $\mathbb{A}$ in Section \ref{sec:ppaA}. 
Firstly, the representation of $e_\ell\Pi$ is given as follows (see \cite{IRRT18}). 

\begin{equation} \label{proj:ppaDpm1}   
\begin{split}
\begin{tikzpicture}
    \coordinate(x) at(1.4,0);
    \coordinate(y) at(0,0.9);
    \node(00) at ($0*(x)+0*(y)$) {\tiny$n-1$}; 
    \node(10) at ($1*(x)+0*(y)$) {\tiny$n-2$};
    \node(20) at ($2*(x)+0*(y)$) {\tiny$n-3$}; 
    \node(30) at ($3*(x)+0*(y)$) {\tiny$\cdots$};
    \node(40) at ($4*(x)+0*(y)$) {\tiny$3$};
    \node(50) at ($5*(x)+0*(y)$) {\tiny$2$}; 
    \node(60) at ($6*(x)+0*(y)$) {\tiny$\pm(-1)^n$};
    \node(01) at ($0*(x)+1*(y)$) {\tiny$n-2$}; 
    \node(11) at ($1*(x)+1*(y)$) {\tiny$n-3$};
    \node(21) at ($2*(x)+1*(y)$) {\tiny$n-4$}; 
    \node(31) at ($3*(x)+1*(y)$) {\tiny$\cdots$};
    \node(41) at ($4*(x)+1*(y)$) {\tiny$2$};
    \node(51) at ($5*(x)+1*(y)$) {\tiny$\mp(-1)^n$}; 
    \node(02) at ($0*(x)+2*(y)$) {\tiny\rotatebox{90}{$\cdots$}}; 
    \node(12) at ($1*(x)+2*(y)$) {\tiny\rotatebox{90}{$\cdots$}};
    \node(22) at ($2*(x)+2*(y)$) {\tiny\rotatebox{90}{$\cdots$}}; 
    \node(32) at ($3*(x)+2*(y)$) {\tiny\rotatebox{135}{$\cdots$}};
    \node(03) at ($0*(x)+3*(y)$) {\tiny$3$}; 
    \node(13) at ($1*(x)+3*(y)$) {\tiny$2$};
    \node(23) at ($2*(x)+3*(y)$) {\tiny$\pm1$}; 
    \node(04) at ($0*(x)+4*(y)$) {\tiny$2$}; 
    \node(14) at ($1*(x)+4*(y)$) {\tiny$\mp1$};
    \node(05) at ($0*(x)+5*(y)$) {\tiny$\pm1$}; 
    \draw[->] (00)--(10); 
    \draw[->] (10)--(20); 
    \draw[->] (20)--(30);
    \draw[->] (30)--(40);
    \draw[->] (40)--(50);
    \draw[->] (50)--(60);
    \draw[->] (01)--(11);
    \draw[->] (11)--(21);
    \draw[->] (21)--(31);
    \draw[->] (31)--(41);
    \draw[->] (41)--(51);
    \draw[->] (03)--(13);
    \draw[->] (13)--(23);
    \draw[->] (04)--(14);
    \draw[<-] (00)--(01);
    \draw[<-] (01)--(02);
    \draw[<-] (02)--(03);
    \draw[<-] (03)--(04);
    \draw[<-] (04)--(05);
    \draw[<-] (10)--(11);
    \draw[<-] (11)--(12);
    \draw[<-] (12)--(13);
    \draw[<-] (13)--(14);
    \draw[<-] (20)--(21);
    \draw[<-] (21)--(22);
    \draw[<-] (22)--(23);
    \draw[<-] (40)--(41);
    \draw[<-] (50)--(51);
\end{tikzpicture}
\end{split}
\end{equation}

Here, each number $i$ shows a one dimensional $k$-vector space $k$ lying on the vertex $i$, and each arrow is the identity map of $k$. 
Then, submodules of $e_\ell\Pi$ correspond bijectively to sub-quivers that are closed under successors. 
We describe them by using lattice path as follows. 

Let $\mathcal{S}':=\{(x,y)\in \mathbb{R}^2 \mid x,y\geq 0, x+y \leq n\}$. 
Let $\mathbb{L}'(n-1)$ be the set of lattice paths $\sigma$ 
starting at $(0,0)$ and ending at one of $(0,n-1),(1,n-2),\ldots,(n-1,0)$. In other words, they are precisely lattice paths of length $n-1$ starting at $(0,0)$. 
Thus, the number of lattice paths in $\mathbb{L}'(n-1)$ is $2^{n-1}$. 
For $\sigma\in \mathbb{L}'(n-1)$, the area $\area(\sigma)$ is defined to be the number of unit squares in $\mathcal{S'}$ which is underneath or right to $\sigma$.

We associate each unit square $[i,i+1]\times [j,j+1] \subset \mathcal{S}'$ to the vertex of \eqref{proj:ppaDpm1} located at $i$-th column from the left and $j$-th row from the bottom. 
Then, we associate $X_\sigma \subset e_\ell\Pi$ to each lattice path $\sigma\in \mathbb{L}'(n)$ as the submodule given by the sub-quiver of \eqref{proj:ppaDpm1} consisting of all vertices underneath or right to $\sigma$. 
In the next figure, we give some examples of lattice paths $\sigma\in \mathbb{L}'(5)$ and the corresponding submodule $X_{\sigma}$ for the case when $n=6$ and $\ell=1$.

\begin{equation*}    
\begin{tabular}{ccccccccc}
    \begin{tikzpicture}
    \coordinate(x) at(0.6,0);
    \coordinate(y) at(0,0.6);
    \coordinate (00) at ($0*(x)+0*(y)$); 
    \coordinate (10) at ($1*(x)+0*(y)$); 
    \coordinate (20) at ($2*(x)+0*(y)$); 
    \coordinate (30) at ($3*(x)+0*(y)$); 
    \coordinate (40) at ($4*(x)+0*(y)$); 
    \coordinate (50) at ($5*(x)+0*(y)$); 
    
    \coordinate (01) at ($0*(x)+1*(y)$); 
    \coordinate (11) at ($1*(x)+1*(y)$); 
    \coordinate (21) at ($2*(x)+1*(y)$); 
    \coordinate (31) at ($3*(x)+1*(y)$); 
    \coordinate (41) at ($4*(x)+1*(y)$); 
    \coordinate (51) at ($5*(x)+1*(y)$); 

    \coordinate (02) at ($0*(x)+2*(y)$); 
    \coordinate (12) at ($1*(x)+2*(y)$); 
    \coordinate (22) at ($2*(x)+2*(y)$); 
    \coordinate (32) at ($3*(x)+2*(y)$); 
    \coordinate (42) at ($4*(x)+2*(y)$); 
    \coordinate (52) at ($5*(x)+2*(y)$); 
    
    \coordinate (03) at ($0*(x)+3*(y)$); 
    \coordinate (13) at ($1*(x)+3*(y)$); 
    \coordinate (23) at ($2*(x)+3*(y)$); 
    \coordinate (33) at ($3*(x)+3*(y)$); 
    \coordinate (43) at ($4*(x)+3*(y)$); 
    \coordinate (53) at ($5*(x)+3*(y)$); 

    \coordinate (04) at ($0*(x)+4*(y)$); 
    \coordinate (14) at ($1*(x)+4*(y)$); 
    \coordinate (24) at ($2*(x)+4*(y)$); 
    \coordinate (34) at ($3*(x)+4*(y)$); 
    \coordinate (44) at ($4*(x)+4*(y)$); 
    \coordinate (54) at ($5*(x)+4*(y)$); 
    
    \coordinate (05) at ($0*(x)+5*(y)$); 
    \coordinate (15) at ($1*(x)+5*(y)$); 
    \coordinate (25) at ($2*(x)+5*(y)$); 
    \coordinate (35) at ($3*(x)+5*(y)$); 
    \coordinate (45) at ($4*(x)+5*(y)$); 
    \coordinate (55) at ($5*(x)+5*(y)$); 

    \draw[dotted, opacity=0.8] (00)--(05);
    \draw[dotted, opacity=0.8] (10)--(15);
    \draw[dotted, opacity=0.8] (20)--(24);
    \draw[dotted, opacity=0.8] (30)--(33);
    \draw[dotted, opacity=0.8] (40)--(42);
    \draw[dotted, opacity=0.8] (50)--(51);
    \draw[dotted, opacity=0.8] (00)--(50);
    \draw[dotted, opacity=0.8] (01)--(51);
    \draw[dotted, opacity=0.8] (02)--(42);
    \draw[dotted, opacity=0.8] (03)--(33);
    \draw[dotted, opacity=0.8] (04)--(24);
    \draw[dotted, opacity=0.8] (05)--(15);
    
    \node at($0.5*(x)+0.5*(y)$) {\footnotesize$5$};
    \node at($1.5*(x)+0.5*(y)$) {\footnotesize$4$};
    \node at($2.5*(x)+0.5*(y)$) {\footnotesize$3$};
    \node at($3.5*(x)+0.5*(y)$) {\footnotesize$2$};
    \node at($4.5*(x)+0.5*(y)$) {\footnotesize$1$};
    \node at($0.5*(x)+1.5*(y)$) {\footnotesize$4$};
    \node at($1.5*(x)+1.5*(y)$) {\footnotesize$3$};
    \node at($2.5*(x)+1.5*(y)$) {\footnotesize$2$};
    \node at($3.5*(x)+1.5*(y)$) {\footnotesize$-1$};
    \node at($0.5*(x)+2.5*(y)$) {\footnotesize$3$};
    \node at($1.5*(x)+2.5*(y)$) {\footnotesize$2$};
    \node at($2.5*(x)+2.5*(y)$) {\footnotesize$1$};
    \node at($0.5*(x)+3.5*(y)$) {\footnotesize$2$};
    \node at($1.5*(x)+3.5*(y)$) {\footnotesize$-1$};
    \node at($0.5*(x)+4.5*(y)$) {\footnotesize$1$};

    \node at(00) {\footnotesize$\bullet$};
    \node at(50) {\footnotesize$\bullet$};
    \draw[thick] (00)--(50);
    \end{tikzpicture}
    & \quad 
    \begin{tikzpicture}
    \coordinate(x) at(0.6,0);
    \coordinate(y) at(0,0.6);
    \coordinate (00) at ($0*(x)+0*(y)$); 
    \coordinate (10) at ($1*(x)+0*(y)$); 
    \coordinate (20) at ($2*(x)+0*(y)$); 
    \coordinate (30) at ($3*(x)+0*(y)$); 
    \coordinate (40) at ($4*(x)+0*(y)$); 
    \coordinate (50) at ($5*(x)+0*(y)$); 
    
    \coordinate (01) at ($0*(x)+1*(y)$); 
    \coordinate (11) at ($1*(x)+1*(y)$); 
    \coordinate (21) at ($2*(x)+1*(y)$); 
    \coordinate (31) at ($3*(x)+1*(y)$); 
    \coordinate (41) at ($4*(x)+1*(y)$); 
    \coordinate (51) at ($5*(x)+1*(y)$); 

    \coordinate (02) at ($0*(x)+2*(y)$); 
    \coordinate (12) at ($1*(x)+2*(y)$); 
    \coordinate (22) at ($2*(x)+2*(y)$); 
    \coordinate (32) at ($3*(x)+2*(y)$); 
    \coordinate (42) at ($4*(x)+2*(y)$); 
    \coordinate (52) at ($5*(x)+2*(y)$); 
    
    \coordinate (03) at ($0*(x)+3*(y)$); 
    \coordinate (13) at ($1*(x)+3*(y)$); 
    \coordinate (23) at ($2*(x)+3*(y)$); 
    \coordinate (33) at ($3*(x)+3*(y)$); 
    \coordinate (43) at ($4*(x)+3*(y)$); 
    \coordinate (53) at ($5*(x)+3*(y)$); 

    \coordinate (04) at ($0*(x)+4*(y)$); 
    \coordinate (14) at ($1*(x)+4*(y)$); 
    \coordinate (24) at ($2*(x)+4*(y)$); 
    \coordinate (34) at ($3*(x)+4*(y)$); 
    \coordinate (44) at ($4*(x)+4*(y)$); 
    \coordinate (54) at ($5*(x)+4*(y)$); 
    
    \coordinate (05) at ($0*(x)+5*(y)$); 
    \coordinate (15) at ($1*(x)+5*(y)$); 
    \coordinate (25) at ($2*(x)+5*(y)$); 
    \coordinate (35) at ($3*(x)+5*(y)$); 
    \coordinate (45) at ($4*(x)+5*(y)$); 
    \coordinate (55) at ($5*(x)+5*(y)$); 

    \fill[white!30!lightgray] (40)rectangle(51); 
    \draw[very thick] (00)--(40)--(41);

    \draw[dotted, opacity=0.8] (00)--(05);
    \draw[dotted, opacity=0.8] (10)--(15);
    \draw[dotted, opacity=0.8] (20)--(24);
    \draw[dotted, opacity=0.8] (30)--(33);
    \draw[dotted, opacity=0.8] (40)--(42);
    \draw[dotted, opacity=0.8] (50)--(51);
    \draw[dotted, opacity=0.8] (00)--(50);
    \draw[dotted, opacity=0.8] (01)--(51);
    \draw[dotted, opacity=0.8] (02)--(42);
    \draw[dotted, opacity=0.8] (03)--(33);
    \draw[dotted, opacity=0.8] (04)--(24);
    \draw[dotted, opacity=0.8] (05)--(15);
    
    \node at($0.5*(x)+0.5*(y)$) {\footnotesize$5$};
    \node at($1.5*(x)+0.5*(y)$) {\footnotesize$4$};
    \node at($2.5*(x)+0.5*(y)$) {\footnotesize$3$};
    \node at($3.5*(x)+0.5*(y)$) {\footnotesize$2$};
    \node at($4.5*(x)+0.5*(y)$) {\footnotesize$1$};
    \node at($0.5*(x)+1.5*(y)$) {\footnotesize$4$};
    \node at($1.5*(x)+1.5*(y)$) {\footnotesize$3$};
    \node at($2.5*(x)+1.5*(y)$) {\footnotesize$2$};
    \node at($3.5*(x)+1.5*(y)$) {\footnotesize$-1$};
    \node at($0.5*(x)+2.5*(y)$) {\footnotesize$3$};
    \node at($1.5*(x)+2.5*(y)$) {\footnotesize$2$};
    \node at($2.5*(x)+2.5*(y)$) {\footnotesize$1$};
    \node at($0.5*(x)+3.5*(y)$) {\footnotesize$2$};
    \node at($1.5*(x)+3.5*(y)$) {\footnotesize$-1$};
    \node at($0.5*(x)+4.5*(y)$) {\footnotesize$1$};

    \node at(00) {\footnotesize$\bullet$};
    \node at(41) {\footnotesize$\bullet$};
    \end{tikzpicture}
    & \quad 
    \begin{tikzpicture}
    \coordinate(x) at(0.6,0);
    \coordinate(y) at(0,0.6);
    \coordinate (00) at ($0*(x)+0*(y)$); 
    \coordinate (10) at ($1*(x)+0*(y)$); 
    \coordinate (20) at ($2*(x)+0*(y)$); 
    \coordinate (30) at ($3*(x)+0*(y)$); 
    \coordinate (40) at ($4*(x)+0*(y)$); 
    \coordinate (50) at ($5*(x)+0*(y)$); 
    
    \coordinate (01) at ($0*(x)+1*(y)$); 
    \coordinate (11) at ($1*(x)+1*(y)$); 
    \coordinate (21) at ($2*(x)+1*(y)$); 
    \coordinate (31) at ($3*(x)+1*(y)$); 
    \coordinate (41) at ($4*(x)+1*(y)$); 
    \coordinate (51) at ($5*(x)+1*(y)$); 

    \coordinate (02) at ($0*(x)+2*(y)$); 
    \coordinate (12) at ($1*(x)+2*(y)$); 
    \coordinate (22) at ($2*(x)+2*(y)$); 
    \coordinate (32) at ($3*(x)+2*(y)$); 
    \coordinate (42) at ($4*(x)+2*(y)$); 
    \coordinate (52) at ($5*(x)+2*(y)$); 
    
    \coordinate (03) at ($0*(x)+3*(y)$); 
    \coordinate (13) at ($1*(x)+3*(y)$); 
    \coordinate (23) at ($2*(x)+3*(y)$); 
    \coordinate (33) at ($3*(x)+3*(y)$); 
    \coordinate (43) at ($4*(x)+3*(y)$); 
    \coordinate (53) at ($5*(x)+3*(y)$); 

    \coordinate (04) at ($0*(x)+4*(y)$); 
    \coordinate (14) at ($1*(x)+4*(y)$); 
    \coordinate (24) at ($2*(x)+4*(y)$); 
    \coordinate (34) at ($3*(x)+4*(y)$); 
    \coordinate (44) at ($4*(x)+4*(y)$); 
    \coordinate (54) at ($5*(x)+4*(y)$); 
    
    \coordinate (05) at ($0*(x)+5*(y)$); 
    \coordinate (15) at ($1*(x)+5*(y)$); 
    \coordinate (25) at ($2*(x)+5*(y)$); 
    \coordinate (35) at ($3*(x)+5*(y)$); 
    \coordinate (45) at ($4*(x)+5*(y)$); 
    \coordinate (55) at ($5*(x)+5*(y)$); 

    \fill[white!30!lightgray] (10)rectangle(51);
    \fill[white!30!lightgray] (21)rectangle(42);
    \fill[white!30!lightgray] (22)rectangle(33);
    \draw[very thick] (00)--(10)--(11)--(21)--(23);

    \draw[dotted, opacity=0.8] (00)--(05);
    \draw[dotted, opacity=0.8] (10)--(15);
    \draw[dotted, opacity=0.8] (20)--(24);
    \draw[dotted, opacity=0.8] (30)--(33);
    \draw[dotted, opacity=0.8] (40)--(42);
    \draw[dotted, opacity=0.8] (50)--(51);
    \draw[dotted, opacity=0.8] (00)--(50);
    \draw[dotted, opacity=0.8] (01)--(51);
    \draw[dotted, opacity=0.8] (02)--(42);
    \draw[dotted, opacity=0.8] (03)--(33);
    \draw[dotted, opacity=0.8] (04)--(24);
    \draw[dotted, opacity=0.8] (05)--(15);
    
    \node at($0.5*(x)+0.5*(y)$) {\footnotesize$5$};
    \node at($1.5*(x)+0.5*(y)$) {\footnotesize$4$};
    \node at($2.5*(x)+0.5*(y)$) {\footnotesize$3$};
    \node at($3.5*(x)+0.5*(y)$) {\footnotesize$2$};
    \node at($4.5*(x)+0.5*(y)$) {\footnotesize$1$};
    \node at($0.5*(x)+1.5*(y)$) {\footnotesize$4$};
    \node at($1.5*(x)+1.5*(y)$) {\footnotesize$3$};
    \node at($2.5*(x)+1.5*(y)$) {\footnotesize$2$};
    \node at($3.5*(x)+1.5*(y)$) {\footnotesize$-1$};
    \node at($0.5*(x)+2.5*(y)$) {\footnotesize$3$};
    \node at($1.5*(x)+2.5*(y)$) {\footnotesize$2$};
    \node at($2.5*(x)+2.5*(y)$) {\footnotesize$1$};
    \node at($0.5*(x)+3.5*(y)$) {\footnotesize$2$};
    \node at($1.5*(x)+3.5*(y)$) {\footnotesize$-1$};
    \node at($0.5*(x)+4.5*(y)$) {\footnotesize$1$};

    \node at(00) {\footnotesize$\bullet$};
    \node at(23) {\footnotesize$\bullet$};
    \end{tikzpicture}
    & \quad 
    \begin{tikzpicture}
    \coordinate(x) at(0.6,0);
    \coordinate(y) at(0,0.6);
    \coordinate (00) at ($0*(x)+0*(y)$); 
    \coordinate (10) at ($1*(x)+0*(y)$); 
    \coordinate (20) at ($2*(x)+0*(y)$); 
    \coordinate (30) at ($3*(x)+0*(y)$); 
    \coordinate (40) at ($4*(x)+0*(y)$); 
    \coordinate (50) at ($5*(x)+0*(y)$); 
    
    \coordinate (01) at ($0*(x)+1*(y)$); 
    \coordinate (11) at ($1*(x)+1*(y)$); 
    \coordinate (21) at ($2*(x)+1*(y)$); 
    \coordinate (31) at ($3*(x)+1*(y)$); 
    \coordinate (41) at ($4*(x)+1*(y)$); 
    \coordinate (51) at ($5*(x)+1*(y)$); 

    \coordinate (02) at ($0*(x)+2*(y)$); 
    \coordinate (12) at ($1*(x)+2*(y)$); 
    \coordinate (22) at ($2*(x)+2*(y)$); 
    \coordinate (32) at ($3*(x)+2*(y)$); 
    \coordinate (42) at ($4*(x)+2*(y)$); 
    \coordinate (52) at ($5*(x)+2*(y)$); 
    
    \coordinate (03) at ($0*(x)+3*(y)$); 
    \coordinate (13) at ($1*(x)+3*(y)$); 
    \coordinate (23) at ($2*(x)+3*(y)$); 
    \coordinate (33) at ($3*(x)+3*(y)$); 
    \coordinate (43) at ($4*(x)+3*(y)$); 
    \coordinate (53) at ($5*(x)+3*(y)$); 

    \coordinate (04) at ($0*(x)+4*(y)$); 
    \coordinate (14) at ($1*(x)+4*(y)$); 
    \coordinate (24) at ($2*(x)+4*(y)$); 
    \coordinate (34) at ($3*(x)+4*(y)$); 
    \coordinate (44) at ($4*(x)+4*(y)$); 
    \coordinate (54) at ($5*(x)+4*(y)$); 
    
    \coordinate (05) at ($0*(x)+5*(y)$); 
    \coordinate (15) at ($1*(x)+5*(y)$); 
    \coordinate (25) at ($2*(x)+5*(y)$); 
    \coordinate (35) at ($3*(x)+5*(y)$); 
    \coordinate (45) at ($4*(x)+5*(y)$); 
    \coordinate (55) at ($5*(x)+5*(y)$); 

    \fill[white!30!lightgray] (00)rectangle(51);
    \fill[white!30!lightgray] (01)rectangle(42);
    \fill[white!30!lightgray] (02)rectangle(33);
    \fill[white!30!lightgray] (03)rectangle(24);
    \fill[white!30!lightgray] (04)rectangle(15);
    \draw[very thick] (00)--(05);

    \draw[dotted, opacity=0.8] (00)--(05);
    \draw[dotted, opacity=0.8] (10)--(15);
    \draw[dotted, opacity=0.8] (20)--(24);
    \draw[dotted, opacity=0.8] (30)--(33);
    \draw[dotted, opacity=0.8] (40)--(42);
    \draw[dotted, opacity=0.8] (50)--(51);
    \draw[dotted, opacity=0.8] (00)--(50);
    \draw[dotted, opacity=0.8] (01)--(51);
    \draw[dotted, opacity=0.8] (02)--(42);
    \draw[dotted, opacity=0.8] (03)--(33);
    \draw[dotted, opacity=0.8] (04)--(24);
    \draw[dotted, opacity=0.8] (05)--(15);
    
    \node at($0.5*(x)+0.5*(y)$) {\footnotesize$5$};
    \node at($1.5*(x)+0.5*(y)$) {\footnotesize$4$};
    \node at($2.5*(x)+0.5*(y)$) {\footnotesize$3$};
    \node at($3.5*(x)+0.5*(y)$) {\footnotesize$2$};
    \node at($4.5*(x)+0.5*(y)$) {\footnotesize$1$};
    \node at($0.5*(x)+1.5*(y)$) {\footnotesize$4$};
    \node at($1.5*(x)+1.5*(y)$) {\footnotesize$3$};
    \node at($2.5*(x)+1.5*(y)$) {\footnotesize$2$};
    \node at($3.5*(x)+1.5*(y)$) {\footnotesize$-1$};
    \node at($0.5*(x)+2.5*(y)$) {\footnotesize$3$};
    \node at($1.5*(x)+2.5*(y)$) {\footnotesize$2$};
    \node at($2.5*(x)+2.5*(y)$) {\footnotesize$1$};
    \node at($0.5*(x)+3.5*(y)$) {\footnotesize$2$};
    \node at($1.5*(x)+3.5*(y)$) {\footnotesize$-1$};
    \node at($0.5*(x)+4.5*(y)$) {\footnotesize$1$};

    \node at(00) {\footnotesize$\bullet$};
    \node at(05) {\footnotesize$\bullet$};
    
    \end{tikzpicture}
    \\ 
    \footnotesize$\area(\sigma)=0$ & 
    \footnotesize$\area(\sigma)=1$&
    \footnotesize$\area(\sigma)=7$&
    \footnotesize$\area(\sigma)=15$\\
    \footnotesize$X_{\sigma}=0$ &
    \footnotesize$X_{\sigma}=1$ &
    \footnotesize$X_{\sigma}= \tiny
    \begin{tikzpicture}[baseline=2mm]
    \coordinate(x) at(135:0.25);
    \coordinate(y) at(45:0.25);
    \node at($3*(x)+0*(y)$) {$4$}; 
    \node at($2*(x)+0*(y)$) {$3$}; 
    \node at($1*(x)+0*(y)$) {$2$}; 
    \node at($0*(x)+0*(y)$) {$1$}; 
    \node at($1*(x)+1*(y)$) {$-1$};
    \node at($2*(x)+2*(y)$) {$1$};
    \node at($2*(x)+1*(y)$) {$2$}; 
    \end{tikzpicture}$ &
    \footnotesize$X_{\sigma}=e_{1}\Pi$ 
\end{tabular}
\end{equation*}

\begin{prop}\cite[Theorem 6.5]{IRRT18}\label{bij:OOppalgDpm1}
Let $\ell=\pm1$. Every submodule of $e_\ell\Pi$ is $\tau$-rigid. 
Furthermore, we have a bijection 
\begin{equation}
\mathbb{L}'(n-1) \xrightarrow{\sim} [e_\ell \Pi]_{\rm s} \cup \{0\}
\quad (\sigma\mapsto X_{\sigma}) 
\end{equation}
satisfying $\dim_kX_{\sigma} = \area(\sigma).$
In particular, we have 
\begin{equation*}
    \#[e_\ell\Pi]_{\rm s} +1 = \#\mathbb{L}'(n-1) = 2^{n-1}. 
\end{equation*}

\end{prop}

From now on, we show the statement of Proposition \ref{prop:OOppalgD} for $\ell=\pm1$ by induction on $n$. 
For $1\leq n\leq 4$, it is easy to check that the assertion holds true. 
Assume that $n> 4$. 
The set $\mathbb{L}'(n-1)$ is a disjoint union of ${\mathbb{L}'}^{\rm E}(n-1)$ and ${\mathbb{L}'}^{\rm N}(n-1)$ consisting of all lattice paths whose first steps are East and North,  respectively. By ignoring the first step of each lattice path, we have bijections 
\begin{equation} \label{eq:OPppalgDpm1}
a\colon \mathbb{L}'(n-2) \xrightarrow{\sim} 
{\mathbb{L}'}^{\rm E}(n-1)
 \quad \text{and} \quad b\colon \mathbb{L}'(n-2) \xrightarrow{\sim} 
 {\mathbb{L}'}^{\rm N}(n-1)
\end{equation}
such that $\area (a(\sigma)) = \area(a)$ and $\area(b(\sigma')) = \area(\sigma') +(n-1)$. 
Using Proposition \ref{bij:OOppalgDpm1}, we have 
\begin{eqnarray*}    
\Dim([e_{\ell}\Pi]_{\rm s})&\overset{{\rm Prop.\,}\ref{bij:OOppalgDpm1}}{=}&\sum_{\sigma\in \mathbb{L}'(n-1)}\area(\sigma)\\ 
&=&\sum_{\sigma\in \mathbb{L}'^{\rm E}(n-1)}\area(\sigma) + \sum_{\sigma'\in \mathbb{L}'^{\rm N}(n-1)}\area(\sigma') \\
&=& \sum_{\sigma\in \mathbb{L}'(n-2)}\area(\sigma)+\sum_{\sigma'\in \mathbb{L}'(n-2)}\area(\sigma') +(n-1)2^{n-2}\\ 
&=&2\sum_{\sigma\in \mathbb{L}'(n-2)}\area(\sigma)+(n-1)2^{n-2}\\ 
&\overset{\rm induction}{=}& (n-1)(n-2)2^{n-3} + (n-1)2^{n-2}\\ 
&=& n(n-1)2^{n-3}
\end{eqnarray*}
as desired. It finishes the proof for the cases $\ell = \pm1$.

\subsubsection{Case $\ell\neq \pm1$} \label{subsec:D2}
Next, we consider the case $\ell\neq \pm1$, that is, $\ell \in\{2,\ldots,n-1\}$. By \cite[Lemma 6.9]{IRRT18}, for two scalars $\alpha,\beta\in k$ with $\alpha+\beta =1$, we can describe the representation of $e_\ell\Pi$ as follows. 

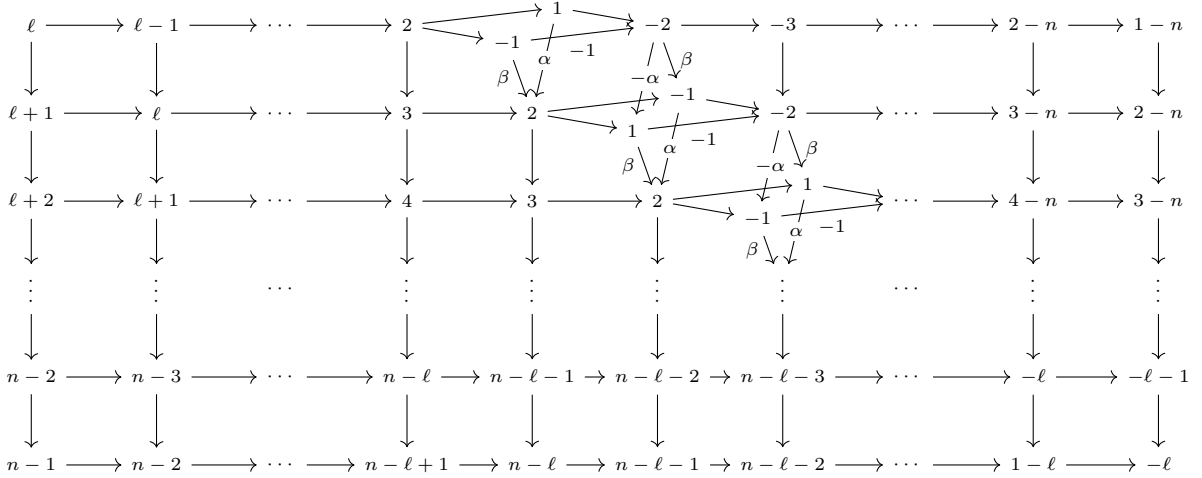
\begin{equation*}
\begin{tikzpicture}[scale=0.9]
    \coordinate(x) at(1.85,0);
    \coordinate(y) at(0,1.3);
    \node(00) at ($0*(x)+0*(y)$) {\tiny$n-1$}; 
    \node(10) at ($1*(x)+0*(y)$) {\tiny$n-2$};
    \node(20) at ($2*(x)+0*(y)$) {\tiny$\cdots$}; 
    \node(30) at ($3*(x)+0*(y)$) {\tiny$n-\ell+1$};
    \node(40) at ($4*(x)+0*(y)$) {\tiny$n-\ell$};
    \node(50) at ($5*(x)+0*(y)$) {\tiny$n-\ell-1$}; 
    \node(60) at ($6*(x)+0*(y)$) {\tiny$n-\ell-2$};
    \node(70) at ($7*(x)+0*(y)$) {\tiny$\cdots$}; 
    \node(80) at ($8*(x)+0*(y)$) {\tiny$1-\ell$};
    \node(90) at ($9*(x)+0*(y)$) {\tiny$-\ell$}; 
    \node(01) at ($0*(x)+1*(y)$) {\tiny$n-2$}; 
    \node(11) at ($1*(x)+1*(y)$) {\tiny$n-3$};
    \node(21) at ($2*(x)+1*(y)$) {\tiny$\cdots$}; 
    \node(31) at ($3*(x)+1*(y)$) {\tiny$n-\ell$};
    \node(41) at ($4*(x)+1*(y)$) {\tiny$n-\ell-1$};
    \node(51) at ($5*(x)+1*(y)$) {\tiny$n-\ell-2$}; 
    \node(61) at ($6*(x)+1*(y)$) {\tiny$n-\ell-3$};
    \node(71) at ($7*(x)+1*(y)$) {\tiny$\cdots$}; 
    \node(81) at ($8*(x)+1*(y)$) {\tiny$-\ell$};
    \node(91) at ($9*(x)+1*(y)$) {\tiny$-\ell-1$}; 
    \node(02) at ($0*(x)+2*(y)$) {\tiny\rotatebox{90}{$\cdots$}}; 
    \node(12) at ($1*(x)+2*(y)$) {\tiny\rotatebox{90}{$\cdots$}};
    \node(22) at ($2*(x)+2*(y)$) {\tiny$\cdots$}; 
    \node(32) at ($3*(x)+2*(y)$) {\tiny\rotatebox{90}{$\cdots$}};
    \node(42) at ($4*(x)+2*(y)$) {\tiny\rotatebox{90}{$\cdots$}};
    \node(52) at ($5*(x)+2*(y)$) {\tiny\rotatebox{90}{$\cdots$}}; 
    \node(62) at ($6*(x)+2*(y)$) {\tiny\rotatebox{90}{$\cdots$}};
    \node(72) at ($7*(x)+2*(y)$) {\tiny$\cdots$}; 
    \node(82) at ($8*(x)+2*(y)$) {\tiny\rotatebox{90}{$\cdots$}};
    \node(92) at ($9*(x)+2*(y)$) {\tiny\rotatebox{90}{$\cdots$}};
    \node(03) at ($0*(x)+3*(y)$) {\tiny$\ell+2$}; 
    \node(13) at ($1*(x)+3*(y)$) {\tiny$\ell+1$};
    \node(23) at ($2*(x)+3*(y)$) {\tiny$\cdots$}; 
    \node(33) at ($3*(x)+3*(y)$) {\tiny$4$};
    \node(43) at ($4*(x)+3*(y)$) {\tiny$3$};
    \node(53) at ($5*(x)+3*(y)$) {\tiny$2$}; 
    \node(63) at ($5.8*(x)+2.8*(y)$) {\tiny$-1$};
    \node(63p) at ($6.2*(x)+3.2*(y)$) {\tiny$1$};
    \node(73) at ($7*(x)+3*(y)$) {\tiny$\cdots$}; 
    \node(83) at ($8*(x)+3*(y)$) {\tiny$4-n$};
    \node(93) at ($9*(x)+3*(y)$) {\tiny$3-n$}; 
    \node(04) at ($0*(x)+4*(y)$) {\tiny$\ell+1$}; 
    \node(14) at ($1*(x)+4*(y)$) {\tiny$\ell$};
    \node(24) at ($2*(x)+4*(y)$) {\tiny$\cdots$}; 
    \node(34) at ($3*(x)+4*(y)$) {\tiny$3$};
    \node(44) at ($4*(x)+4*(y)$) {\tiny$2$};
    \node(54) at ($4.8*(x)+3.8*(y)$) {\tiny$1$};
    \node(54p) at ($5.2*(x)+4.2*(y)$) {\tiny$-1$};
    \node(64) at ($6*(x)+4*(y)$) {\tiny$-2$};
    \node(74) at ($7*(x)+4*(y)$) {\tiny$\cdots$}; 
    \node(84) at ($8*(x)+4*(y)$) {\tiny$3-n$};
    \node(94) at ($9*(x)+4*(y)$) {\tiny$2-n$}; 
    \node(05) at ($0*(x)+5*(y)$) {\tiny$\ell$}; 
    \node(15) at ($1*(x)+5*(y)$) {\tiny$\ell-1$};
    \node(25) at ($2*(x)+5*(y)$) {\tiny$\cdots$}; 
    \node(35) at ($3*(x)+5*(y)$) {\tiny$2$};
    \node(45) at ($3.8*(x)+4.8*(y)$) {\tiny$-1$};
    \node(45p) at ($4.2*(x)+5.2*(y)$) {\tiny$1$};
    \node(55) at ($5*(x)+5*(y)$) {\tiny$-2$}; 
    \node(65) at ($6*(x)+5*(y)$) {\tiny$-3$};
    \node(75) at ($7*(x)+5*(y)$) {\tiny$\cdots$}; 
    \node(85) at ($8*(x)+5*(y)$) {\tiny$2-n$};
    \node(95) at ($9*(x)+5*(y)$) {\tiny$1-n$}; 
    \draw[->] (00)--(10); 
    \draw[->] (10)--(20); 
    \draw[->] (20)--(30);
    \draw[->] (30)--(40);
    \draw[->] (40)--(50);
    \draw[->] (50)--(60);
    \draw[->] (60)--(70);
    \draw[->] (70)--(80);
    \draw[->] (80)--(90);
    \draw[->] (01)--(11);
    \draw[->] (11)--(21);
    \draw[->] (21)--(31);
    \draw[->] (31)--(41);
    \draw[->] (41)--(51);
    \draw[->] (51)--(61);
    \draw[->] (61)--(71);
    \draw[->] (71)--(81);
    \draw[->] (81)--(91);
    \draw[->] (03)--(13);
    \draw[->] (13)--(23);
    \draw[->] (23)--(33);
    \draw[->] (33)--(43);
    \draw[->] (43)--(53);
    \draw[->] (53)--(63);
    \draw[->] (53)--(63p);
    \draw[->] (63)--node[below]{\tiny$-1$}(73);
    \draw[->] (63p)--(73);
    \draw[->] (73)--(83);
    \draw[->] (83)--(93);
    \draw[->] (04)--(14);
    \draw[->] (14)--(24);
    \draw[->] (24)--(34);
    \draw[->] (34)--(44);
    \draw[->] (44)--(54p);
    \draw[->] (44)--(54);
    \draw[->] (54)--node[below]{\tiny$-1$}(64);
    \draw[->] (54p)--(64);
    \draw[->] (64)--(74);
    \draw[->] (74)--(84);
    \draw[->] (84)--(94);
    \draw[->] (05)--(15);
    \draw[->] (15)--(25);
    \draw[->] (25)--(35);
    \draw[->] (35)--(45);
    \draw[->] (35)--(45p);
    \draw[->] (45)--node[below]{\tiny$-1$}(55);
    \draw[->] (45p)--(55);
    \draw[->] (55)--(65);
    \draw[->] (65)--(75);
    \draw[->] (75)--(85);
    \draw[->] (85)--(95);
    \draw[<-] (00)--(01);
    \draw[<-] (01)--(02);
    \draw[<-] (02)--(03);
    \draw[<-] (03)--(04);
    \draw[<-] (04)--(05);
    \draw[<-] (10)--(11);
    \draw[<-] (11)--(12);
    \draw[<-] (12)--(13);
    \draw[<-] (13)--(14);
    \draw[<-] (14)--(15);
    \draw[<-] (30)--(31);
    \draw[<-] (31)--(32);
    \draw[<-] (32)--(33);
    \draw[<-] (33)--(34);
    \draw[<-] (34)--(35);
    \draw[<-] (40)--(41);
    \draw[<-] (41)--(42);
    \draw[<-] (42)--(43);
    \draw[<-] (43)--(44);
    \draw[<-] (44)--node[left]{\tiny$\beta$}(45);
    \draw[<-] (44)--node[fill=white, inner sep = 0.7mm]{\tiny$\alpha$}(45p);
    \draw[<-] (50)--(51);
    \draw[<-] (51)--(52);
    \draw[<-] (52)--(53);
    \draw[<-] (53)--node[left]{\tiny$\beta$}(54);
    \draw[<-] (53)--node[fill=white, inner sep = 0.7mm]{\tiny$\alpha$}(54p);
    \draw[<-] (54)--node[fill=white, inner sep = 0.7mm]{\tiny$-\alpha$}(55);
    \draw[<-] (54p)--node[right]{\tiny$\beta$}(55);
    \draw[<-] (60)--(61);
    \draw[<-] (61)--(62);
    \draw[<-] (62)--node[left]{\tiny$\beta$}(63);
    \draw[<-] (62)--node[fill=white, inner sep = 0.7mm]{\tiny$\alpha$}(63p);
    \draw[<-] (63)--node[fill=white, inner sep = 0.7mm]{\tiny$-\alpha$}(64);
    \draw[<-] (63p)--node[right]{\tiny$\beta$}(64);
    \draw[<-] (64)--(65);
    \draw[<-] (80)--(81);
    \draw[<-] (81)--(82);
    \draw[<-] (82)--(83);
    \draw[<-] (83)--(84);
    \draw[<-] (84)--(85);
    \draw[<-] (90)--(91);
    \draw[<-] (91)--(92);
    \draw[<-] (92)--(93);
    \draw[<-] (93)--(94);
    \draw[<-] (94)--(95);
\end{tikzpicture}
\end{equation*}

Here, each number $i$ shows a one dimensional $k$-vector space $k$ lying on the vertex $i$ if $i\geq -1$ and $-i$ if $i\leq -2$. We simply represent it by using arrays as follows. 
\begin{equation}\label{array:D}
\begin{split}
\begin{tabular}{|ccccccccccccc|cc}
    \hline 
    {\tiny $\ell$} &{\tiny $\ell-1$}&{\tiny $\cdots$}&{\tiny $2$}&{\tiny $\substack{1 \\ {-1}}$} &{\tiny $-2$}&{\tiny $\cdots$}&{\tiny $1-\ell$}&{\tiny $-\ell$} &{\tiny $-\ell-1$} &{\tiny $\cdots$}&{\tiny $2-n$}&{\tiny $1-n$}
    \\ \hline 
    {\tiny $\ell+1$} &{\tiny $\ell$}&{\tiny $\cdots$}&{\tiny $3$}&{\tiny $2$}&{\tiny $\substack{1\\-1}$}&{\tiny $\cdots$}&{\tiny $2-\ell$}&{\tiny $1-\ell$}&{\tiny $-\ell$}&{\tiny $\cdots$}&{\tiny $3-n$}&{\tiny $2-n$}\\ \hline
    \rotatebox{90}{\tiny$\quad \cdots\quad $}&\rotatebox{90}{\tiny$\quad \cdots\quad $}&\rotatebox{90}{\tiny$\quad \cdots\quad $}&\rotatebox{90}{\tiny$\quad \cdots\quad $} &\rotatebox{90}{\tiny$\quad \cdots\quad $}&\rotatebox{90}{\tiny$\quad \cdots\quad $}&\rotatebox{90}{\tiny$\quad \cdots\quad $}&\rotatebox{90}{\tiny$\quad \cdots\quad $}&\rotatebox{90}{\tiny$\quad \cdots\quad $}&\rotatebox{90}{\tiny$\quad \cdots\quad $}&\rotatebox{90}{\tiny$\quad \cdots\quad $}&\rotatebox{90}{\tiny$\quad \cdots\quad $}&\rotatebox{90}{\tiny$\quad \cdots\quad $}\\ \hline
    {\tiny $n-2$}&{\tiny $n-3$}&{\tiny $\cdots$}&{\tiny $\ell$}& {\tiny $\ell-1$}&{\tiny $\ell-2$}&{\tiny $\cdots$}&{\tiny $\substack{1\\-1}$} &{\tiny $-2$} & {\tiny $-3$} &{\tiny $\cdots$}&{\tiny $-\ell$}&{\tiny $-\ell-1$}\\ \hline 
    {\tiny $n-1$}&{\tiny $n-2$}&{\tiny $\cdots$} & {\tiny $\ell+1$}&{\tiny $\ell$}&{\tiny $\ell-1$}& {\tiny $\cdots$} &{\tiny $2$} & {\tiny $\substack{1\\-1}$} &{\tiny $-2$}&{\tiny $\cdots$} &{\tiny $1-\ell$}& {\tiny $-\ell$}   
    \\ \hline 
\end{tabular}
\end{split}
\end{equation}

Now, we recall a classification of $\tau$-rigid submodules of $e_{\ell}\Pi$ given in \cite[Section 6.2]{IRRT18}. 
Remark that our convention is dual to theirs. Let 
\begin{equation}\nonumber
    \mathbb{U}_\ell:=\{(u_{\ell+1},\cdots,u_n)\in\{\pm1,\cdots,\pm n\}^{n-\ell}\mid  u_{\ell+1}>\cdots>u_n\}. 
\end{equation}
For each $u \in \mathbb{U}_{\ell}$, we define sub-arrays $S(u)$ of \eqref{array:D} by 
\begin{equation*}
    S(u) := \begin{tabular}{|l|l}\hline 
        {\tiny $C(u_{n},n-1)$}  \\ \hline
        {\tiny $C(u_{n-1},n-2)$}  \\ \hline
        \quad \quad  \rotatebox{90}{$\cdots$}  \\ \hline
        {\tiny $C(u_{\ell+3},\ell+2)$}  \\ \hline
        {\tiny $C(u_{\ell+2},\ell+1)$}  \\ \hline
        {\tiny $C(u_{\ell+1},\ell)$}  \\ \hline
    \end{tabular}
\quad \text{where} \quad 
    C(j,m) := \begin{cases}\tabcolsep = 1.5pt
        \begin{tabular}{rc}
        $\emptyset$ & \text{$-m>j$,} \\ 
        \begin{tabular}{|cccccc|}\hline 
            {\tiny $j$} & {\tiny $j-1$} & {\tiny $j-2$}&{\tiny $\cdots$}& {\tiny $-m+1$}&{\tiny $-m$} \\ \hline   
        \end{tabular} & \text{$-1\geq j \geq -m$,} \\ 
        \begin{tabular}{|cccccc|}\hline 
            {\tiny $1$} & {\tiny $-2$} & {\tiny $-3$} & {\tiny $\cdots$}& {\tiny $-m+1$}&{\tiny $-m$} \\ \hline   
        \end{tabular} & \text{$j=1$,} \\
        \begin{tabular}{|cccccc|}\hline 
            {\tiny $\substack{1\\-1}$} & {\tiny $-2$} & {\tiny $-3$}&{\tiny $\cdots$}& {\tiny $-m+1$}&{\tiny $-m$} \\ \hline 
        \end{tabular} & \text{$j =2$,} \\ 
        \begin{tabular}{|cccccccccc|}\hline 
            {\tiny $j-1$} & {\tiny $j-2$} & {\tiny $\cdots$}& {\tiny $2$}&{\tiny $\substack{1\\-1}$} &{\tiny $-2$}&&{\tiny $\cdots$}&{\tiny $-m+1$}&{\tiny $-m$}\\ \hline   
        \end{tabular} & \text{$j \geq 3$.} 
        \end{tabular}
    \end{cases}
\end{equation*}
Then, let $Y_u$ be a submodule of $e_{\ell}\Pi$ corresponding to $S(u)$. 
In addition, let ${\rm n}(u)$ be the number of integers appearing in the diagram $S(u)$.

\begin{prop}\cite[Theorem 6.12]{IRRT18}\label{bij:OOppalgD2}
For $\ell\in \{2,\ldots,n-1\}$, we have a bijection 
\begin{equation*}\label{bij PiD}
\mathbb{U}_\ell \xrightarrow{\sim} [e_\ell\Pi]_{\rm s} \cup\{0\} \quad (u \mapsto Y_u)
\end{equation*} 
satisfying $\dim_kY_u = {\rm n}(u).$
In particular, we have 
\begin{equation*}
    \#[e_\ell\Pi]_{\rm s}+1 = \# \mathbb{U}_\ell = 2^{n-\ell}\binom{n}{\ell}. 
\end{equation*} 
\end{prop}

\begin{lemm}\label{n to w*}
For each $u\in \mathbb{U}_{\ell}$, we have 
\begin{equation*}
    {\rm n}(u) = \sum_{i=\ell+1}^n (i + u_i^*), 
    \quad 
\text{where} \quad  
    u_i^* :=
\begin{cases}
u_i &  u_i < 0, \\
u_i-2 & u_i > 0.
\end{cases}
\end{equation*}
\end{lemm}

\begin{proof}
By definition, the number of integers appearing in $C(j,m)$ is exactly $0$ (respectively, $j+m+1$ and $j+m-1$) if $-m > j$ (respectively, $-1 \geq j \geq -m$ and $j >0$). 

Now, let $u\in \mathbb{U}_{\ell}$. 
From the previous sentence, the number of integers appearing in $C(u_{i},i-1)$ is exactly $i+u_i^*$ whenever $u_i^* \geq -i+1$. 
On the other hand, $C(u_{i},i-1) = \emptyset$ implies $u_{i} = -i$ since they satisfy $u_{\ell+1} > \cdots > u_{n} \geq -n$. 
Thus, we get the assertion. 
\end{proof}

Using the above results, we show the statement of Proposition \ref{prop:OOppalgD} for $\ell \in \{2,\ldots,n-1\}$ as follows. 
We have 
\begin{eqnarray} \nonumber \hspace{-10mm}
\Dim([e_\ell\Pi]_{\rm s}) &\overset{\rm Prop.\,\ref{bij:OOppalgD2}}
{=}& \sum_{u \in \mathbb{U}_{\ell}} {\rm n}(u)\\
&\overset{\rm Lem.\,\ref{n to w*}}{=}& 
\sum_{u\in \mathbb{U}_\ell}
\left(\sum_{i=\ell+1}^n(i + u_i^*)\right) \nonumber \\ 
&=&\#\mathbb{U}_\ell \cdot \sum_{i=\ell+1}^n i + \sum_{u\in \mathbb{U}_{\ell}}\left(\sum_{i=\ell+1}^n u_i\right) - 
2 \sum_{u\in \mathbb{U}_\ell} \#\{i\mid u_i > 0\}. 
\label{cal ppalgDl}
\end{eqnarray}

We compute each summand of \eqref{cal ppalgDl}.
By the definition of $\mathbb{U}_\ell$, there is a one-to-one correspondence $\mathbb{U}_{\ell}\to \mathbb{U}_{\ell}$ given by $u = (u_{\ell+1},\ldots,u_{n}) 
\mapsto -u := (-u_{n},\ldots,-u_{\ell+1})$. 
Since $u=-(-u)$ and $\sum_{i=\ell+1}^n (u_i+(-u_{i}))=0$ for all $u\in \mathbb{U}_{\ell}$, we deduce that 
\begin{equation*}
    \sum_{u\in \mathbb{U}_\ell}\left(\sum_{i=\ell+1}^n u_i\right)  = \frac{1}{2}\sum_{u\in \mathbb{U}_\ell}\left(\sum_{i=\ell+1}^n (u_i+(-u_i)) \right) =0. 
\end{equation*}
On the other hand, one can easily check that  
\begin{equation*}
    \sum_{u\in \mathbb{U}_\ell} \#\{i\mid u_i >0\} = (n-\ell)2^{n-\ell-1}\binom{n}{\ell}. 
\end{equation*}
Thus, \eqref{cal ppalgDl} equals to 
\begin{eqnarray*} 
(n+\ell+1)(n-\ell)2^{n-\ell-1}\binom{n}{\ell} - (n-\ell) 2^{n-\ell}\binom{n}{\ell} = (n+\ell-1)(n-\ell)2^{n-\ell-1}\binom{n}{\ell}. 
\end{eqnarray*}
We get the assertion for $\ell\neq \pm1$ and 
finish the proof of Proposition \ref{prop:OOppalgD}.

\subsection{Type $\mathbb{E}$}
Let $Q$ be a Dynkin quiver of type $\mathbb{E}_n$ ($n=6,7,8$) and $\Pi:=\Pi(Q)$ the preprojective algebra of $Q$. 
Our results are the following. 

\begin{lemm}\label{lem:OdimppaE}
    Let $\Pi=\Pi(Q)$ be the preprojective algebra of a Dynkin quiver $Q$ of type $\mathbb{E}$. Then, the numbers $\Dim([e_\ell \Pi]_{\rm s})$ are given by the following table. 
    \begin{equation*}    
    \begin{tabular}{rrrrrrrrrr}
    \hline
        $Q\backslash \ell$ & $1$ & $2$ & $3$ & $4$ & $5$ & $6$ & $7$ & $8$\\
        $\mathbb{E}_6$ &$216$ & $3240$ & $15120$& $792$ & $3240$ & $216$\\ 
        $\mathbb{E}_7$ &$2142$ & $66528$ & $483840$& $14112$ & $151200$ & $19656$ & $756$ \\ 
        $\mathbb{E}_8$ &$99360$ & $6289920$ & $65318400$& $1175040$ & $26611200$ &  $5080320$ & $383040$ & $6960$ \\ \hline
    \end{tabular}
\end{equation*}
\end{lemm}

\begin{proof}
We will give a way to compute the above numbers by using Theorem \ref{tau weyl} here. 
Let $Q$ be a Dynkin quiver and $\Pi=\Pi(Q)$ the preprojective algebra of $Q$. 
    We denote by $S_i$ the simple $\Pi$-module at $i\in Q_0$.
    In the rest, we fix $\ell\in Q_0$. 
    Recall that $[e_{\ell}\Pi]_{\rm s}$ is the set of isomorphism classes of all $\tau$-rigid submodules of $e_{\ell}\Pi$. 
    By Theorem \ref{tau weyl}, we find that 
    every $\tau$-rigid submodule of $e_{\ell}\Pi$ can be obtained from $e_{\ell}\Pi$ by multiplying ideals $I_{i}=\Pi(1-e_i)\Pi$ ($i\in Q_0$) repeatedly from the right. 
    Notice that, for any $M\in \mod \Pi$, $MI_i$ is the smallest amongst submodules $N$ of $M$ satisfying the condition that any composition factor of $M/N$ is isomorphic to $S_i$ (see \cite[Proposition 3.12]{IZ20} for example). 
    Furthermore, since $\Ext^{1}_{\Pi}(S_i,S_i)=0$, we have that $MI_i$ is isomorphic to the kernel of a homomorphism 
    $f := \left[
    \begin{smallmatrix}
        f_1 \\ \rotatebox{90}{$\cdots$} \\ f_d
    \end{smallmatrix}
    \right] \colon M \to S^{\oplus d}$, 
    where $f_1,\ldots,f_d$ is a basis of $\Hom_{\Pi}(M,S_i)$. 
    Using these observations, we are able to list all $\tau$-rigid submodules of $e_{\ell}\Pi$ up to isomorphism. 
    
    In a practical manner, our calculation is done in the  following way (For an implementation, we can use the {\rm GAP}  package {\rm QPA} \cite{QPA}, for instance). 
    Let $\mathcal{L}_0 := \{e_{\ell} \Pi\}$. 
    After defining $\mathcal{L}_t$ ($t\geq0$), whenever it is non-empty, let $\mathcal{L}_{t+1}$ be a list of modules obtained by the following process ($\star$).  
    
    ($\star$) Let $\mathcal{L}_{t+1}$ be an empty list. 
    For every $M \in \mathcal{L}_t$ and $i\in Q_0$, do the following. 
    \begin{itemize}
        \item Compute a basis $f_1,\ldots,f_d$ of the hom-space $\Hom_{\Pi}(M,S_i)$.
        \item Construct a homomorphism $f := \left[\begin{smallmatrix}
            f_1 \\ \rotatebox{90}{$\cdots$} \\  f_d
        \end{smallmatrix}\right] \colon M \to S_i^{\oplus d}$. 
        \item Compute the kernel $U := \Ker(f)$ of $f$. 
        \item Add $U$ to $\mathcal{L}_{t+1}$ if $U\neq 0$ and $U$ is not isomorphic to any object of $\bigcup_{j=0}^{t+1}\mathcal{L}_j$.
    \end{itemize}

    It is worth mentioning that, the last term of the above process can be replaced with the following statement: 
    \begin{itemize}
        \item Add $U$ to $\mathcal{L}_{t+1}$ if $U\neq 0$ and 
        $
        g(U)\not\in \{g(X)\mid \text{$X\in \mathcal{L}_j$, $j\in \{1,\ldots,t+1\}$}\}$,
    \end{itemize}
    where $g(X)$ denotes the $g$-vector of $X$ (see \cite[Section 5.1]{AIR14} for the definition of $g$-vectors). 
    In fact, it follows from \cite[Theorem 5.5]{AIR14} that the map $X \mapsto g(X)$ is an injection from the set of isomorphism classes of $\tau$-rigid modules. 
    
    Since $e_{\ell}\Pi$ is finite dimensional, there is the smallest integer, say $p$, such that $\mathcal{L}_p\neq \emptyset$ and  $\mathcal{L}_{p+1} = \emptyset$.
    Then, a union $\bigcup_{j=0}^p \mathcal{L}_j$ coincides with $[e_{\ell}\Pi]_{\rm s}$ from our observation stated in the first paragraph. 
    Finally, the number $\Dim([e_{\ell}\Pi]_{\rm s})$ is given by the sum of the $k$-dimensions of all modules in this list. 
\end{proof}

\begin{thm}
    Let $\Pi=\Pi(Q)$ be the preprojective algebra of a Dynkin quiver $Q$ of type $\mathbb{E}$. Then the $d$-polynomial $d(\Pi;t)$ is given by Table \ref{Table:dimppaE}. 
\begin{table}[h]\scriptsize
\renewcommand{\arraystretch}{0.95}
\begin{tabular}{r|rrrrrrrrrrrrr} 
    $n\backslash\,j$ & $0$ & $1$ & $2$ & $3$ & $4$ & $5$ & $6$ & $7$  \\ \hline
    $6$ & $22824$ & $538128$ & $3499200$ & $9072000$ & $10108800$ & $4043520$ \\ 
    $7$ & $738234$ & $27461448$ & $267083208$ & $1058400000$ & $1977091200$ & $1737469440$ & $579156480$ \\ 
    $8$ & $104964240$& $6395822880$& $90320832000$& $515410560000$& $1438746624000$& $2087401881600$& $1511903232000$& $431972352000$ \\
\end{tabular}    
\\  \ 
    \caption{Numbers $d_j$ for $d$-polynomials of preprojective algebras of type $\mathbb{E}$}
    \label{Table:dimppaE}
\end{table}
\end{thm}

\begin{proof}
Let $Q$ be a Dynkin quiver of type $\mathbb{E}_n$.
For $\ell = 1,\ldots,n$, the underlying graphs of $Q_{\I}$ are given as follows. 
\begin{eqnarray}\label{eq:EI}
&\mathbb{D}_{n-1},\ 
\mathbb{A}_1\sqcup \mathbb{A}_{n-2},\ 
\mathbb{A}_2\sqcup \mathbb{A}_1 \sqcup \mathbb{A}_{n-4},\ 
\mathbb{A}_{n-1},\ 
\mathbb{A}_4\sqcup \mathbb{A}_{n-5},\ 
\mathbb{D}_5\sqcup \mathbb{A}_{n-6}, \\ \nonumber
&\mathbb{E}_6\sqcup \mathbb{A}_{n-7} \ \text{($n=7,8$)} \quad \text{and} \quad 
\mathbb{E}_7\sqcup \mathbb{A}_{n-8} \ \text{($n=8$)}.
\end{eqnarray} 
We compute the case of $\mathbb{E}_6$. We have
    \begin{enumerate}[\rm $\bullet$]
        \item $\Eul(Q_{\bar{1}};t+1) = 
        \Eul(\mathbb{D}_{5};t+1) = 
        t^5 + 162t^4 + 1440t^3 + 4160t^2 + 4800t + 1920$.
        \item $\Eul(Q_{\bar{2}};t+1) = 
        \Eul(\mathbb{A}_{1};t+1)\Eul(\mathbb{A}_{4};t+1) = 
        t^5 + 32t^4 + 210t^3 + 540t^2 + 600t + 240$.
        \item $\Eul(Q_{\bar{3}};t+1) = \Eul(\mathbb{A}_{2};t+1)\Eul(\mathbb{A}_1;t+1)\Eul(\mathbb{A}_{2};t+1) = 
        t^5 + 14t^4 + 72t^3 + 168t^2 + 180t + 72$.
        \item $\Eul(Q_{\bar{4}};t+1) = 
        \Eul(\mathbb{A}_5;t+1) = 
        t^5 + 62t^4 + 540t^3 + 1560t^2 + 1800t + 720$. 
        \item $\Eul(Q_{\bar{5}};t+1) = 
        \Eul(\mathbb{A}_4;t+1)\Eul(\mathbb{A}_1;t+1) = 
        t^5 + 32t^4 + 210t^3 + 540t^2 + 600t + 240$.
        \item $\Eul(Q_{\bar{6}};t+1) = 
        \Eul(\mathbb{D}_5;t+1) = 
        t^5 + 162t^4 + 1440t^3 + 4160t^2 + 4800t + 1920$.
        \end{enumerate} 
Using Lemma \ref{lem:OdimppaE}, we have 
    \begin{eqnarray*}
        d(\Pi;t) &=& 216\Eul(Q_{\bar{1}};t+1) + 3240\Eul(Q_{\bar{2}};t+1) + 15120\Eul(Q_{\bar{3}};t+1) + \\ 
        &&792\Eul(Q_{\bar{4}};t+1) + 3240\Eul(Q_{\bar{5}};t+1) + 216\Eul(Q_{\bar{6}};t+1) \\ 
        &=& 22824 t^{5} + 538128 t^4 + 3499200 t^3 + 9072000 t^2 + 
        10108800 t + 4043520.  
    \end{eqnarray*}
The cases of $\mathbb{E}_7$ and $\mathbb{E}_8$ can be obtained in a similar way. 
\end{proof}
\section{$d$-polynomials of path algebras}\label{section path alg}
In this section, we study $d$-polynomials of path algebras of Dynkin type. 
Let $Q$ be a Dynkin quiver with $n$ vertices. 
Let $W=W(Q):=\langle s_i \mid i\in Q_0\rangle$ be the Coxeter group of the underlying graph of $Q$. 

We recall the definition of the $W$-Narayana numbers. 
For $w\in W$, the \emph{absolute length} $l_{\rm Ab}(w)$ of $w$ is defined by the length of the shortest word for $w$ as a product of arbitrary reflections. 
We denote by $c$ the \emph{(admissible) Coxeter element}, that is, $c=c(Q)=s_{u_1}\ldots s_{u_n}$, where $\{u_1,\ldots, u_n\}=\{1,\ldots ,n\}$ and 
$s_j$ precedes $s_i$ in any word for $c$ whenever there is an arrow from $j$ to $i$. 
Let $[{\rm id}, c] := \{w\in W\mid {\rm id}\leq_{\rm Ab} w \leq_{\rm Ab} c\}$ be the interval.  
Then, we call  
\begin{equation*}
    N(W,j) : = \#\{w\in [{\rm id},c] \mid l_{\rm Ab}(w)=j\} \quad \text{($0\leq j \leq n$)}
\end{equation*} 
the \emph{$W$-Narayana numbers}. 
In addition, the \emph{$W$-Narayana polynomial} is defined by 
\begin{equation*}
    \Cat(Q;t) = \Cat(W;t) := \sum_{j=0}^n N(W,j)t^j.
\end{equation*}
It is known that $\Cat(Q;t)$ depends only on the underlying graph of $Q$. 
Thus, we simply write the corresponding $W$-Narayana polynomials for Dynkin diagrams $\mathbb{A}_n$, $\mathbb{D}_n$ and $\mathbb{E}_n$ 
by $\Cat(\mathbb{A}_n;t)$, $\Cat(\mathbb{D}_n;t)$ and $\Cat(\mathbb{E}_n;t)$, respectively.

The following result asserts that the $h$-polynomial of the path algebra of $Q$ coincides with the $W$-Narayana polynomial.  

\begin{thm}\cite{IT09} 
\label{thm:h-Narayana}
Let $kQ$ be the path algebra of $Q$. 
Then, we have $h(kQ;t)=\Cat(Q;t)$ and $f(kQ;t) = \Cat(Q;t+1)$. 
\end{thm}

In particular, we have $\#\stautilt kQ = \Cat(Q)$, 
where $\Cat(Q)$ is known as the \emph{$W$-Catalan number} and their numbers are given by the following table (see also \cite{ONFR15} for these numbers). 
\begin{equation}\label{table:CatW}
\begin{tabular}{ccccccc}\hline
    $Q$& $\mathbb{A}_n$ & $\mathbb{D}_n$& $\mathbb{E}_6$ &$\mathbb{E}_7$ &$ \mathbb{E}_8$ \\ 
    $\Cat(Q)$ & $C_{n+1}$  & 
    $\left[\begin{smallmatrix}
    2n-1 \\ n-1
\end{smallmatrix}\right]$
    & $833$& $4160$ & $25080$ \\ \hline 
\end{tabular} 
\end{equation}
Here, $C_n:=\frac{1}{n+1}\binom{2n}{n}$ is the $n$-th Catalan number and $\left[\begin{smallmatrix}
    s \\ t
\end{smallmatrix}\right] := \frac{s+t}{s}\binom{s}{t}$.   

We refer to \cite[Section 12.3]{Petersen15} for tables of $W$-Narayana numbers for low rank. 
Similar to the $W$-Eulerian polynomials, for a given disjoint union $Q=Q_1\sqcup \cdots \sqcup Q_{m}$ of Dynkin quivers, 
we set $\Cat(Q;t):= \prod_{i=1}^m \Cat(Q_i;t)$.

\subsection{$\tau$-orbits and reduction}
Let $Q$ be a Dynkin quiver with $n$ vertices and $H:=kQ$ the path algebra of $Q$. 
We denote by $\ind H$ the set of isomorphism classes of indecomposable $H$-modules. 
In this case, every indecomposable $H$-module is $\tau$-rigid (see \cite[Chapter VII Corollary 5.14]{ASS06} for example). 
That is, $\ind H=\itaurigid H$. 
Moreover, it is decomposed into $\tau$-orbits as follows. 

\begin{defi}
We write $U\sim_{\tau} V$ for two $U,V\in \ind H$ if they lie in the same $\tau$-orbit, i.e., $U\cong \tau^sV$ for some $s\in \mathbb{Z}$. 
Then, it gives an equivalence relation on the set $\ind H$. 
\end{defi}

Let $P_{\ell}$ be the indecomposable projective $H$-module corresponding to $\ell\in Q_0$. 
Then, the set $\{P_{\ell}\mid \ell \in Q_0\}$ forms a set of complete representatives under the above relation $\sim_{\tau}$, and their equivalence classes are given by 
$$
[P_{\ell}]_{\tau} := \{\tau^{-s} P_{\ell} \mid s\geq 0\}. 
$$ 

Our result is the following. 

\begin{thm} \label{thm:OOpathalg}
Let $Q$ be a Dynkin quiver. 
Let $H:=kQ$ be the path algebra of $Q$ and $\Pi:=\Pi(Q)$ the preprojective algebra of $Q$. Then, the following statements  hold. 

\begin{enumerate}[\rm (1)]
    \item For any $\ell \in Q_0$ and any $V\cong \tau^{-s} P_{\ell}$, we have 
    \begin{equation}\nonumber
        \link(V) \cong \Delta(kQ_{\I}).    
    \end{equation}
In particular, the equivalence relation $\sim_{\tau}$ on $\ind H$ is compatible with links. 
\item We have 
\begin{equation}\nonumber \label{dj for path}
    d(H;t) = \sum_{\ell\in Q_0} \dim_k(e_{\ell}\Pi) \Cat(Q_{\I};t+1). 
\end{equation} 
In particular, the $d$-polynomial $d(H;t)$ does not depend on an orientation of $Q$. 
\item We have 
\begin{equation}\nonumber
    \Dim(\ind H) = \dim_k\Pi \quad \text{and} \quad 
    \Dim(\stautilt H) = \sum_{\ell\in Q_0} \dim_k(e_{\ell}\Pi) \Cat(Q_{\I}).  
\end{equation}
\end{enumerate}
\end{thm}

To prove this, we recall the following well-known facts. 
For the convenience of the reader, we give a proof. 

\begin{lemm} \label{dimension of pp alg}
For any $\ell \in Q_0$, we have 
\begin{equation*} \label{path algebra DimO}
    \Dim ([P_{\ell}]_{\tau}) = \dim_k(e_\ell \Pi). 
\end{equation*}
\end{lemm}

\begin{proof}
It is enough to show that 
$$(e_\ell\Pi)_{H}=\bigoplus_{s\geq0}\tau^{-s}P_\ell,$$
where $(e_\ell\Pi)_{H}$ denotes the restriction of $e_\ell\Pi$ to $H$. 
Recall that, by \cite{BGL87}, we have $\tau^-=\Ext_{H}^1(DH,-)=-\otimes_{H}\Ext_{A}^1(DH,H)$ and $$\Pi_{H}=\bigoplus_{s\geq0}\tau^{-s}H=\bigoplus_{s\geq0}\Ext_{H}^1(DH,H)^{\otimes_{H}^s}.$$
Then, we have $e_\ell\Ext_{H}^1(DH,H)=\Ext_{H}^1(DH,e_\ell H)=\tau^-(e_\ell H)$. 
Therefore we have 
$$(e_\ell\Pi)_{H}=\bigoplus_{s\geq0}e_\ell\Ext_{H}^1(DH,H)^{\otimes_{H}^s}=\bigoplus_{s\geq0}\tau^{-s}P_\ell$$ 
and get the conclusion. 
\end{proof}

\begin{lemm}\label{lem:simp tau-invariant}
    Let $U,V\in \ind H$. If $U\cong \tau^sV$ for some $s\in \mathbb{Z}$, then we have $\link(V)\cong \link(U)$.
\end{lemm}

\begin{proof}
Let $(-)^{\ast}:=\Hom_H(-,H)\colon \proj H \xrightarrow{\sim} 
\proj H^{\rm op}$ be a functor. 
For $M\in \mod H$ with a minimal projective presentation $P^{M}_1 \xrightarrow{d_1} P_{0}^M \xrightarrow{d_0} M \to 0$, 
we recall that the transpose ${\rm Tr}M \in \mod H^{\rm op}$ of $M$ is defined by the cokernel of the map $d_1^{\ast}$.

By \cite[Theorem 2.14]{AIR14}, we have a bijection 
\begin{equation}\label{eq:dagger}
(-)^{\dagger} \colon \taurigidpair H \to \taurigidpair H^{\rm op}, \quad 
(M,P) \mapsto ({\rm Tr}M \oplus P^{\ast}, M_{\rm pr}^{\ast}),
\end{equation}
where $M_{\rm pr}$ denotes the maximal projective direct summand of $M$.

We denote by $\nu:=D(-)^{\ast}$ the Nakayama functor of $\mod H$. 
For a $\tau$-rigid pair $(M,P)$, letting 
$(N,Q) := D(M,P)^{\dagger}$, we have  
\[
N = D({\rm Tr}M \oplus P^{\ast}) = \tau M \oplus \nu P 
\quad \text{and} \quad 
Q = D(M_{\rm pr}^{\ast}) = \nu M_{\rm pr} 
\]
with $Q$ injective. Then, we have $\Hom_H(\tau^- N,N)=0$ and $\Hom_{H}(N,Q) =0$. 
Indeed, $D(-)^{\dagger}$ gives a bijection between the set of isomorphism classes of $\tau$-rigid pairs for $H$ and the set of isomorphism classes of $\tau^-$-rigid pairs for $H$ (We refer to \cite[Section 2.2]{AIR14} for the precise definitions and statements about $\tau^-$-rigid modules). 
Since $H$ is hereditary, $N$ is $\tau$-rigid by
\[
\Hom_H(N,\tau N) \cong \Hom_H(\tau^- N, N) =0, 
\]
whereas $\Hom(\nu^-Q,N) \cong D\Hom_H(N,Q) =0$ holds by the property of the Nakayama functor. 
This shows that $(M,P)^{\natural} := (N,\nu^-Q) = (\tau M\oplus \nu P, M_{\rm pr})$ is a $\tau$-rigid pair for $H$. 

Thus, we have a bijective map 
\begin{equation}\label{eq:natural}
(-)^{\natural} \colon \taurigidpair H \to \taurigidpair H, 
\quad 
(M,P) \mapsto (\tau M\oplus \nu P, M_{\rm pr}).    
\end{equation}
In fact, the bijectivity of the above map is immediate from its construction and basic properties (see \cite[Chapter IV Proposition 2.10]{ASS06}) of $\tau$.
Furthermore, it restricts to a bijection 
$\taurigidpair^j H \xrightarrow{\sim} \taurigidpair^j H$ for every $j$.
Thus, \eqref{eq:natural} gives an automorphism of $\Delta(H)$.

Now, let $V$ be a non-projective indecomposable $H$-module. 
For a given basic $\tau$-rigid pair $(M,P)$, we have that $(V,0)$ is a direct summand of $(M,P)$ if and only if $(\tau V,0)$ is a direct summand of $(M,P)^{\natural}$. 
By the definition of links, this implies that the map \eqref{eq:natural} gives an isomorphism $\link(V)\cong \link(\tau V)$ between links. 
Finally, the assertion of the statement immediately follows from this result. 
\end{proof}

Now, we are ready to prove Theorem \ref{thm:OOpathalg}. 

\begin{proof}[Proof of Theorem \ref{thm:OOpathalg}]
(1) Let $\ell\in Q_0$. 
Firstly, we clearly have an isomorphism $\link(P_{\ell})\cong \Delta(kQ_{\I})$ of simplicial complexes since we have $C_{P_{\ell}}\cong kQ_{\I}$ for the algebra obtained by the reduction at $P_{\ell}$ (see Example \ref{example:reduction}). 
On the other hand, if $V\cong \tau^{-s} P_{\ell}$ for some $s\in\mathbb{Z}$, then we obtain $\link(V)\cong \link(P_{\ell})$ by applying Lemma \ref{lem:simp tau-invariant}. 
Thus, we get $\link(V)\cong \Delta(kQ_{\I})$ as desired. 

(2) and (3). We have shown in Lemma \ref{dimension of pp alg} that
$\Dim([P_{\ell}]_{\tau})=\dim_ke_{\ell}\Pi$ for all $\ell\in Q_0$. 
Then, we get the assertions by (1) and Proposition \ref{prop:orbit decomposition}. 
\end{proof}

In the rest of this section, we compute $d$-polynomials of path algebras of Dynkin types $\mathbb{A}$, $\mathbb{D}$ and $\mathbb{E}$ using Theorem \ref{thm:OOpathalg}.

\subsection{Type $\mathbb{A}$}
For Dynkin quivers of type $\mathbb{A}$, our formula is the following. 

\begin{thm}\label{type A path}
Let $H:=kQ$ be the path algebra of a Dynkin quiver $Q$ of type $\mathbb{A}_n$. Then, we have 
\begin{equation}\label{eq:d pathA}
    d(H;t) = \sum_{\ell=1}^n \ell (n-\ell+1)
    \Cat(\mathbb{A}_{\ell-1};t+1)\Cat(\mathbb{A}_{n-\ell};t+1). 
\end{equation}
In particular, we have
\begin{eqnarray*}
\Dim(\itaurigid H) = \frac{1}{6}n(n+1)(n+2)
\quad \text{and} \quad
\Dim(\stautilt H) = 4^{n+1}- 
(n+2) C_{n+2}. 
\end{eqnarray*}
For $n \leq 9$, we describe coefficients $d_j$ $(0\leq j \leq n-1)$ of the polynomial \eqref{eq:d pathA} in Table \ref{Table:dimpathA}. 

\begin{table}[h]\scriptsize
\renewcommand{\arraystretch}{0.95}
\begin{tabular}{r| r r r r r r r r r r}
    $n\backslash\,j$ & $0$ & $1$ & $2$ & $3$ & $4$ & $5$ & $6$ & $7$ & $8$\\ \hline 
    $1$ & $1$  \\ 
    $2$ & $4$& $8$ \\
    $3$ & $10$& $46$& $46$ \\ 
    $4$ & $20$& $156$& $348$& $232$  \\ 
    $5$ & $35$ & $406$ & $1499$ & $2186$ & $1093$ & \\ 
    $6$ & $56$ & $896$ & $4824$ & $11456$ & $12360$ & $4944$ \\ 
    $7$ & $84$ & $1764$ & $12888$ & $44026$ & $76458$ & $65334$ & $21778$\\ 
    $8$ & $120$ & $3192$ & $30192$ & $138340$ & $342140$ & $466500$ & $329644$ & $94184$ \\ 
    $9$ & $165$ & $5412$ & $64086$ & $376354$ & $1237622$ & $2384738$ & $2670586$ & $1607720$ & $401930$ \\
\end{tabular}    
\\ \ 
    \caption{Numbers $d_j$ for $d$-polynomials of path algebras of type $\mathbb{A}$}
    \label{Table:dimpathA}
\end{table}
\end{thm}

\begin{proof}
Let $\Pi:=\Pi(Q)$ be the preprojective algebra of $Q$. 
Let $\ell\in Q_0$. 
By Lemma \ref{dimension of pp alg}, we have  
$$
\Dim([P_\ell]_{\tau}) = \dim_k(e_{\ell} \Pi) = \ell(n-\ell+1) 
$$ 
since the quiver representation of $e_{\ell}\Pi$ is given by \eqref{proj:ppaA}. 
On the other hand, the underlying graph of $Q_{\I}$ is a disjoint union $\mathbb{A}_{\ell-1}\sqcup \mathbb{A}_{n-\ell}$. 
Thus, we have 
$$
\Cat(Q_{\I};t) = \Cat(\mathbb{A}_{\ell-1};t)\Cat(\mathbb{A}_{n-\ell};t). 
$$
Applying Theorem \ref{thm:OOpathalg}, we get the desired equation \eqref{eq:d pathA}.
For the latter assertions, we have  
\begin{eqnarray*}
\Dim(\itaurigid H) = \sum_{\ell=1}^n\Dim([P_\ell]_{\tau}) 
=
\sum_{\ell=1}^n\ell (n-\ell+1) = \frac{1}{6}n(n+1)(n+2).
\end{eqnarray*}
On the other hand, we have
\begin{equation}\label{ellCell}
    \Dim(\stautilt H) = \sum_{\ell=1}^n \Dim([P_{\ell}]_{\tau})\cdot \Cat(Q_{\I}) = \sum_{\ell=1}^n\ell C_{\ell}(n-\ell+1)C_{n-\ell+1},  
\end{equation}
where we use $\Cat(Q_{\I}) = C_{\ell}C_{n-\ell+1}$ by \eqref{table:CatW}. 
For simplicity, let $c_m:=\binom{2m}{m}=(m+1)C_m$ for $i\geq 0$. We recall that a generating function for Catalan numbers is given by (e.g. see \cite{Stanley15}) 
\begin{equation}
    \sum_{m\geq 0} C_m z^m = \frac{1-\sqrt{1-4z}}{2z}. 
\end{equation} 
Then, we can deduce that 
\begin{equation}\label{power4}
\sum_{m\geq 0} \left(\sum_{i=0}^{m-1} c_{i}c_{m-i-1} \right) z^{m-1} = \frac{1}{1-4z} \quad \text{and} \quad \sum_{i=0}^{m-1}c_{i}c_{m-i-1} = 4^{m-1}. 
\end{equation}
On the other hand, they satisfy 
\begin{eqnarray}\label{b inductive}
    c_{i}c_{n-i+1} = iC_i(n-i+1)C_{n-i+1}+(n+2)C_i C_{n-i+1}. 
\end{eqnarray}
Consequently, we obtain  
\begin{eqnarray*}
4^{n+1}&\overset{\eqref{power4}}{=}&\sum_{i=0}^{n+1}c_ic_{n-i+1}\\
&\overset{\eqref{b inductive}}{=}&2c_{n+1}+ \sum_{i=1}^n i C_i(n-i+1)C_{n-i+1}+(n+2)\sum_{j=1}^n
C_jC_{n-j+1}\\
&\overset{\eqref{ellCell}}{=}&2c_{n+1}+ \Dim(\stautilt H) + (n+2)C_{n+2}- 2c_{n+1}\\
&=&\Dim(\stautilt H)+(n+2)C_{n+2}.  
\end{eqnarray*}
It finishes the proof of Theorem \ref{type A path}.
\end{proof}

\subsection{Type $\mathbb{D}$}

For Dynkin quivers of type $\mathbb{D}$, we have the following result.

\begin{theorem}\label{type D path}
Let $H:=kQ$ be the path algebra of Dynkin quiver $Q$ of type $\mathbb{D}_n$. Then, we have
\begin{equation}\label{eq:d pathD}
\begin{array}{llll}
    d(H;t) &=& n(n-1)\Cat(\mathbb{A}_{n-1};t+1) \\ 
    &+& \displaystyle \sum_{\ell=2}^{n-1}(n-\ell)(n+\ell-1)\Cat(\mathbb{A}_{n-\ell-1};t+1)\Cat(\mathbb{D}_n;t+1).
\end{array}
\end{equation}
Here, we regard $\mathbb{D}_2 = \mathbb{A}_1\times \mathbb{A}_1$ and $\mathbb{D}_3 = \mathbb{A}_3$. 
In particular, we have 
\begin{eqnarray*}
    \Dim(\itaurigid H) &=& \frac{1}{3}n(n+1)(2n+1) \quad \text{and} \\ 
    \Dim(\stautilt H) &=& n(n-1)C_n+\sum_{\ell=2}^{n-1}(n-\ell)(n+\ell-1)C_{n-\ell}\left[\begin{smallmatrix} 2\ell-1 \\ \ell-1\end{smallmatrix}\right]. 
\end{eqnarray*}
For $n \leq 9$, we describe coefficients $d_j$ $(0\leq j \leq n-1)$ of the polynomial \eqref{eq:d pathD} in Table \ref{Table:dimpathD}. 

\begin{table}[h]\scriptsize
\renewcommand{\arraystretch}{0.95}
\begin{tabular}{r| r r r r r r r r r r}
    $n\backslash\,j$ & $0$ & $1$ & $2$ & $3$ & $4$ & $5$ & $6$ & $7$ & $8$\\ \hline 
    $4$ &  $28$ & $222$ & $498$ & $332$ \\ 
    $5$ &  $60$ & $724$ & $2716$ & $3984$ & $1992$\\ 
    $6$ &  $110$ & $1874$ & $10376$ & $24964$ & $27070$ & $10828$ \\ 
    $7$ &  $182$ & $4158$ & $31628$ & $110306$ & $193568$ & $166098$ & $55366$\\ 
    $8$ &  $280$ & $8260$ & $82308$ & $388036$ & $975060$ & $1340652$ & $950628$ & $271608$\\ 
    $9$ &  $408$ & $15096$ & $190416$ & $1159294$ & $3894538$ & $7598986$ & $8568190$ & $5173024$ & $1293256$\\
\end{tabular}   
\\  \ 
    \caption{Numbers $d_j$ for $d$-polynomials of path algebras of type $\mathbb{D}$}
    \label{Table:dimpathD}
\end{table}
\end{theorem}

\begin{proof}
Let $\Pi:=\Pi(Q)$ be the preprojective algebra of $Q$. 
For $\ell\in Q_0$, since the representation of $e_\ell\Pi$ given in \eqref{proj:ppaDpm1} and \eqref{array:D}, we have 
$$
\Dim([P_\ell]_{\tau}) 
\overset{\rm Lem.\,\ref{dimension of pp alg}}{=} 
\dim_ke_{\ell}\Pi = 
\begin{cases}
\frac{1}{2}n(n-1) & \text{if $\ell = \pm1$}, \\
(n-\ell)(n+\ell-1) & 
\text{if $\ell\neq \pm1$.}
\end{cases}
$$
On the other hand, the underlying graph of $Q_{\I}$ is 
$\mathbb{A}_{n-1}$ for $\ell=\pm1$ and $\mathbb{A}_{n-\ell-1}\sqcup \mathbb{D}_{\ell}$ for $\ell\neq \pm1$. 
Thus, we have  
$$
{\rm Cat}(Q_{\I};t) = 
\begin{cases}
{\rm Cat}(\mathbb{A}_{n-1};t) & 
\text{if $\ell = \pm1$,} \\
{\rm Cat}(\mathbb{A}_{n-\ell-1};t)
{\rm Cat}(\mathbb{D}_{\ell};t) & 
\text{if $\ell\neq \pm1$.} 
\end{cases}
$$
From these facts, we get the former assertion by Theorem \ref{thm:OOpathalg}(2). 
In addition, we get the latter assertion 
since $\Cat(Q_{\I})= C_n$ for $\ell=\pm1$ and 
$\Cat(Q_{\I}) = C_{n-\ell}\left[\begin{smallmatrix} 2\ell-1 \\ \ell-1\end{smallmatrix}\right]$ 
for $\ell\neq \pm1$ by \eqref{table:orderW}. 
\end{proof}

\subsection{Type $\mathbb{E}$}
Finally, we consider Dynkin quivers of type $\mathbb{E}$. 

\begin{thm}
Let $H:=kQ$ be the preprojective algebra of a Dynkin quiver $Q$ of type $\mathbb{E}$. 
Then, the $d$-polynomial $d(H;t)$ is given by Table \ref{Table:dimpathE}. 

\begin{table}[h]\scriptsize
\renewcommand{\arraystretch}{0.95}
\begin{tabular}{r| r r r r r r r r r r}
    $n\backslash\,j$ & $0$ & $1$ & $2$ & $3$ & $4$ & $5$ & $6$ & $7$ \\ \hline 
    $6$ &  $156$ & $2704$ & $15110$ & $36520$ & $39670$ & $15868$ \\ 
    $7$ &  $399$ & $9498$ & $73827$ & $260560$ & $460035$ & $395706$ & $131902$ \\
    $8$ &  $1240$ & $39392$ & $408048$ & $1967180$ & $5007100$ & $6931596$ & $4928756$ & $1408216$\\ 
\end{tabular}  
\\  \ 
    \caption{Numbers $d_j$ for $d$-polynomials of path algebras of type $\mathbb{E}$}
    \label{Table:dimpathE}
\end{table}
\end{thm}

\begin{proof}
Let $\Pi=\Pi(Q)$ be the preprojective algebra of $Q$. 
By a direct calculation, the numbers 
$\Dim([P_{\ell}]_{\tau}) = \dim_k(e_{\ell} \Pi)$ are given by the following table. 
\begin{equation*}
    \begin{tabular}{rrrrrrrrrr}
    \hline
        $Q\backslash \ell$ & $1$ & $2$ & $3$ & $4$ & $5$ & $6$ & $7$ & $8$\\
        $\mathbb{E}_6$ &$16$ & $30$ & $42$& $22$ & $30$ & $16$\\ 
        $\mathbb{E}_7$ &$34$ & $66$ & $96$& $49$ & $75$ & $52$ & $27$ \\ 
        $\mathbb{E}_8$ &$92$ & $182$ & $270$& $136$ & $220$ &  $168$ & $114$ & $58$ \\ \hline
    \end{tabular}
\end{equation*}

We compute the case $\mathbb{E}_6$. 
Recall that the underlying graph of $Q_{\I}$ is given by \eqref{eq:EI}.
Then, we have 
\begin{enumerate}[\rm $\bullet$]
    \item $\Cat(Q_{\bar{1}};t+1) = \Cat(\mathbb{D}_{5};t+1) =  t^5 + 25t^4 + 160t^3 + 410t^2 + 455t + 182$.
    \item $\Cat(Q_{\bar{2}};t+1) = \Cat(\mathbb{A}_{1};t+1)\Cat(\mathbb{A}_{4};t+1) = 
    t^5 + 16t^4+ 84t^3+ 196t^2+ 210t+ 84$.
    \item $\Cat(Q_{\bar{3}};t+1) = \Cat(\mathbb{A}_{2};t+1)\Cat(\mathbb{A}_1;t+1)\Cat(\mathbb{A}_{2};t+1) = 
    t^5 + 12t^4 + 55t^3 + 120t^2 + 125t + 50$.
    \item $\Cat(Q_{\bar{4}};t+1) = \Cat(\mathbb{A}_5;t+1) = t^5 + 20t^4 + 120t^3 + 300t^2 + 330t + 132$. 
    \item $\Cat(Q_{\bar{5}};t+1) = \Cat(\mathbb{A}_4;t+1)\Cat(\mathbb{A}_1;t+1) = 
    t^5 + 16t^4 + 84t^3 + 196t^2 + 210t + 84$.
    \item $\Cat(Q_{\bar{6}};t+1) = \Cat(\mathbb{D}_5;t+1) = 
    t^5 + 25t^4 + 160t^3 + 410t^2 + 455t + 182$.
\end{enumerate}
Using Proposition \ref{thm:OOpathalg}, we obtain the desired equation by 
\begin{eqnarray*}
    d(H;x) &=& 
    16\Cat(Q_{\bar{1}};t+1) + 
    30\Cat(Q_{\bar{2}};t+1) + 
    42\Cat(Q_{\bar{3}};t+1) + \\ 
    &&22\Cat(Q_{\bar{4}};t+1) + 
    30\Cat(Q_{\bar{5}};t+1) +
    16\Cat(Q_{\bar{6}};t+1)\\ 
    &=& 156t^5 + 2704t^4 + 15110t^3 + 36520t^2 + 39670t + 15868. 
\end{eqnarray*}
The cases for $\mathbb{E}_7$ and $\mathbb{E}_8$ can be obtained in a similar way. 
\end{proof}

\section{Generating functions of $d$-polynomials} \label{sec:generating_function}
In combinatorics, the method of generating functions provides a fundamental tool to solve enumeration problems of topics of interest (we refer to \cite{Stanley99} for example). 
We end this paper with discussing the generating function defined by a family of $h$-polynomials and $d$-polynomials.

We first prepare the following basic terminology. 
For a (possibly infinite) sequence $(a_0,a_1,a_2, \ldots)$, we recall that the \emph{ordinary generating function} is the series 
\begin{equation*}
\sum_{n\geq 0} a_n z^n = a_0 + a_1z + a_2z^{2} +a_3z^3 +  \cdots, 
\end{equation*}
while the \emph{exponential generating function} is the series 
\begin{equation*}
\sum_{n\geq 0} a_n \frac{z^n}{n!} = a_0 + a_1z + a_2\frac{z^{2}}{2} + a_3\frac{z^{3}}{6} + \cdots. 
\end{equation*}
We have $\frac{1}{1-z} = 1+z+z^2+z^3+\cdots$ and $e^z = 1+ z + \frac{z^2}{2} + \frac{z^3}{6} + \cdots$, which are obtained by letting $a_i=1$ for all $i$, respectively. 

Now, we recall the following famous result, which is originally due to Euler \cite{Euler}. 
Throughout, let $\mathbb{A}_{-1}:=\mathbb{A}_0$. 
We define the exponential generating function for Eulerian polynomials by 
\begin{equation}\label{gen Euler_poly}
 S(t,z) :=  
    \sum_{n\geq 0} \Eul(\mathbb{A}_{n-1};t) \frac{z^{n}}{n!}.  
\end{equation}
It is known (see \cite[Section 1.6]{Petersen15} for example) that it satisfies the equation 
\begin{equation*}
    \frac{\partial}{\partial z} S(t,z) = tS(t,z)^2 + (1-t)S(t,z). 
\end{equation*}
Solving the differential equation $f'(z)=tf(z)^2 + (1-t)f(z)$ with $f(0)=1$, we obtain the following closed-form expression. 
 
\begin{prop}[Euler]  \label{prop:Euler}
    \begin{equation}\label{expression Euler}
       S(t,z) = \frac{t-1}{t-e^{z(t-1)}}. 
    \end{equation} 
\end{prop}

As we mentioned in Theorem \ref{thm:h of ppalg}, Eulerian polynomials coincide with $h$-polynomials of preprojective algebras of type $\mathbb{A}$. 
Inspired by this fact, we will introduce four kinds of generating functions associated with a family of algebras as follows.

\begin{defi}
Let $A_{\bullet} := (A_n)_{n\geq0}$ be a (possibly infinite) sequence of finite dimensional algebras such that $A_n$ is $g$-finite for each $n$.
We define generating functions
$$
\mathcal{H}(A_{\bullet};t,z):= \sum_{n\geq 0} h(A_n;t) z^{n} \quad \text{and} \quad 
\mathscr{H}(A_{\bullet};t,z):= \sum_{n\geq 0} h(A_n;t) \frac{z^{n}}{n!}, 
$$
which we call the \emph{ordinary and exponential generating functions of $h$-polynomials of $A_{\bullet}$}, respectively. 
Similarly, we call   
$$
\mathcal{D}(A_{\bullet};t,z):= \sum_{n\geq 0} d(A_n;t) z^{n} \quad \text{and} \quad 
\mathscr{D}(A_{\bullet};t,z):= \sum_{n\geq 0} d(A_n;t) \frac{z^{n}}{n!}
$$
the \emph{ordinary and exponential generating functions of $d$-polynomials of $A_{\bullet}$}, respectively. 
\end{defi}

In the above, we notice that the generating functions of $f$-polynomials can be obtained from those of $h$-polynomials as  
\begin{equation*}
    \mathcal{H}(A_{\bullet};t-1,z) = \sum_{n\geq 0} f(A_n;t)z^n \quad \text{and} \quad 
    \mathscr{H}(A_{\bullet};t-1,z) = \sum_{n\geq 0} f(A_n;t)\frac{z^n}{n!}
\end{equation*}
respectively. 

In our context, the above Proposition \ref{prop:Euler} can be formulated as follows: 
Let $\Pi(\mathbb{A}_{n})$ be the preprojective algebras of type $\mathbb{A}_{n}$ for all $n$, and let  
$\Pi(\mathbb{A}_{\bullet}):=(\Pi(\mathbb{A}_{n-1}))_{n\geq 0}$. 
Then, we have 
\begin{eqnarray*}
    \mathscr{H}(\Pi(\mathbb{A}_{\bullet}),t,z) &:=& \sum_{n\geq 0} h(\Pi(\mathbb{A}_{n-1};t)) \frac{z^{n}}{n!} \\ 
    &=& \sum_{n\geq 0} \Eul(\mathbb{A}_{n-1};t)) \frac{z^{n}}{n!} \\ 
    &=& S(t,z). 
\end{eqnarray*}
Thus, Proposition \ref{prop:Euler} gives 

\begin{prop}\label{prop:Euler-ppaA}
$$\mathscr{H}(\Pi(\mathbb{A}_{\bullet});t,z)=\frac{t-1}{t-e^{z(t-1)}}.$$    
\end{prop}

It is natural to consider the analogue for $d$-polynomials. 
Using their description in Theorem 4.11, we can also get a closed-form expression for the
exponential generating function of $d$-polynomials of $\Pi(\mathbb{A}_{\bullet})$. 

Our result is the following.  

\begin{theorem}\label{thm:gen ppaA}
We have 
$$
\mathscr{D}(\Pi(\mathbb{A}_{\bullet});t,z) = \frac{z^2t^4 e^{2z(t+2)}}{2(e^{z}(t+1) - e^{z(t+1)})^4}. 
$$

\end{theorem}

\begin{proof}
We will prove the equality 
    \begin{equation*}
        \mathscr{D}(\Pi(\mathbb{A}_{\bullet});t-1,z) = \frac{z^2}{2}\left(\frac{\partial}{\partial z} S(t,z)\right)^2.
    \end{equation*}
After that, we can easily get the desired equation from \eqref{expression Euler} by replacing $t\to t+1$.  

We recall from Theorem \ref{thm:dpoly ppaA} that 
the $d$-polynomial of $\Pi(\mathbb{A}_n)$ is given by 
\begin{eqnarray*}
d(\Pi(\mathbb{A}_{n});t) = 
\sum_{\ell=1}^n 
\frac{(n+1)!}{2\cdot (\ell-1)!(n-\ell)!}
\Eul(\mathbb{A}_{\ell-1};t+1) \Eul(\mathbb{A}_{n-\ell};t+1). 
\end{eqnarray*}
Applying this result, we get 
\begin{eqnarray*}    
\mathscr{D}(\Pi(\mathbb{A}_{\bullet});t-1,z) &=& \sum_{n\geq 0} d(\Pi(\mathbb{A}_{n-1});t-1)\frac{z^n}{n!} \\ 
&=& 
\sum_{n\geq 0} \left(\sum_{\ell=1}^{n-1} 
\frac{n!}{2\cdot (\ell-1)!(n-\ell-1)!}
\Eul(\mathbb{A}_{\ell-1};t) \Eul(\mathbb{A}_{n-\ell-1};t)  \right) \frac{z^n}{n!} \\ 
&=& \frac{1}{2} \sum_{n\geq 0} \left(\sum_{\ell=1}^{n-1} \Eul(\mathbb{A}_{\ell-1};t)\frac{z^{\ell-1}}{(\ell-1)!} 
\Eul(\mathbb{A}_{n-\ell-1};t)\frac{z^{n-\ell-1}}{(n-\ell-1)!}
\right) \\ 
&=& \frac{z^2}{2} 
\cdot \frac{\partial}{\partial z} 
\left(
\sum_{i\geq 0} \Eul(\mathbb{A}_{i-1};t)\frac{z^{i}}{i!} 
\right)
\cdot \frac{\partial}{\partial z} \left(
\sum_{j\geq 0} \Eul(\mathbb{A}_{j-1};t)\frac{z^{j}}{j!} 
\right) \\ 
&=& \frac{z^2}{2} 
\left( \frac{\partial}{\partial z} S(t,z)
\right)^2
\end{eqnarray*}
as desired. It finishes the proof. 
\end{proof}

Next, we study generating functions for the path algebras of type $\mathbb{A}$ and give their explicit formula. 
Contrary to the case of Eulerian polynomials, 
we consider the ordinary generating function of Narayana polynomials: 
\begin{equation}\label{gen Narayana}
    C(t,z) := 
    \sum_{n\geq 0} \Cat(\mathbb{A}_{n-1};t) z^{n}.   
\end{equation}
Then, the equation
\begin{equation*}
    tzC(t,z)^2-(1+z(t-1))C(t,z) + 1 = 0 
\end{equation*}
gives the following result. 

\begin{prop}{\rm (see \cite[Section 2.3]{Petersen15} for example)} 
    \begin{equation}\label{expression Narayana}
        C(t,z) = \frac{1 - z (t - 1) - \sqrt{1 - 2 z (t + 1) + z^2 (t - 1)^2}}{2 t z}. 
    \end{equation}
\end{prop}

We apply this result to compute generating functions for path algebras of type $\mathbb{A}$. 
Let $Q_n$ be Dynkin quivers of type $\mathbb{A}_n$ for all $n$ and $kQ_{\bullet}:=(kQ_{n-1})_{n\geq 0}$. 
Then, we have 
\begin{eqnarray*}
    \mathcal{H}(kQ_{\bullet};t,z) &:=& \sum_{n\geq 0} h(kQ_{n-1};t) z^{n} \\ 
    &=& \sum_{n\geq 0} \Cat(\mathbb{A}_n;t) z^{n} \\ 
    &=& C(t,z)
\end{eqnarray*}
and 
\begin{eqnarray*}
    \mathcal{D}(kQ_{\bullet};t,z) &=& 
    \sum_{n\geq 0} d(kQ_{n-1};t) z^{n}. 
\end{eqnarray*}

Our results are the following. 

\begin{thm}\label{thm:gen pathA}
We have 
\begin{eqnarray*}
\mathcal{H}(kQ_{\bullet};t,z) &=& \frac{1 - z (t-1) - f(t,z)}{2 t z} \quad \text{and} \\ 
\mathcal{D}(kQ_{\bullet};t,z) &=& \frac{(z(t+2)-z^2t^2+f(t+1,z)-f(t+1,z)^2)^2}{4(t+1)^2z^2f(t+1,z)^2}, 
\end{eqnarray*}
where $f(t,z) := \sqrt{1 - 2 z (t + 1) + z^2 (t - 1)^2}$. 
\end{thm}

\begin{proof}
The first equation is \eqref{expression Narayana}. 
For the second one, our proof is similar to the proof of Theorem \ref{thm:gen ppaA}. 
We will show 
\begin{equation*}
        \mathcal{D}(kQ_{\bullet};t-1,z) = z^2\left(\frac{\partial}{\partial z} C(t,z)\right)^2.  
\end{equation*}
Then, we obtain the desired equality from \eqref{expression Narayana} by replacing $t\to t+1$. 

For a Dynkin quiver $Q$ of type $\mathbb{A}_n$, by Theorem \ref{type A path}, we have  
\begin{equation*}
    d(kQ;t) = \sum_{\ell=1}^n \ell (n-\ell+1)
    \Cat(\mathbb{A}_{\ell-1};t+1)\Cat(\mathbb{A}_{n-\ell};t+1). 
\end{equation*}
Then, it gives  
\begin{eqnarray*}
    \mathcal{D}(kQ_{\bullet};t-1,z) &=& \sum_{n\geq 0} d(kQ_{n-1};t-1) z^n \\ 
    &=&
    \sum_{n\geq 0} \left(
    \sum_{\ell=1}^{n-1} \ell (n-\ell)
    \Cat(\mathbb{A}_{\ell-1};t)\Cat(\mathbb{A}_{n-\ell-1};t)
    \right) z^n \\ 
    &=&
    \sum_{n\geq 0} \left(
    \sum_{\ell=1}^{n-1} \ell \Cat(\mathbb{A}_{\ell-1};t)z^{\ell-1} 
    (n-\ell)\Cat(\mathbb{A}_{n-\ell-1};t)z^{n-\ell-1}
    \right)\\ 
    &=&
    z^2 \cdot \frac{\partial}{\partial z} \left( 
    \sum_{i\geq 0}
    \Cat(\mathbb{A}_{i-1};t)z^{i} 
    \right)\cdot 
    \frac{\partial}{\partial z} \left( 
    \sum_{j\geq 0}
    \Cat(\mathbb{A}_{j-1};t)z^{j} 
    \right) \\ 
    &=&
    z^2\left( \frac{\partial}{\partial z} C(t,z) \right)^2. 
\end{eqnarray*}
We finish the proof. 
\end{proof}

\section{Questions}\label{question}

At the end of the paper, 
we pose the following questions related to the results in Section \ref{sec:generating_function}.

\begin{question}
 Can we extend the results of Theorem \ref{thm:gen ppaA} and Theorem \ref{thm:gen pathA} to other types? $($The closed form of the exponential generating function for Eulerian polynomials of type $\mathbb{D}$ is well-known \cite{Brenti94,Chow03}$)$. 
\end{question}

\begin{question}
\begin{enumerate}[\rm (1)]
\item Is there a family of algebras that provides a well-known class of generating functions by polynomials? 
More generally, what kinds of generating functions can we realize as $f$-, $h$- and $d$-polynomials of algebras? 

\item What types of properties of algebras are captured by generating functions?
For example, the closed form of the exponential generating function of d-polynomials of preprojective algebras is given in Theorem \ref{thm:gen ppaA}, while the ordinary generating functions of one of the path algebras is given in Theorem \ref{thm:gen pathA}. What properties of algebras determine this difference? 
\end{enumerate}
\end{question}

\section*{Acknowledgments} 
T.A is supported by JSPS Grant-in-Aid for Transformative Research Areas (A) 22H05105. 
Y.M is supported by Grant-in-Aid for Scientific Research 20K03539. 
The project began with discussions at the workshop \emph{Unsolved problems in tilting theory and the use of computer aids} in March 2023. We are very grateful to the organizers and participants. 
We thank the referee for their kind comments, which greatly improve the paper.

\bibliographystyle{alpha} 

\begin{thebibliography}{ONFR15}

\bibitem[AHI{\etalchar{+}}22]{AHIKM22}
Toshitaka Aoki, Akihiro Higashitani, Osamu Iyama, Ryoichi Kase, and Yuya
  Mizuno.
\newblock Fans and polytopes in tilting theory {I}: Foundations.
\newblock {\em arXiv:2203.15213}, 2022.

\bibitem[AHI{\etalchar{+}}23]{AHIKM23}
Toshitaka Aoki, Akihiro Higashitani, Osamu Iyama, Ryoichi Kase, and Yuya
  Mizuno.
\newblock Fans and polytopes in tilting theory {II}: $g$-fans of rank 2.
\newblock {\em arXiv:2301.01498}, 2023.

\bibitem[AIR14]{AIR14}
Takahide Adachi, Osamu Iyama, and Idun Reiten.
\newblock {$\tau$}-tilting theory.
\newblock {\em Compos. Math.}, 150(3):415--452, 2014.

\bibitem[AMN20]{AMN20}
Hideto Asashiba, Yuya Mizuno, and Ken Nakashima.
\newblock Simplicial complexes and tilting theory for {B}rauer tree algebras.
\newblock {\em J. Algebra}, 551:119--153, 2020.

\bibitem[ASS06]{ASS06}
Ibrahim Assem, Daniel Simson, and Andrzej Skowro\'{n}ski.
\newblock {\em Elements of the representation theory of associative algebras.
{V}ol. 1}, volume~65 of {\em London Mathematical Society Student Texts}.
\newblock Cambridge University Press, Cambridge, 2006.
\newblock Techniques of representation theory.


\bibitem[BB05]{BB05}
Anders Bj\"{o}rner and Francesco Brenti.
\newblock {\em Combinatorics of {C}oxeter groups}, volume 231 of {\em Graduate
  Texts in Mathematics}.
\newblock Springer, New York, 2005.

\bibitem[BGL87]{BGL87}
Dagmar Baer, Werner Geigle, and Helmut Lenzing.
\newblock The preprojective algebra of a tame hereditary {A}rtin algebra.
\newblock {\em Comm. Algebra}, 15(1-2):425--457, 1987.

\bibitem[Bre94]{Brenti94}
Francesco Brenti.
\newblock {$q$}-{E}ulerian polynomials arising from {C}oxeter groups.
\newblock {\em European J. Combin.}, 15(5):417--441, 1994.

\bibitem[Cho03]{Chow03}
Chak-On Chow.
\newblock On the {E}ulerian polynomials of type {$D$}.
\newblock {\em European J. Combin.}, 24(4):391--408, 2003.

\bibitem[DIJ19]{DIJ19}
Laurent Demonet, Osamu Iyama, and Gustavo Jasso.
\newblock {$\tau$}-tilting finite algebras, bricks, and {$g$}-vectors.
\newblock {\em Int. Math. Res. Not. IMRN}, (3):852--892, 2019.

\bibitem[DIR{\etalchar{+}}23]{DIRRT23}
Laurent Demonet, Osamu Iyama, Nathan Reading, Idun Reiten, and Hugh Thomas.
\newblock Lattice theory of torsion classes: beyond {$\tau$}-tilting theory.
\newblock {\em Trans. Amer. Math. Soc. Ser. B}, 10:542--612, 2023.

\bibitem[EJR18]{EJR18}
Florian Eisele, Geoffrey Janssens, and Theo Raedschelders.
\newblock A reduction theorem for {$\tau$}-rigid modules.
\newblock {\em Math. Z.}, 290(3-4):1377--1413, 2018.

\bibitem[Eul87]{Euler}
Leonhard Euler.
\newblock {\em Institutiones calculi differentialis cum eius usu in analysi
  finitorum ac doctrina serierum}.
\newblock In typographeo Petri Galeatii, Ticini, 1787.

\bibitem[Hil15]{Hille15}
Lutz Hille.
\newblock Tilting modules over the path algebra of type {$\Bbb A$}, polytopes,
  and {C}atalan numbers.
\newblock In {\em Lie algebras and related topics}, volume 652 of {\em Contemp.
  Math.}, pages 91--101. Amer. Math. Soc., Providence, RI, 2015.

\bibitem[IRRT18]{IRRT18}
Osamu Iyama, Nathan Reading, Idun Reiten, and Hugh Thomas.
\newblock Lattice structure of {W}eyl groups via representation theory of
  preprojective algebras.
\newblock {\em Compos. Math.}, 154(6):1269--1305, 2018.

\bibitem[IT09]{IT09}
Colin Ingalls and Hugh Thomas.
\newblock Noncrossing partitions and representations of quivers.
\newblock {\em Compos. Math.}, 145(6):1533--1562, 2009.

\bibitem[IZ20]{IZ20}
Osamu Iyama and Xiaojin Zhang.
\newblock Classifying {$\tau$}-tilting modules over the {A}uslander algebra of {$K[x]/(x^n)$}.
\newblock {\em J. Math. Soc. Japan}, 72(3):731--764, 2020.

\bibitem[Jas15]{Jasso15}
Gustavo Jasso.
\newblock Reduction of {$\tau$}-tilting modules and torsion pairs.
\newblock {\em Int. Math. Res. Not. IMRN}, (16):7190--7237, 2015.

\bibitem[Miz14]{Mizuno14}
Yuya Mizuno.
\newblock Classifying {$\tau$}-tilting modules over preprojective algebras of
  {D}ynkin type.
\newblock {\em Math. Z.}, 277(3-4):665--690, 2014.

\bibitem[ONFR15]{ONFR15}
Mustafa A.~A. Obaid, S.~Khalid Nauman, Wafaa~M. Fakieh, and Claus~Michael
  Ringel.
\newblock The number of support-tilting modules for a {D}ynkin algebra.
\newblock {\em J. Integer Seq.}, 18(10):Article 15.10.6, 24, 2015.

\bibitem[Pet15]{Petersen15}
T.~Kyle Petersen.
\newblock {\em Eulerian numbers}.
\newblock Birkh\"{a}user Advanced Texts: Basler Lehrb\"{u}cher. [Birkh\"{a}user
  Advanced Texts: Basel Textbooks]. Birkh\"{a}user/Springer, New York, 2015.
\newblock With a foreword by Richard Stanley.

\bibitem[QPA]{QPA}
The {QPA}-team. {QPA} -- {Q}uivers and {P}ath {A}lgebras -- a {GAP} package,
  version 1.34 (https://folk.ntnu.no/oyvinso/qpa/).

\bibitem[Slo]{OEIS}
N.~J.~A. Sloane.
\newblock {O}nline {E}ncyclopedia of {I}nteger {S}equences (http://oeis.org).

\bibitem[Sta99]{Stanley99}
Richard~P. Stanley.
\newblock {\em Enumerative combinatorics. {V}ol. 2}, volume~62 of {\em
  Cambridge Studies in Advanced Mathematics}.
\newblock Cambridge University Press, Cambridge, 1999.
\newblock With a foreword by Gian-Carlo Rota and appendix 1 by Sergey Fomin.

\bibitem[Sta15]{Stanley15}
Richard~P. Stanley.
\newblock {\em Catalan numbers}.
\newblock Cambridge University Press, New York, 2015.

\end{thebibliography}

\newcommand{\etalchar}[1]{$^{#1}$}

\end{document}